\documentclass[11pt, twosideside]{amsart}
\usepackage{amsmath}
\usepackage{amsthm}
\usepackage{amssymb}
\usepackage{amsmath}
\usepackage{amsthm}
\usepackage{amssymb}
\usepackage{fancyhdr}
\usepackage{enumerate}
\usepackage{amscd}
\usepackage[all]{xy}
\usepackage{graphicx}

\usepackage[tmargin=1in,bmargin=1in,lmargin=1in,rmargin=1in]{geometry}
\theoremstyle{plain}
\newtheorem{lemma}{Lemma}[section]
\newtheorem{theorem}[lemma]{Theorem}
\newtheorem{proposition}[lemma]{Proposition}

\newtheorem{corollary}[lemma]{Corollary}

\numberwithin{equation}{section}

\theoremstyle{definition}
\newtheorem{remark}[lemma]{Remark}
\newtheorem{definition}[lemma]{Definition}
\newtheorem{example}[lemma]{Example}
\newtheorem{notation}[lemma]{Notation}

\begin{document}
\title[Localized deformation with DEC]{Localized deformation for initial data sets with the dominant energy condition}
\author{Justin Corvino}
\address{Department of Mathematics, Lafayette College, Easton, PA 18042, USA}
\email{corvinoj@lafayette.edu}
\author{Lan-Hsuan Huang}
\address{Department of Mathematics, University of Connecticut, Storrs, CT 06269, USA }
\email{lan-hsuan.huang@uconn.edu}
\thanks{The first named author was partially supported by the NSF through grant DMS~1207844. The second named author was partially supported by the NSF through DMS~1308837 and DMS~1452477. This material is also based upon work supported by the NSF under Grant No.~0932078~000, while both authors were in residence at the Mathematical Sciences Research Institute in Berkeley, California, during the Fall 2013 program in Mathematical General Relativity.}

\begin{abstract} We consider localized deformation for initial data sets of the Einstein field equations with the dominant energy condition. Deformation results with the weak inequality need to be handled delicately. We introduce a modified constraint operator to absorb the first order change of the metric in the dominant energy condition. By establishing the local surjectivity theorem, we can promote the dominant energy condition to the strict inequality by compactly supported variations and obtain new gluing results with the dominant energy condition. The proof of local surjectivity is a modification of the earlier work for the usual constraint map by the first named author and R.~Schoen~\cite{Corvino-Schoen:2006} and by P.~Chru\'{s}ciel and E.~Delay~\cite{Chrusciel-Delay:2003}, with some refined analysis.   
\end{abstract}

\maketitle

\section{Introduction}
Deformations to obtain the strict dominant energy condition are important analytical tools in the study of initial data sets. Among various applications, the most prominent one is perhaps  the proof of the Positive Mass Theorem by  R.~Schoen and S.-T.~Yau, in which they use the strict dominant energy condition in conjunction with the stability of a minimal hypersurface (or more generally a marginally outer trapped hypersurface) to study the geometry and topology of the manifold. Their deformation results for asymptotically flat manifolds are \emph{global} because their argument involves a conformal change of the metric, and the resulting variations of the initial data set, which satisfy an elliptic equation, cannot have compact supports, see ~\cite{Schoen-Yau:1979-pmt1} for the scalar curvature operator and~\cite{Schoen-Yau:1981-pmt2} for the dominant energy condition with nonzero current density~$J$. A general global deformation result for asymptotically flat initial data sets is obtained by the second named author with M.~Eichmair, D.~Lee, and R. Schoen as a central analytical step in the proof of the spacetime Positive Mass Theorem~\cite{Eichmair-Huang-Lee-Schoen:2016}.

In contrast, if one restricts to compactly supported variations of initial data sets, so-called \emph{localized} deformations, there is an obstruction to deform to the strict dominant energy condition. For vacuum initial data sets, the obstruction is related to whether the corresponding spacetime has Killing vector fields. More specifically, work of A.~Fischer and J.~Marsden \cite{Fischer-Marsden:1973, Fischer-Marsden:1975} shows that the constraint map is locally surjective if the kernel of the formal $L^2$ adjoint of the linearized constraint operator is trivial, a condition which V.~Moncrief~\cite{Moncrief:1975} proves is equivalent to the absence of spacetime Killing vector fields.  Fischer-Marsden's proof uses the elliptic splitting of the function spaces on a closed manifold in an essential way.  For compact manifolds with boundary,  the first named author uses a variational approach to prove a local surjectivity result for the scalar curvature operator~\cite{Corvino:2000}, and then with R.~Schoen for the full constraint map~\cite{Corvino-Schoen:2006}.  P.~Chru\'{s}ciel and E.~Delay introduce finer weighted spaces and derive a systematic approach to localized deformations for the constraint map in various settings~\cite{Chrusciel-Delay:2003}. Localized deformations play an important role in gluing constructions (e.g. \cite{Corvino:2000, Corvino-Schoen:2006, Chrusciel-Delay:2003, Chrusciel-Corvino-Isenberg:2011, Isenberg-Mazzeo-Pollack:2003, Chrusciel-Isenberg-Pollack:2005}) and have applications to  rigidity type results (e.g. \cite[Theorem 8]{Corvino:2000}), but all earlier results focus on initial data sets that are either \emph{vacuum} or have the \emph{strict} dominant energy condition.  In this paper, we present localized deformations without assuming either condition, by introducing a new modified constraint operator, and obtain new gluing applications. The rigidity type application is developed in \cite{Huang-Lee-rigid}. 

For non-vacuum initial data sets, there is a serious technical detail in deforming to the strict inequality from a weak inequality.  Essentially, the deformation does not seem to follow directly from local surjectivity of the constraint map, as we now discuss. (Please refer Section~\ref{sec:prelim} for the relevant definitions.)  Suppose the constraint map $\Phi$ is locally surjective at an initial data set $(g, \pi)$. Given $(\psi, V)$ (sufficiently small), suppose one solves for a small deformation $(h, w)$ to achieve $\Phi(g+h, \pi+w) = \Phi(g, \pi) + (2\psi, V)$. That is, the mass and current densities of the deformed initial data set $(\bar{g}, \bar{\pi}) = (g+h, \pi+w)$ are $\bar{\mu} = \mu+\psi$ and $\bar{J} = J +V$. 
The norm of $\bar{J}$ is taken with respect to the deformed metric, and  
\[
|\bar{J}|_{\bar{g}}=|J+V|_{g+h} \le |J+V|_g + \tfrac12 |h|_g |J+ V|_g + O(|h|_g^2).
\] 
Thus for the deformed data $(\bar{g}, \bar{\pi})$ 
\begin{align}\label{equation:dec}
	\bar{\mu} - |\bar{J}|_{\bar{g}} \ge \mu - |J+V|_g +\psi -  \tfrac12 |h|_g |J+ V|_g  + O(|h|_g^2).
\end{align}
Note that $h$ depends on the choice of $(\psi, V)$ and the estimates do \emph{not} indicate that $\psi$ can dominate the first order change  involving $h$ to promote the dominant energy condition (while we can arrange for $\psi$ to dominate the term $O(|h|_g^2)$).  
  
In this work, we introduce a \emph{modified constraint map}.  Given a vector field $V$, a metric $g$ and a symmetric $(0,2)$-tensor $h$, let $h \cdot_g V$ denote the vector field dual (with respect to $g$) to the tensor contraction of $h$ and $V$.  For a fixed initial data set $(g, \pi)$ and a vector field $W$, let $\Phi^W_{(g,\pi)}$ be defined by 
\begin{align*}
 \Phi^W_{(g, \pi)}(\gamma, \tau)= \Phi(\gamma, \tau)+  \left(0, \tfrac{1}{2} \gamma \cdot_g \left(J +W\right)\right),
\end{align*}
where $J= \textup{div}_g \pi$ is the current density of $(g,\pi)$.  When $W=V$, the additional term is designed to absorb the first order change that results in the term $\tfrac12 |h|_g |J+ V|_g$ from \eqref{equation:dec} and is motivated by the linear map introduced in~\cite{Eichmair-Huang-Lee-Schoen:2016}. We establish a sufficient condition, in terms of the modified operator $\Phi^0_{(g, \pi)}$ (setting~$W=0$), to promote to the dominant energy condition. Throughout this paper, we let $\overline{\Omega}$ be a compact connected smooth manifold-with-boundary, with manifold interior $\Omega$, unless otherwise indicated.

\begin{theorem} \label{theorem:main}
 Let $( g, \pi) \in C^{4,\alpha}(\overline{\Omega})\times C^{3,\alpha}(\overline{\Omega})$ be an initial data set. Suppose that the kernel of $D\Phi^0_{(g, \pi)}|_{(g,\pi)}^*$ is trivial on $\Omega$. Then there is a $C^{2,\alpha}(\overline{\Omega})$ neighborhood $\mathcal W$ of the zero vector and  constants $\epsilon>0$, $C>0$ such that for $(\psi, V) \in \mathcal{B}_0\times (\mathcal B_1 \cap \mathcal W)$  with $\| (\psi, V)\|_{\mathcal{B}_0\times \mathcal B_1} \le \epsilon$,  there exists $(h, w)\in \mathcal{B}_2 \times \mathcal{B}_2$ with $\| (h, w) \|_{\mathcal{B}_2 \times \mathcal{B}_2} \le C\| (\psi,V)\|_{\mathcal B_0\times \mathcal B_1}$ such that $(\bar{g}, \bar{\pi}) = (g+ h, \pi + w)\in C^{2,\alpha}(\overline\Omega)\times C^{2,\alpha}(\overline{\Omega})$ is an initial data set and satisfies 
\[	
	\bar{\mu} - |\bar{J}|_{\bar{g}} \ge \mu - |J+V|_g +  \psi.
\]
\end{theorem} 
The weighted Banach spaces  $\mathcal{B}_k = \mathcal{B}_k (\Omega) \subset C^{k,\alpha}_{\mathrm{loc}}(\Omega)$ for $k=0, 1, 2$ and the respective norms are defined in Section~\ref{subsection:weighted-Holder}. In particular, if $(g,\pi)$ satisfies the dominant energy condition, the above theorem gives a sufficient condition to deform to  the strict dominant energy condition in $\Omega$ by setting $V=0$ and $\psi >0$ in $\Omega$. 

Our proofs also give the following version of Theorem~\ref{theorem:main} that includes higher order regularity and uniformity in the neighborhood of an arbitrary initial data set.

\begin{theorem}  \label{theorem:main2}
Let $k\ge 0$.  Let $( g_0, \pi_0) \in C^{k+4,\alpha}(\overline{\Omega})\times C^{k+3,\alpha}(\overline{\Omega})$ be an initial data set. Suppose that the kernel of $D\Phi^0_{(g_0, \pi_0)}|_{(g_0,\pi_0)}^*$ is trivial on $\Omega$. Then there is a $C^{k+4,\alpha}(\overline{\Omega})\times C^{k+3,\alpha}(\overline{\Omega})$ neighborhood  $\mathcal{U}$ of $(g_0, \pi_0)$, a $C^{k+2,\alpha}(\overline{\Omega})$ neighborhood $\mathcal{W}$ of the zero vector, and constants $\epsilon>0$, $C>0$ such that for $(g, \pi) \in \mathcal{U}$ and for $(\psi,V) \in C^{k,\alpha}_c(\Omega) \times (C^{k+1,\alpha}_c(\Omega) \cap \mathcal W)$  with $\| (\psi, V)\|_{\mathcal B_0\times \mathcal B_1} \le \epsilon$,  there exists $(h, w)\in C^{k+2, \alpha}_c (\Omega) \times C^{k+2, \alpha}_c (\Omega)$ with $\| (h, w) \|_{C^{k+2, \alpha} \times C^{k+2, \alpha}} \le C\| (\psi, V)\|_{C^{k,\alpha}\times C^{k+1, \alpha}} $ such that $(\bar{g}, \bar{\pi}) = (g+ h, \pi + w)\in C^{k+2,\alpha}(\overline\Omega)\times C^{k+2,\alpha}(\overline{\Omega})$ is an initial data set that satisfies 
\[	
	\bar{\mu} - |\bar{J}|_{\bar{g}} \ge \mu - |J+V|_g +  \psi.
\]
If, in addition, $(g, \pi)\in C^{\infty}(\overline{\Omega})$ and $(\psi, V)\in C^{\infty}_c(\Omega)$, then we can achieve $(\bar{g}, \bar{\pi})\in C^{\infty}(\overline{\Omega})$.
\end{theorem}

As an application, we give the following gluing construction of initial data sets from interpolation. This extends the scalar curvature result of E.~Delay~\cite{Delay:2011}, but the presence of $|J|$ adds an analytical subtlety. 

\begin{theorem} \label{theorem:interpolating-main}
Let $k\ge 0$.  Let $(g_0, \pi_0) \in C^{k+4,\alpha} (\overline{\Omega})\times C^{k+3,\alpha}(\overline{\Omega})$ be an initial data set.  Suppose that the kernel of $D\Phi^0_{(g_0, \pi_0)}|_{(g_0, \pi_0)}^*$ is trivial on $\Omega$. Let $0\le \chi \le 1$ be a smooth function such that $\chi(1-\chi)$ is supported on a compact subset of $\Omega$. Then there exists a $C^{k+4,\alpha}(\overline{\Omega})\times C^{k+3,\alpha}(\overline{\Omega})$ neighborhood  $\mathcal{U}$ of $(g_0, \pi_0)$ such that for  $(g_1, \pi_1), (g_2, \pi_2)\in \mathcal{U}$ and for $(g, \pi) = \chi (g_1, \pi_1) + (1-\chi ) (g_2, \pi_2)$, 
 there exists a pair of symmetric tensors $(h, w)$ supported in $\Omega$ such that the initial data set $(\bar{g} , \bar{\pi})= (g+h, \pi + w)\in C^{k+2,\alpha} (\overline{\Omega})\times C^{k+2,\alpha}(\overline{\Omega})$ satisfies 
\[
	\bar{\mu} - |\bar{J}|_{\bar{g}} \ge \chi (\mu_1 - |J_1|_{g_1}) + (1-\chi) (\mu_2 - |J_2|_{g_2}).
\] 
If, in addition, $(g_1, \pi_1), (g_2,\pi_2)\in C^{\infty}(\overline{\Omega})$, then we can achieve $(\bar{g}, \bar{\pi})\in C^{\infty}(\overline{\Omega})$.
\end{theorem}

We will consider deformation and gluing constructions in the asymptotically flat setting, for which it is essential to make use of the modified operator $\Phi^W_{(g,\pi)}$ for $W$ not necessarily $0$. As an application, we show that for any asymptotically flat initial data set \emph{without} assuming the no-kernel condition, one can solve for a new initial data set that interpolates to a model initial data set in a way that the dominant energy condition also interpolates.  Note that at vacuum data, the modified operator $\Phi^0_{(g, \pi)}$ recovers the usual constraint map, so that in particular the adjoint operator $(D\Phi^0_{(g_{\mathbb E}, 0)})^*$ at the flat data has a kernel. In the asymptotic gluing, the deformation may occur far into the asymptotically flat end, so we need to take into account  the finite-dimensional approximate kernel from the flat data, by employing an \emph{admissible family} (see Definition~\ref{definition:admissible}).  This is carried out as in the vacuum case, but we remark that the admissible family in our setting can include not only the Kerr family, but also non-vacuum ones such as the Kerr-Newman family. 

Let $\chi$ be a smooth cutoff function that is $\chi=1$ on the Euclidean unit ball $B_1$ and $\chi=0$ outside $B_2$ with $\chi(1-\chi)$ supported on a compact subset of $B_2 \setminus \overline{B_1}$. Let $\chi_R (y) = \chi(y/R)$ be the rescaled cutoff function. 
\begin{theorem} \label{theorem:gluing}
Let $k\ge 0$. Let $(M,  g, \pi)\in C_{\mathrm{loc}}^{k+4,\alpha} \times C_{\mathrm{loc}}^{k+3,\alpha}$ be an asymptotically flat initial data set with the ADM energy-linear momentum $(E,P)$. Given $\epsilon>0$, there exists $R_0>0$ such that for any $R\geq R_0$, there is an initial data set  $(\bar{g}, \bar{\pi}) \in C^{k+2, \alpha}_{\mathrm{loc}}\times C^{k+2,\alpha}_{\mathrm{loc}}$ with  
\begin{align*}
	(\bar{g}, \bar{\pi}) &= (g, \pi) \quad \mbox{in }  B_R\\
	(\bar{g}, \bar{\pi})&=(g^{\theta}, \pi^{\theta}) \quad \mbox{in } M\setminus B_{2R}
\end{align*}
for some $(g^\theta, \pi^\theta)$ in an admissible family for $(g, \pi)$ so that $(\bar{g}, \bar{\pi})$  satisfies the inequality
\[
	\bar{\mu} - |\bar{J}|_{\bar{g}} \ge \chi_R(\mu - |J|_g) + (1-\chi_R) (\mu^{\theta} - |J^{\theta}|_{g^{\theta}})
\]
with strictly larger ADM energy $E^\theta >E$ and  
\[
	|P^\theta - P|<  E^\theta- E< \epsilon.
\]
If, in addition, $(g, \pi), (g^\theta,\pi^\theta)\in C^{\infty}$, then we can achieve $(\bar{g}, \bar{\pi})\in C^{\infty}$.

\end{theorem}
\begin{remark}
If $E\ge |P|$ in the above theorem, then the ADM mass of $(g^\theta, \pi^\theta)$ is strictly larger than that of $(g,\pi)$, i.e. $\sqrt{(E^\theta)^2 - |P^\theta|^2 }> \sqrt{E^2 - |P|^2}$, by direct manipulations.  
\end{remark}
  
Note that our gluing results do not fully recover the vacuum gluing results even for vacuum initial data sets, since we do not obtain equality for the dominant energy condition. On the other hand, our gluing construction includes the feature to bring up the ADM energies by promoting the dominant energy condition. In particular, we can glue initial data for a Kerr solution (including Schwarzschild data) to initial data for a (different) Kerr solution through a region where the dominant energy condition holds, which may not be (directly) feasible by the vacuum theorem.

The paper is organized as follows.  In Section 2, we introduce basic properties of the modified constraint map and some analytical preliminaries, including an improved estimate for weighted spaces (Proposition~\ref{proposition:weighted-norm}).  In Section~3, we study localized deformation by the modified constraint map and prove Theorem~\ref{theorem:main} and Theorem~\ref{theorem:interpolating-main} (with $k=0$). In Section 4, we prove asymptotic gluing results to an admissible family, including Theorem~\ref{theorem:gluing} (with $k=0$). In Section 5 and Section 6, we prove the local surjectivity theorems for the modified operator and the projected operator, respectively; the analysis follows closely that from the vacuum case, but we include the details to emphasize the uniformity of various required estimates. 

\section{Preliminaries} \label{sec:prelim}

We use the Einstein summation convention, summing over repeated upper and lower indices, throughout, and we use the convention that a semicolon denotes a covariant derivative, while a comma denotes a partial derivative. 

\subsection{Initial data sets}
Let $n\ge 3$. An $n$-dimensional \emph{initial data set} is an $n$-dimensional manifold~$M$ equipped with a $C^2_{\mathrm{loc}}$ Riemannian metric $g$ and a $C^1_{\mathrm{loc}}$ symmetric $(2, 0)$ tensor $K$. The \emph{mass density} $\mu$ and the \emph{current density} $J$ are defined by
\begin{align*}
	\mu &= \frac{1}{2}(R_g - |K|^2_g + (\textup{tr}_g K)^2)\\
	J &= \textup{div}_g K - d(\textup{tr}_g K)
\end{align*}
where $R_g=g^{ij}R_{ij}$ is the scalar curvature of $g$, with $R_{ij}$ the components of the Ricci tensor. 
It is convenient for us to consider the momentum $(2,0)$ tensor
\[
	\pi^{ij}	 = K^{ij} - (\textup{tr}_g K) g^{ij}.
\]
Abusing terminology slightly, we  refer to $(g, \pi)$ as an \emph{initial data set} throughout this paper. The initial data set is said to satisfy the \emph{dominant energy condition} if 
\[
	\mu\ge |J|_g
\] 
holds everywhere in $M$. 

The \emph{constraint map} is defined by
\begin{align*}
  \Phi(g, \pi)=\left(R(g)+ \tfrac{1}{n-1} (\mathrm{tr}_g \pi)^2 - | \pi|^2_g,\; \textup{div}_g\pi\right) =(2\mu, J).
\end{align*}
The linearization  is given by the following formula (see, for example, \cite[Lemma 20]{Eichmair-Huang-Lee-Schoen:2016})
\begin{align} \label{equation:lin}
\begin{split}
	D\Phi|_{(g,\pi)} (h , w) &=\Big(L_g h -2 h_{ij} \pi_\ell^i \pi^{j\ell} - 2 \pi^j_k w^k_j  +\tfrac{2}{n-1}\mbox{tr}_g \pi (h_{ij} \pi^{ij} + \mbox{tr}_g w), \\
	&\qquad ( \mbox{div}_g w)^i - \tfrac{1}{2} \pi^{jk} h_{jk;\ell} g^{\ell i} + \pi^{jk} h^i_{j;k} +\tfrac{1}{2} \pi^{ij} (\mbox{tr}_g h)_{,j}\Big).
\end{split}
\end{align}
Here all indices are raised or lowered with respect to $g$.  The linearized scalar curvature operator $L_g(h)= -\Delta_g(\mbox{tr}_g h) + \mbox{div}_g \mbox{div}_g (h) - h^{ij} R_{ij}$ appears above.  The formal $L^2$ adjoint operator of $D\Phi|_{(g,\pi)}$ is given by 
 \begin{align}
 \begin{split}
 	&D\Phi|_{(g,\pi)} ^*(f, X) \\ & = \left(  L_g^*f +\big( \tfrac{2}{n-1} (\mbox{tr}_g \pi) \pi_{ij} - 2 \pi_{ik} \pi^k_j \big) f\right.\\
	& \quad+ \tfrac{1}{2} \left( g_{i\ell}g_{jm} (L_X\pi)^{\ell m} + (X^k_{;k}) \pi_{ij}- X_i \pi^k_{j;k} - X_j\pi^k_{i;k} - X_{k;m} \pi^{km} g_{ij} - X_k \pi^{km}_{;m} g_{ij} \right), \\
	&\quad \left. -\tfrac{1}{2} (\mathcal{D}_g X)^{ij} + \big(\tfrac{2}{n-1} (\mbox{tr}_g \pi ) g^{ij}- 2 \pi^{ij}  \big) f\right),\label{equation:lin-adj}
	\end{split}
 \end{align}
 where $L_g^*f = -(\Delta_g f)g + \textup{Hess}_g f - f \textup{Ric}(g)$,   $L_X \pi$ is the Lie derivative, and $ \mathcal{D}_g X = L_X g$ is the Lie derivative operator $(\mathcal D_g X)^{ij}= X^i_{\; ;\ell}g^{\ell j}+ X^j_{\; ;\ell}g^{\ell i}$.  See  \cite[Lemma 2.3]{Corvino-Schoen:2006} for $n=3$, \cite[Lemma 20]{Eichmair-Huang-Lee-Schoen:2016} for general $n$.

\subsection{Modified constraint map}\label{subsection:modified}
Let $(g, \pi)$ be an initial data set and let $W$ be a vector field. We define the \emph{modified} constraint map ${\Phi}_{(g,\pi)}^W$ by 
\begin{equation} \label{eq:def-mco} 
 \Phi^W_{(g, \pi)}(\gamma, \tau)= \Phi(\gamma, \tau)+  (0, \tfrac{1}{2} \gamma \cdot_g (\textup{div}_g \pi +W)), 
\end{equation}
where $ (\gamma \cdot_g Y)^i= g^{ij} \gamma_{jk} Y^k$ in local coordinates. The linearized operator at $(g, \pi)$ is denoted by $D{\Phi}_{(g,\pi)}^W=D{\Phi}_{(g,\pi)}^W|_{(g,\pi)} $ and has the following expression
\begin{align} \label{equation:modified}
	&D{\Phi}_{(g,\pi)}^W(h , w) =D\Phi |_{(g,\pi)} (h,w) + (0,\tfrac{1}{2} h\cdot_g (\textup{div}_g \pi +W)).
\end{align}
 The formal $L^2$ adjoint operator has the expression
\begin{align} \label{equation:modified-adjoint}
	(D{\Phi}^W_{(g,\pi)})^*(f, X) = D\Phi|^*_{(g,\pi)} (f, X) + \left( \tfrac{1}{4} [ X_i (\textup{div}_g \pi +W)_j+X_j  (\textup{div}_g \pi+W)_i], 0\right),
\end{align}
where the indices are lowered by $g$.

\subsection{Kernel of the adjoint operators}
We include regularity results for any kernel element $(f,X)\in H^2_{\mathrm{loc}}(U)\times H^1_{\mathrm{loc}}(U)$ where $U\subset M$ is an open subset, from which we can obtain higher order regularity depending on the smoothness of the initial data sets.  The analysis is similar to the scalar curvature operator in \cite[Section 2.2]{Corvino:2000}. 

\begin{proposition}   \label{proposition:kernel0}
Let $k\geq 2$, $\alpha \in (0,1)$. Let $(g,\pi) \in C_{\textup{loc}}^{k,\alpha}(U)\times C_{\textup{loc}}^{k-1, \alpha}(U)$ be an initial data set. Suppose that $(f, X)\in H^2_{\mathrm{loc}}(U)\times H^1_{\mathrm{loc}}(U)$ satisfies $D\Phi|^*_{(g,\pi)}(f, X)=0$ weakly.  Then the following holds:
\begin{enumerate}
\item $(f,X)\in C^{k,\alpha}_{\textup{loc}}(U) \times C^{k,\alpha}_{\textup{loc}}(U)$.  
\item If $(g,\pi) \in C^{k,\alpha}(\overline{U})\times C^{k-1, \alpha}(\overline{U})$, then $(f, X) \in C^{k-2,\alpha}(\overline{U}) \times C^{k-2,\alpha}(\overline{U})$.
\item  If $U$ is connected, the space of solutions $(f, X)\in H^2_{\mathrm{loc}}(U)\times H^1_{\mathrm{loc}}(U)$ to the homogeneous equation $D\Phi|^*_{(g, \pi)}(f,X)=0$ is finite-dimensional, and a non-trivial solution cannot vanish on any open subset of $U$.  
\end{enumerate}
\end{proposition}  

\begin{proof}
Let $(f, X) \in H^2_{\mathrm{loc}}(U)\times H^1_{\mathrm{loc}}(U)$ satisfy $D\Phi|_{(g,\pi)} ^*(f, X)=0$. Taking the trace of the first component of \eqref{equation:lin-adj} gives an equation for $\Delta_g f$. Using this equation, we can eliminate the term $\Delta_g f$ from the first component of the system $D\Phi|_{(g,\pi)} ^*(f, X)=0$ to obtain
\begin{align} 
	f_{;ij} &= A_{ij} f+ B_{ijk} X^k + C^{\ell}_{ijk} X^k_{\; ;\ell}  \label{equation:kernel1}\; ,
\end{align}	
where $A_{ij}$, $B_{ijk}$ and $C^{\ell}_{ijk}$ are functions locally computed as polynomials in the components $g_{ab}$, $g^{ab}$, $\pi_{ab}$, $\partial_c g_{ab}$, $\partial^2_{cd} g_{ab}$, and $\partial_c \pi_{ab}$.  The other components in $D\Phi|_{(g,\pi)} ^*(f, X)=0$ constitute the following system:
\begin{align}
	\tfrac{1}{2} ( \mathcal{D}_g X)^{ij}&=\big(\tfrac{2}{n-1} (\mbox{tr}_g \pi ) g^{ij}- 2 \pi^{ij}  \big) f .	\label{eq:ker-X}
\end{align}
Thus $\mathcal D_g X \in H^1_{\mathrm{loc}}(U)$, and by commuting the order of derivatives and using the Ricci formula, we have
\begin{align} \label{equation:kernel2}	
\begin{split}
\Big[ (\mathcal D_g X)_{ij;k} & +  (\mathcal D_g X)_{ki;j}- (\mathcal D_g X)_{jk;i})\Big]  \\
 &=( X_{i;jk}+ X_{i;kj} )+ (X_{j;ik}-X_{j;ki}) + (X_{k;ij}-X_{k;ji})\\
 &= 2  X_{i;jk}+ (R^{\ell}_{kji} + R^{\ell}_{ikj} + R^{\ell}_{ijk}) X_{\ell},
 \end{split}
\end{align}
where the sign convention for the Riemannian curvature tensor is so that the Ricci tensor $R_{jk} = R^{\ell}_{\ell jk}$. Along with (\ref{eq:ker-X}), this implies that $X\in H^2_{\mathrm{loc}}(U)$.  
By taking the trace of the first component of $D\Phi|^*_{(g,\pi)}(f,X)=0$ and the divergence of the other components,  $(f, X)$ satisfies a second order elliptic linear system (see \cite[Proposition 3.1]{Corvino-Schoen:2006} for $n=3$ and \cite[Lemma 20]{Eichmair-Huang-Lee-Schoen:2016} for general $n$). The desired interior regularity $(f, X)\in  C^{k,\alpha}_{\mathrm{loc}}(U) \times C^{k,\alpha}_{\mathrm{loc}}(U)$ follows from elliptic regularity.  

 Let $\gamma = \gamma(t)$ be a geodesic. Because $\nabla_{\gamma'}\gamma'=0$, we have 
\begin{align}\label{equation:geodesic}
\begin{split}
(f\circ \gamma)''(t)&= f_{; ij}|_{\gamma(t)} \dot \gamma^i(t) \dot \gamma^j(t)\\
	\left(\frac{D^2 X(\gamma(t))}{dt^2}\right)^k &= (\nabla_{\dot \gamma}\nabla_{\dot\gamma} X)^k = X^k_{\; ; ij}|_{\gamma(t)} \dot\gamma^i(t) \dot \gamma^j(t). 
	\end{split}
\end{align}
We have the formula for $f_{;ij}$ in \eqref{equation:kernel1}. The term $X^k_{;ij}$ is obtained from \eqref{equation:kernel2}:
\begin{align} \label{equation:hessianX}
\begin{split}
	X^i_{;jk} &=  \frac{1}{2}g^{i\ell} \left[ (\mathcal D_g X)_{\ell j;k}+ (\mathcal D_g X)_{k\ell;j}- (\mathcal D_g X)_{jk;\ell})\right]  -\frac{1}{2} g^{i\ell}\widetilde{C}^{p}_{\ell jk}  X_{p} \\
	&= \widetilde{A}^i_{jk} f + \widetilde{B}^{i\ell}_{jk} f_{; \ell} -\frac{1}{2} g^{i\ell} \widetilde{C}^{p}_{\ell jk}X_{p},
\end{split}
\end{align}
where $\widetilde{A}^i_{jk}, \widetilde{B}^{i\ell}_{jk}, \widetilde{C}^{\ell}_{ijk}$ are locally computed as polynomials in $g_{ab}$, $g^{ab}$, $\pi_{ab}$, $\partial_c g_{ab}$, $\partial^2_{cd} g_{ab}$, and $\partial_c \pi_{ab}$.  For convenience, let $\{ E_i(t) , i = 1, \dots, n\}$ be a parallel orthonormal frame field along~$\gamma$.   Let $X(\gamma(t))= X^i (t) E_i(t)$, and $X^0(t)= f(\gamma(t))$.  Let   $Z(t)$ be the column vector with the components $X^0(t), X^1(t), \dots, X^n(t)$. Then by \eqref{equation:geodesic}, \eqref{equation:kernel1}, and \eqref{equation:hessianX}, the vector satisfies a second-order linear system of ordinary differential equations along any geodesic $\gamma$ in $U$:
\[
	Z''(t)= A(t) Z'(t) + B(t) Z(t),
\]
where $A(t)$ and $B(t)$ are $(n+1)\times (n+1)$ matrix functions whose components are computed locally as polynomials in $g_{ij}$, $g^{ij}$, $\pi_{ij}$, $\partial_k g_{ij}$, $\partial^2_{k\ell} g_{ij}$, and $\partial_k \pi_{ij}$, evaluated along $\gamma$. If $U$ is connected, then $(f,X)$  is determined by its 1-jet at a point in $U$, and thus the dimension of the kernel is at most $(n+1)^2$, and any non-trivial element in the kernel cannot vanish on an open subset.

Boundary regularity for $(f,X)$ follows from the ODE argument. By extending the initial data set $(g, \pi)$ in a neighborhood of the boundary, we may assume that $\overline{U}$ is in a manifold interior. Let $q \in \partial U$ and let $2r>0$ be the injectivity radius at $q$.  For a point $p\in U$ with $d(p, q)<r$, we can extend $(f,X)$ on $ B_r(p)\setminus U$ along the unique geodesic in $B_r(p)$ starting at $p$ with initial velocity $v= \exp^{-1}(x)$ that reaches $x \in B_r(p)\setminus U$. By the smooth dependence of solutions of ODE on parameters (note that $\exp^{-1}$ is $C^{k-1,\alpha}$ and the ODE system involves $\mathrm{Ric}(g)$, which accounts in part for the ensuing regularity), we have that $(f,X)\in C^{k-2,\alpha}_{\mathrm{loc}}(B_r(p))\times C^{k-2,\alpha}_{\mathrm{loc}}(B_r(p))$.

\end{proof}

The above proposition also applies to the modified constraint operator because its adjoint operator $(D\Phi^W_{(g,\pi)})^*$ differs from $D\Phi|^*_{(g,\pi)}$ only by  a zero-th order term. Essentially the same proof implies the following statement.

\begin{proposition} \label{proposition:kernel}
Let $k\geq 2$, $\alpha\in (0,1)$. Let $(g, \pi) \in C_{\mathrm{loc}}^{k,\alpha}(U)\times C_{\mathrm{loc}}^{k-1,\alpha}(U)$ be an initial data set. Let $W\in C^{k-2,\alpha}_{\mathrm{loc}}(U)$ be a vector field.  Suppose that $(f, X)\in H^2_{\mathrm{loc}}(U)\times H^1_{\mathrm{loc}}(U)$ satisfies $(D\Phi^W_{(g,\pi)})^*(f,X)=0$ weakly.  Then the following holds:
\begin{enumerate}
\item $(f, X) \in C_{\mathrm{loc}}^{k,\alpha}(U) \times C_{\mathrm{loc}}^{k,\alpha}(U)$.
\item If  $(g, \pi), (\gamma, \tau) \in C^{k,\alpha}(\overline{U})\times C^{k-1,\alpha}(\overline{U})$, $W\in C^{k-2, \alpha}(\overline{U})$, then $(f, X) \in C^{k-2,\alpha}(\overline{U}) \times C^{k-2,\alpha}(\overline{U})$.  
\item If  $U$ is connected, the space of solutions $(f, X)\in H^2_{\mathrm{loc}}(U)\times H^1_{\mathrm{loc}}(U)$ to the homogeneous equation $(D\Phi^W_{(g,\pi)})^*(f,X)=0$ is finite-dimensional, and a non-trivial solution cannot vanish on any open subset of $U$. 
\end{enumerate}
\end{proposition}

\begin{definition}
The \emph {kernel}  of $(D\Phi^W_{(g,\pi)})^*$ on $U$ is the set $K\subset H^2_{\mathrm{loc}}(U)\times H^1_{\mathrm{loc}}(U)$ which consists of those $(f,X)$ that satisfy $(D\Phi^W_{(g,\pi)})^*(f, X)=0$ weakly. The kernel of $(D\Phi^W_{(g,\pi)})^*$ is said to be \emph{trivial} on $U$ if $K=\{0 \}$. 
\end{definition}

\begin{example}\label{example:linearization}
Consider the flat data  $( g_{\mathbb{E}}, 0)$ on an open connected subset of $\mathbb{R}^3$ and $W=0$. Then the modified operator is the usual constraint map, and its formal $L^2$ adjoint operator $D\Phi|_{(g_{\mathbb{E}}, 0)}^*(f, X) = (-(\Delta_{g_{\mathbb{E}}} f )g_{\mathbb{E}} + \textup{Hess}_{g_{\mathbb{E}}}f, -\frac{1}{2} \mathcal{D}_{g_{\mathbb{E}}}X)$ has a ten-dimensional kernel $K = K_0 \oplus K_1$, where
\begin{align*}
	&K_0 = \mbox{span} \{  1, x^1, x^2, x^3\}\\
	&K_1 = \mbox{span} \left\{\frac{\partial }{\partial x^1}, \frac{\partial }{\partial x^2}, \frac{\partial }{\partial x^3}, x\times \frac{\partial}{\partial x^1} ,  x\times \frac{\partial}{\partial x^2},  x\times \frac{\partial}{\partial x^3}\right \}.
\end{align*}
\end{example}

\subsection{Weighted Sobolev spaces} \label{subsection:weighted}

Let $d_{g}(x)= d_g(x,\partial \Omega)$ be the distance to the boundary with respect to $g$; the boundary is assumed to be a smooth hypersurface, so near $\partial \Omega$, $d$ is as regular as $g$.  We will work with  uniformly equivalent metrics in a bounded open set $\mathcal{U}_0$ in the space of $C^m(\overline{\Omega})$ $(m\ge 2)$  Riemannian metrics such that $\| d_g \|_{C^m}$ is uniformly bounded near $\partial \Omega$. We will establish a framework uniformly across~$\mathcal{U}_0$ in what follows. 

 Let $V_\Omega= \{ x\in \Omega: d_{g}(x)<r_0\mbox{ for some }g\in \mathcal{U}_0\}$ be a thin regular collar neighborhood of $\partial \Omega$.  There is $r_0\in (0,\tfrac{1}{2})$ sufficiently small so that a neighborhood of $V_\Omega$ is foliated by smooth (as regular as the metric $g$ is) level sets of $d_g$ and that $d_g(x)\le \frac{1}{2} $ for all $x\in V_\Omega$ and $g\in \mathcal{U}_0$. 

 Let $0<r_1<r_0$ be fixed. Define a smooth positive monotone function $\tilde \rho: (0, \infty) \rightarrow \mathbb R$  such that  $\tilde \rho(t)= e^{-1/t}$ for $t\in (0, r_1)$ and $\tilde\rho(t)= 1$ for $t> r_0$.
 
 \begin{notation}[Exponential weight function] 
For $N>0$, let $\rho_g$ be the positive function on $\Omega$ defined by
\[
	\rho_g(x)= ({\tilde\rho}\circ d_g(x))^N.
\] 
\end{notation}
We will eventually fix $N$ to be a large number, chosen for some $k\in \mathbb Z_+$ so that 
\begin{align} \label{equation:N}
	N > \max\{4(4k-3), 4C_0\}
\end{align}
where $C_0>0 $ is the constant appearing in \eqref{equation:X}. While the discussion in this section holds for all $k\le m$, in this paper we only apply \eqref{equation:N}  for $k\le 2$ in the variational argument in Section~\ref{section:modified}.

Let $L^2_{\rho_g}(\Omega, g)$ be the set of functions or tensor fields $u$ such that $|u|\rho_g^{\frac{1}{2}}\in L^2(\Omega, g)$ with the norm defined by
\[
\| u \|_{L^2_{\rho_g}(\Omega, g)} = \left( \int_{\Omega} |u|^2\rho_g\, d\mu_g \right)^{\frac{1}{2} }.
\] 
The pairing $\langle u, v\rangle_{L^2_{\rho_g}(\Omega, g)} = \langle u\rho_g^{\frac{1}{2}}, v\rho_g^{\frac{1}{2}}\rangle_{L^2(\Omega, g)}$ makes $L^2_{\rho_g}(\Omega, g)$ a Hilbert space. Let $H^k_{\rho_g}(\Omega, g)$ be the Hilbert space of functions or tensor fields whose covariant derivatives up to and including order $k$ with respect to $g$ are also in $L^2_{\rho_g}(\Omega, g)$ with the norm defined as
\[
	\| u \|^2_{H^k_{\rho_g}(\Omega, g)} = \sum_{j=0}^k \| \nabla^j _g u \|^2_{L^2_{\rho_g}(\Omega, g)} = \sum_{j=0}^k  \int_{\Omega } | \nabla_g^{j} u |^2 \rho_g \, d\mu_g.
\]
By \cite[Lemma 2.1]{Corvino-Schoen:2006}, $H^k(\Omega, g)$ (and hence $C^{\infty}(\overline{\Omega})$) is dense in $H^k_{\rho_g}(\Omega, g)$.  We note that  the tensor fields in $H^k(\Omega,g)$ are the same across $g$, and while we can further shrink $\mathcal U_0$ so that the norms are equivalent for $g\in \mathcal U_0$ as well, the weighted norms $H^k_{\rho_g}(\Omega, g)$ may not be equivalent as $g$ varies in $\mathcal U_0$. (Note that they would be equivalent if $\rho$ were a power weight function, i.e. $\tilde \rho(t)= t$ near the boundary, and such a weight can be used for simplicity to establish the finite regularity results \cite{Corvino:2000, Corvino-Schoen:2006}.)  We often suppress $\Omega$ and $g$ from the notation when it is clear from the context. 

It is useful to compare the norms $\| u \|_{H^k_{\rho_g}}$ and  $\| u \rho_g^{\frac{1}{2}}\|_{H^k}$. We shall show that  $u \in H^k_{\rho_g}$ implies 
\[
\| u \rho_g^{\frac{1}{2}} \|_{H^k} \le C \| u \|_{H^k_{\rho_g}},
\] 
where the constant $C$ is uniform in~$\mathcal U_0$.   

We begin with basic lemmas for which we work at a fixed metric $g\in C^m(\overline{\Omega})$ ($m\geq 2$) and write $d = d_g$ and $\rho=\rho_g$.  We note explicit dependence of the constants in the estimates so that the estimates will hold uniformly across $\mathcal U_0$.  Let $k\in \{ 1, 2, \ldots, m\}$.

\begin{lemma}\label{lemma:rho}
Let $N\ge 1$. There is a constant $C>0$ independent of $N$ (depending only on $k$, $\| \log \tilde\rho\|_{C^{k}([r_1, r_0])}$ and $\| d \|_{C^k(\overline{V_\Omega})}$) such that
\[
	|\nabla^k (\rho^{\frac{1}{2}}) | \le CN^k  \rho^{\frac{1}{2}}d^{-2k}.
\]
holds on $\Omega$.
\end{lemma}
\begin{proof}
By direct computation, 
\begin{align}\label{equation:rho-boundary}
	\nabla (\rho^{\frac{1}{2}}) (x) =\left\{ \begin{array}{ll}\frac{1}{2} N   \rho^{\frac{1}{2}} d^{-2} \nabla d & \mbox{if $0<d(x)\le r_1$} \\
	 \frac{1}{2}N   \rho^{\frac{1}{2}} d^{-2}\left(d^2(\log\tilde{\rho})'(d)\right)\nabla d &  \mbox{if $r_1\le d(x)\le r_0$}\\
	 0 \qquad & \mbox{if $d(x) \ge r_0$}\end{array}\right..
\end{align}
This implies the estimate for $k=1$. The estimate for $k>1$ follows from induction. 
\end{proof}

\begin{lemma}\label{lemma:rho2} 
Let $N\ge 1$. There is $r_2\in (0,r_1)$  independent of $N$ (depending only on $\| \Delta d\|_{C^0(\overline{V_\Omega})}$)  such that if $0< d(x)\le  r_2$, we have
 \begin{equation} \label{eq:lap-rho} 
\frac{1}{2} N^2 d^{-4} \rho \le \Delta \rho.
\end{equation}
\end{lemma}
\begin{proof}
By direct computation, for  $0<d(x) < r_1$, 
\[
\Delta \rho=N^2 d^{-4} \rho (1+ N^{-1} d^2 \Delta d - 2N^{-1}d).
\]
If $r_2$ is sufficiently small, for $0< d(x) \le r_2$, 
\[
	1+ N^{-1} d^2 \Delta d - 2N^{-1}d \ge 1 - r_2^2\| \Delta d\|_{C^0(\overline{V_\Omega})} - 2r_2 \ge  \frac{1}{2}.
\]
\end{proof}

In the next lemma we use a cutoff function $\xi=\xi_g$ to handle estimates near the boundary. Let $0\le \xi_g\le 1$ be smooth with $\xi_g=0$ on the compact subset $\{ x\in \Omega: d_g(x) \ge r_2\}$ and $\xi_g= 1$ in a collar neighborhood $\{ x\in \Omega: d_g(x) \le r_2/2\}$ of $\partial \Omega$ with $|\nabla \xi|_g \le 4/r_2$.  In the following lemma, $u$ can be a function or a tensor field.

\begin{lemma} \label{lemma:ibp}
 For $j \in \{ 1, 2, \ldots, m\}$, and for $u\in C^j(\overline{\Omega})$, if $N \ge 4(4j-3)$, then 
\begin{align*}
	\int_\Omega \xi |u|^2 d^{-4j} \rho \, d\mu_g & \le \left( \frac{4}{ N}\right)^j \int_\Omega \xi | \nabla^j u|^2 \rho\, d\mu_g \\
	&\qquad  + \sum_{i=1}^j \left(\frac{4}{N} \right)^{j+1-i}\sup_\Omega (|\nabla \xi| d^{-4i+2}) \| \nabla^{j-i} u \|_{L_\rho^2(\Omega)}^2.
\end{align*}
\end{lemma}
\begin{proof}
Multiplying $\xi |u|^2 d^{-4j+4}$ to \eqref{eq:lap-rho} and noting $|\nabla \rho| \le Nd^{-2} \rho$ and $d\le 1$ in the supports of $\xi$ and $\nabla \xi$ by \eqref{equation:rho-boundary}, we have
\begin{align*}
	\frac{1}{2} N^2 \int_\Omega & \xi |u|^2 d^{-4j} \rho\, d\mu_g \le \int_\Omega \xi |u|^2 d^{-4j+4} \Delta \rho \, d\mu_g\\
	&\le  \int_\Omega \left[ 2 \xi |\nabla u ||u| d^{-4j+4}  + (4j-4)\xi |u|^2 d^{-4j+3}  +  | \nabla \xi| |u|^2 d^{-4j+4} \right] |\nabla \rho| \, d\mu_g\\
	&\le N \int_\Omega \left[ 2 \xi |\nabla u ||u| d^{-4j+2}  + (4j-4)\xi |u|^2 d^{-4j+1}  +  | \nabla \xi| |u|^2 d^{-4j+2} \right] \rho\, d\mu_g\\
	&\le N\left( (4j-3) \int_\Omega \xi |u|^2 d^{-4j} \rho\, d\mu_g+ \int_\Omega \xi |\nabla u|^2 d^{-4j+4} \rho\, d\mu_g + \sup_\Omega (| \nabla \xi| d^{-4j+2}) \|u\|_{L^2_\rho}^2\right),  
\end{align*}
where we applied the AM-GM inequality in the last inequality. Absorbing the first term into the left hand side, we have 
\begin{align*}
	&\int_\Omega \xi |u|^2 d^{-4j} \rho\, d\mu_g\le    \frac{4}{N}\int_\Omega \xi |\nabla u|^2 d^{-4j+4} \rho\, d\mu_g +\frac{4}{N} \sup_\Omega (| \nabla \xi| d^{-4j+2}) \|u\|_{L^2_\rho}^2  .
\end{align*}
This proves the case $j=1$. The case $j>1$ follows by induction. 

\end{proof}

\begin{corollary}\label{corollary:weight}
Let $u\in H^k_{\rho}(\Omega)$ and $N \ge  4(4k-3)$. For $j \in\{ 0,1, \dots,  k\}$,
\[
	\int_\Omega |\nabla^{k-j} u |^2 d^{-4j} \rho\, d\mu_g \le C \| u\|^2_{H^k_\rho(\Omega)},
\]
where $C$ depends on $N$, $j$ and $r_2$.
\end{corollary}
\begin{proof}
By density, it suffices to prove the estimate for $u\in C^{\infty}(\overline{\Omega})$. By Lemma~\ref{lemma:ibp},
\begin{align} \label{inequality:boundary-estimate}
\begin{split}
	\int_{\Omega} & |\nabla^{k-j} u |^2 d^{-4j} \rho\, d\mu_g  \\
	&= \int_{\Omega} \xi |\nabla^{k-j} u |^2 d^{-4j} \rho \, d\mu_g +\int_{\Omega} (1-\xi) |\nabla^{k-j} u |^2 d^{-4j} \rho\, d\mu_g \\
	&\le \left( \frac{4}{ N}\right)^j \int_\Omega \xi |\nabla^k u|^2 \rho + \sum_{i=1}^j \left( \frac{4}{ N} \right)^{j+1-i} \sup_\Omega (|\nabla \xi| d^{-4i+2}) \| \nabla^{k-i}u\|^2_{L^2_\rho}\\
	&\quad + \sup_\Omega((1-\xi) d^{-4j}) \| \nabla^{k-j} u\|^2_{L^2_\rho}. 
\end{split}
\end{align}	
This implies the desired inequality. 
\end{proof}

\begin{proposition}\label{proposition:weighted-norm}
Let  $u\in H^k_\rho (\Omega)$ and $N \ge  4(4k-3)$. Then 
\[
	\| u\rho^{\frac{1}{2}}\|_{H^k(\Omega)} \le C \| u\|_{H^k_\rho(\Omega)},
\]
where $C$ depends on $N$, $k$, $r_2$, $\| \log \tilde{\rho}\|_{C^{k}([r_1, r_0])}$ and $\| d \|_{C^k(\overline{V_\Omega})}$.
\end{proposition}
\begin{proof}
Recall
\[
	\| u\rho^{\frac{1}{2} } \|^2_{H^k} = \sum_{j=0}^k \| \nabla^j (u\rho^{\frac{1}{2}}) \|^2_{L^2}.
\]
We may assume $u\in C^\infty(\overline{\Omega})$, using density. By Lemma~\ref{lemma:rho}, for some constants $C_{ij}$, 
\begin{align*}
	|\nabla^j (u\rho^{\frac{1}{2}})| &\le | \nabla^j u| \rho^{\frac{1}{2}} + \sum_{i=1}^{j} C_{ij}|\nabla^{j-i} u| \; |\nabla^i (\rho^{\frac{1}{2}})|\\
	&\le  | \nabla^j u| \rho^{\frac{1}{2}} + C\sum_{i=1}^{j}N^i |\nabla^{j-i} u | d^{-2i} \rho^{\frac{1}{2}}.
\end{align*}
Therefore,
\begin{align*}
	\| \nabla^j (u\rho^{\frac{1}{2}}) \|^2_{L^2} & = \int_\Omega |\nabla^j (u\rho^{\frac{1}{2}})|^2\, d\mu_g \\
	&\le 2 \int_\Omega  | \nabla^j u|^2 \rho\, d\mu_g + C^2N^{2j} \sum_{i=1}^j \int_\Omega |\nabla^{j-i} u |^2 d^{-4i} \rho\, d\mu_g.
\end{align*}
This implies the desired estimate by Corollary~\ref{corollary:weight}.
\end{proof}

\subsection{Weighted H\"{o}lder spaces} \label{subsection:weighted-Holder}

Here we follow the idea of \cite{Chrusciel-Delay:2003} to consider weighted H\"{o}lder norms in small balls $B_{\phi(x)}(x)$ that cover $\Omega$. The weight function $\phi = \phi_g$ satisfies the following properties with uniform estimates across $g$ in a $C^m(\overline{\Omega})$ neighborhood $\mathcal{U}_0$. We suppress the subscript $g$ in  $\phi, d, \rho, \nabla$, when the context is clear. Recall the neighborhood $V_\Omega$ defined in the previous section, and suppose we have chosen $N$ suitably as in \eqref{equation:N}.

\begin{proposition} \label{proposition:phi} For $g\in \mathcal U_0$, we define  $\phi(x)=(d(x))^2$ in $V_\Omega$. There exists a constant $C>0$, uniform across $\mathcal U_0$, such that we can extend $\phi$ to $\Omega$ with $0<\phi<1$ and with the following properties.
\begin{enumerate}
\item  $\phi$ has a positive lower bound on $\Omega\setminus V_\Omega$ uniformly in $g\in \mathcal{U}_0$, and for each $x$, $\phi(x)<d(x)$, so that $\overline{B_{\phi(x)}(x)}\subset \Omega$.  \label{item:1}
\item \label{item:2} For $x \in \Omega$ and $k\le m$, we have 
\begin{align*}
|\phi^k \rho^{-1} \nabla^k \rho| &\le C.
\end{align*}
\item \label{item:3} For $x\in \Omega$ and for $y\in B_{\phi(x)}(x)$, we have
\begin{align}\label{equation:weight}
\begin{split}
	C^{-1} \rho(y) &\le \rho(x) \le C \rho(y)\\
	C^{-1} \phi(y) &\le \phi(x) \le C \phi(y).
\end{split}
\end{align}
\end{enumerate}
\end{proposition}
\begin{proof} 
 \eqref{item:1} is obvious. \eqref{item:2} follows by Lemma~\ref{lemma:rho}. It is clear that  \eqref{item:3} holds for $x\in \Omega\setminus V_\Omega$ and $y\in B_{\phi(x)}(x)$,  since both $\rho$ and $\phi$ have positive uniform lower bounds.  For $x\in V_\Omega$ and $y\in B_{\phi(x)} (x)$, the triangle inequality implies $d(x)-\phi(x) \le d(y) \le d(x) + \phi(x)$. The desired estimates follow since $\phi = d^2$ and $d\le \frac{1}{2}$ in $V_\Omega$. 
\end{proof}

Let  $r,s \in \mathbb{R}$ and $\varphi = \phi^r \rho^s$. For  $u\in C^{k,\alpha}_{\mathrm{loc}}(\Omega)$,  we define the weighted H\"{o}lder norms $\| u \|_{C^{k,\alpha}_{\phi, \varphi}(\Omega)}$ by 
\[
	\| u \|_{C^{k, \alpha}_{\phi, \varphi} (\Omega)} =  \sup_{x\in \Omega} \left(\sum_{j=0}^k \varphi(x) \phi^j(x) \| \nabla_g^j u \|_{C^0(B_{\phi(x)}(x))} + \varphi(x) \phi^{k+\alpha}(x) [\nabla^k_g u]_{0, \alpha; B_{\phi(x)}(x)}\right).
\] 
Note that $\phi$ is to make the norm scaling invariant with respect to the size of the ball. The weighted H\"{older} space $C^{k, \alpha}_{\phi, \varphi}(\Omega)$ consists of $C^{k,\alpha}_{\mathrm{loc}}(\Omega)$ functions or tensor fields with finite $C_{\phi, \varphi}^{k,\alpha}(\Omega)$ norm. If $u\in C_{\phi, \varphi}^{k,\alpha}(\Omega)$, then $u$ is dominated by $\varphi^{-1}$ in the sense that $u = O(\varphi^{-1})$ and $\nabla^j u = O(\varphi^{-1}\phi^{-j})$ near the boundary. The norms are equivalent to those introduced in \cite[Appendix A]{Chrusciel-Delay:2003} (cf. \cite{Corvino-Eichmair-Miao:2013}). 

Note that differentiation is a continuous map from $C^{k, \alpha}_{\phi, \varphi}$ to $C^{k-1, \alpha}_{\phi, \phi \varphi}$. For  $u\in C^{k,\alpha}_{\phi, \varphi}(\Omega), v\in C^{k,\alpha}(\overline{\Omega})$ we have $uv \in C^{k,\alpha}_{\phi, \varphi}(\Omega)$ with 
\[
	\| uv\|_{C^{k,\alpha}_{\phi, \varphi}(\Omega)} \le C \| u \|_{C^{k,\alpha}_{\phi, \varphi}(\Omega)} \| v \|_{C^{k,\alpha}(\overline{\Omega})}
\]	
where $C$ depends only on $k$.  Furthermore, using Lemma~\ref{lemma:rho}, it follows that multiplication by $\rho$ is a continuous map from $C^{k, \alpha}_{\phi, \varphi}$ to $C^{k, \alpha}_{\phi, \varphi \rho^{-1}}$. 

We will use the following Banach spaces $\mathcal{B}_k(\Omega)$ (for functions or tensor fields): 
 \begin{align*}
 	\mathcal{B}_k(\Omega) &= C^{k,\alpha}_{\phi, \phi^{4-k+\frac{n}{2}}\rho^{-\frac{1}{2}}}(\Omega) \cap L^2_{\rho^{-1}}(\Omega) \qquad (\mbox{for $k=0, 1,2$})\\
	\mathcal{B}_3(\Omega) &= C^{3,\alpha}_{\phi, \phi^{1+\frac{n}{2}}\rho^{\frac{1}{2}}} (\Omega) \cap H^{1}_\rho(\Omega)\\
 	\mathcal{B}_4(\Omega) &= C^{4,\alpha}_{\phi, \phi^{\frac{n}{2}}\rho^{\frac{1}{2}}} (\Omega) \cap H^{2}_\rho(\Omega)\; ,
 \end{align*}
with the Banach norms: 
\begin{align*}
	\| u \|_{\mathcal{B}_k(\Omega) } &=   \|u\|_{C^{k,\alpha}_{\phi, \phi^{4-k+\frac{n}{2}}\rho^{-\frac{1}{2}}}(\Omega) } + \| u\|_{L^2_{\rho^{-1}}(\Omega) }  \qquad (\mbox{for $k=0, 1,2$})\\
	 \| X \|_{\mathcal{B}_3(\Omega) } &=  \|X\|_{C^{3,\alpha}_{\phi, \phi^{1+\frac{n}{2}}\rho^{\frac{1}{2}}}(\Omega) } +\| X\|_{H^1_{\rho}(\Omega) } \\
	\| f \|_{\mathcal{B}_4(\Omega) } &=   \|f\|_{C^{4,\alpha}_{\phi, \phi^{\frac{n}{2}}\rho^{\frac{1}{2}}}(\Omega) } + \| f\|_{H^2_{\rho}(\Omega) }.
\end{align*}
It is clear that these Banach spaces contain the smooth functions with compact supports in $\Omega$, and that $\mathcal B_2(\Omega) \subset \mathcal B_1(\Omega) \subset \mathcal B_0 (\Omega)$.

We will frequently use the product norms:
\begin{align*}
	\| (\psi, V) \|_{\mathcal{B}_0\times \mathcal{B}_1} &= \|\psi \|_{\mathcal{B}_0} + \| V \|_{\mathcal{B}_1}\\
		\| (h,w)\|_{\mathcal{B}_2\times \mathcal{B}_2} &= \|h \|_{\mathcal{B}_2} + \| w \|_{\mathcal{B}_2} \\
		\| (f,X)\|_{\mathcal{B}_4\times \mathcal{B}_3}  &= \|f \|_{\mathcal{B}_4} + \| X \|_{\mathcal{B}_3}.
\end{align*}


\begin{remark} \label{rmk:weight-cts} We remark that in results we discuss in the ensuing sections, we will be working on $\overline{\Omega}$, a compact connected smooth manifold-with-boundary $\partial \Omega$ and manifold interior $\Omega$, e.g. an annulus $\Omega=A_1$.  A precompact connected open subset $\Omega'$ so that $\overline{\Omega'}$ is a smooth manifold-with-boundary, with manifold interior $\Omega'$, will be called a \emph{precompact smooth subdomain} of $\Omega$.  For some results, we will work with metrics $g$ in a neighborhood of a metric $g_0$ (e.g. metrics close to Euclidean $g_{\mathbb E}$ on $A_1$), and for several results $g$ will be a convex combination of two metrics $g_1$ and $g_2$.  For the H\"{o}lder spaces $C^{k,\alpha}(\overline{\Omega})$ that appear in these results, we can use any chosen metric to build the norm, e.g. $g_0$ or $g_1$, since any two such norms will be equivalent.  As for the weighted norms defining $\mathcal B_\ell(\Omega)$, e.g. $\mathcal B_\ell(A_1)$, we could work with norms built on the fixed metric $g_0$, respectively $g_1$, or when working at a metric $g$, we could work with weighted norms built with respect to $g$. We note that while the weighted norms for two different metrics are not necessarily equivalent when using exponential weights, relevant estimates below are suitably formulated so long as we use the same metric in defining the weighted norms on either side of the inequality.  As such, we just need to use a metric consistently in this way in construing the stated results.

We also note that an analysis of the proofs given shows that (by adjusting constants, and choosing the value of $N$ higher and $r_2$ smaller to absorb terms as needed) we could actually work with respect to a fixed weight function, e.g. $\rho_{g_0}$, cf. Remark \ref{rmk:cts-dep}.  While it seems natural to work with $\rho_g$ when working at a metric $g$, what we require from the weight function is that it behaves as noted in the estimates above in terms of a defining function for the boundary, such as $d_{g_0}$.   \end{remark}


\section{Localized deformation with the no-kernel condition} \label{section:localized}
In this section, we show how to locally deform an initial data set while controlling the dominant energy inequality.  To do this, we employ a modified constraint operator to handle the first order change in $|J|_g$ under the perturbation. 

The modified map
\[
	\Phi^W_{(g, \pi)}(\gamma, \tau)= \Phi(\gamma, \tau)+ (0, \tfrac{1}{2} \gamma \cdot_g (\textup{div}_g \pi +W))
\]
 differs from the  usual constraint map only by a term of lower order in derivatives and hence has similar analytic properties as the usual constraint map. In particular, we have the following local surjectivity theorem. The proof, which is deferred to Section~\ref{section:modified}, is a straightforward modification of the proof for the constraint map  in \cite{Corvino-Schoen:2006}.  In our proof, we obtain uniform estimates in a neighborhood of an arbitrary initial data set, which sharpen the estimates  in \cite{Corvino-Schoen:2006}, for the usual constraint map at a fixed initial data set or in a neighborhood of the flat data.  
 
For notational simplicity, we denote  $D\Phi^W_{(g,\pi)} = D\Phi^W_{(g,\pi)}|_{(g,\pi)}$ and its formal $L^2$ adjoint operator by $(D\Phi^W_{(g,\pi)})^*$. 

\begin{theorem} \label{thm:nl} 
Let $( g_0, \pi_0) \in C^{4,\alpha}(\overline{\Omega})\times C^{3,\alpha}(\overline{\Omega})$ be an initial data set, and let $W_0\in C^{2,\alpha}(\overline{\Omega})$ be a vector field.  Suppose that the kernel of $(D \Phi^{W_0}_{(g_0, \pi_0)})^*$ is trivial on $\Omega$. Then there is a neighborhood~$\mathcal U$ of $(g_0, \pi_0)$ in $C^{4,\alpha}(\overline{\Omega}) \times C^{3,\alpha}(\overline{\Omega})$, a  neighborhood $\mathcal W$ of $W_0$ in $C^{2,\alpha}(\overline{\Omega})$,  and constants $\epsilon>0$, $C>0$ such that for $(g, \pi)\in \mathcal U$, $W\in \mathcal W$, and for $(\psi,V)\in \mathcal B_0\times \mathcal B_1$ with $\|(\psi,V)\|_{\mathcal B_0\times \mathcal B_1} \leq \epsilon$, there is a pair of symmetric tensors $(h,w)\in \mathcal B_2\times \mathcal B_2$ with  $\|(h,w)\|_{\mathcal{B}_2\times \mathcal{B}_2} \leq C \| (\psi, V)\|_{\mathcal{B}_0\times \mathcal{B}_1}$ such that the initial data set $(g+h, \pi+w)$ satisfies
\begin{equation} \label{eq:sys-2psi}
	\Phi^W_{(g, \pi)} (g+h , \pi + w)  = \Phi^W_{(g, \pi)} (g, \pi )+ (2\psi, V).
\end{equation}
\end{theorem}

\begin{remark} \label{remark:Phi}
Rewriting the identity $\Phi^W_{(g, \pi)} (g+h , \pi + w)  = \Phi^W_{(g, \pi)} (g, \pi )+ (2\psi, V)$ in terms of the usual constraint map, we see that $(h, w)$ solves
\[
	\Phi(g+h, \pi+w) = \Phi(g, \pi) + (2\psi, V) - \left(0,  \tfrac{1}{2}h\cdot_g \left(\textup{div}_g \pi+W\right)\right).
\]
\end{remark}

The following computational lemma gives the motivation behind our definition of the modified map. 

\begin{lemma} \label{lemma:almost-DEC}
Let $( g, \pi) $ be an initial data set. Suppose the initial data set $(\bar{g}, \bar{\pi}) = (g+h, \pi+w) $ satisfies 
\[
	\Phi^V_{(g,\pi)}(g+h, \pi+w) = \Phi_{(g,\pi)}^V(g, \pi) + (2\psi, V).
\]
If $|h|_g\le3$, then 
\[
	\bar{\mu} - |\bar{J}|_{\bar{g}}\ge \mu + \psi - |J+V|_g,
\]	
where $(\bar{\mu}, \bar{J}) $ and $(\mu, J)$ are the mass and current densities of $(g, \pi)$ and $(\bar{g}, \bar{\pi})$, respectively.
\end{lemma}
\begin{proof}
Denote by $Y= J+V$. By Remark~\ref{remark:Phi}, 
\[
	\bar{\mu} = \mu + \psi\quad \mbox{and} \quad \bar{J} = Y-  \frac{1}{2}h\cdot_g Y.
\]
We compute the norm of $\bar{J}$ with respect to $\bar{g}$ below:
\begin{align*}
	|\bar{J}|_{\bar{g}}^2& = (g+h)_{ij} \left(Y^i- \frac{1}{2} ( h \cdot_g Y)^i\right)\left(Y^j - \frac{1}{2} ( h \cdot_g Y )^j\right)\\
	&= (g+h)_{ij}  \left(Y^i Y^j - Y^i ( h \cdot_g Y)^j + \frac{1}{4} ( h \cdot_g Y)^i( h \cdot_g Y)^j \right)\\
	&= |Y|_g^2 + h_{ij} Y^iY^j - g_{ij} Y^i( h \cdot_g Y)^j - h_{ij} Y^i (h\cdot_g Y)^j \\
	&\quad + \frac{1}{4} |h\cdot_g Y|_g^2 + \frac{1}{4} h_{ij}  ( h \cdot_g Y)^i( h \cdot_g Y)^j \\
	&= |Y|_g^2 -  \frac{3}{4} |h\cdot_g Y|_g^2 + \frac{1}{4} h_{ij}  ( h \cdot_g Y)^i( h \cdot_g Y)^j\\
	&\le |Y|_{g}^2.
\end{align*} 

\end{proof}

Theorem~\ref{theorem:main} directly follows from Theorem~\ref{thm:nl} (setting $W_0=0$) and Lemma~\ref{lemma:almost-DEC}. In particular, by choosing $V=0$ and $\psi>0$ on $\Omega$ with $\|\psi\|_{\mathcal B_2}$ sufficiently small, we have an immediate corollary.

\begin{corollary}
Let $(\overline{\Omega}, g_0, \pi_0)$ be as in Theorem~\ref{thm:nl}, with $W_0=0$. Suppose that $(g_0, \pi_0)$ satisfies the dominant energy condition. Then there exists $(h, w)\in \mathcal{B}_2(\Omega)\times \mathcal{B}_2(\Omega)$ such that the initial data set $(\bar{g}, \bar{\pi} ) = (g_0+h, \pi_0+w)$ satisfies the strict dominant energy condition $\bar{\mu} > |\bar{J}|_{\bar{g}}$  in~$\Omega$.  
\end{corollary}

\begin{definition} In the sections which follow, we will consider estimates where some parameter fields (initial data sets, auxiliary cutoff functions) may be allowed to vary.  In the case where we can choose a single constant so that the estimate holds for all nearby parameter fields (with respect to specified norms), we say that the constant \emph{depends locally uniformly} on the fields. \end{definition}

We apply Lemma~\ref{lemma:almost-DEC} to initial data sets which may not satisfy the dominant energy condition but have an error from interpolation. The following computational lemma suggests an appropriate choice of $(\psi, V)$ to absorb the error term. 

\begin{lemma}\label{lemma:interpolation-Phi-2}
Let $(g_1, \pi_1), (g_2, \pi_2)\in C^{2,\alpha}(\overline{\Omega})\times C^{2,\alpha}(\overline{\Omega})$ be initial data sets.  Let $0\le \chi \le 1$ be a smooth function such that $\chi(1-\chi)$ is supported on a compact subset of  $\Omega$. Denote by $(g,\pi) = \chi(g_1,\pi_1) + (1-\chi) (g_2, \pi_2)$. Let 
\[
	(2\psi, V) =-\Phi(g, \pi)+ \chi \Phi(g_1, \pi_1) + (1-\chi) \Phi(g_2, \pi_2) + (2\psi_0, 0)
\]
for some $\psi_0 \in \mathcal B_0$. Then 
\[
	\| (\psi,V)\|_{\mathcal{B}_0\times \mathcal B_1} \le  C\big(\| \chi (1-\chi)\|_{\mathcal{B}_1} +\|\nabla^2 \chi\|_{\mathcal B_0}+ \| \nabla \chi \|_{\mathcal B_1} \big) \| (g_1 - g_2, \pi_1 - \pi_2)\|_{C^{2,\alpha} \times C^{2,\alpha} } + \| \psi_0\|_{\mathcal B_0}
\]
where the constant $C$ depends locally uniformly on $(g_1, \pi_1), (g_2, \pi_2) \in C^{2,\alpha}(\overline{\Omega})\times C^{2,\alpha}(\overline{\Omega})$ and on $\chi \in C^{2,\alpha}(\overline{\Omega})$.

Furthermore, 
\begin{align*}
	|J+V|_g &= |\chi J_1 + (1-\chi)J_2|_g\\
	& \le \chi |J_1|_{g_1} + (1-\chi) |J_2|_{g_2} + \chi(1-\chi) (|g_1-g_2|_{g_1} |J_1|_{g_1} + | g_1-g_2|_{g_2} |J_2|_{g_2}),
\end{align*}
where $J, J_1, J_2$ are the current densities of $(g,\pi), (g_1,\pi_1), (g_2,\pi_2)$, respectively. 
\end{lemma}
\begin{proof}
The estimate of $(\psi, V)$ follows by Lemma~\ref{lemma:interpolation-Taylor}. The inequality for $J+V$ follows by direct computation:

\begin{align} \label{equation:interpolation-J}
\begin{split}
	|\chi J_1 + (1-\chi)J_2|_{g} &\le \chi |J_1|_{g} + (1-\chi) |J_2|_{g}\\
	&\le \chi |J_1|_{g_1} + (1-\chi) |J_2|_{g_2}\\
	&\quad + \chi  (1-\chi) |g_1 - g_2|_{g_1}  |J_1|_{g_1}  + (1-\chi) \chi |g_1 - g_2|_{g_2}  |J_2 |_{g_2}.
\end{split}
 \end{align}
\end{proof}

We now prove a statement which along with Remark~\ref{rmk:cpt-spt} establishes Theorem~\ref{theorem:interpolating-main} for $k=0$; the version of higher $k$ is handled later in Section~\ref{subsection:higher-regularity}. The statement we prove here can be expressed precisely as follows.

\begin{theorem} \label{theorem:interpolating}
Let $(g_0, \pi_0)\in C^{4,\alpha}(\overline{\Omega})\times C^{3,\alpha}(\overline{\Omega})$ be an initial data set.  Suppose that the kernel of $(D\Phi^0_{(g_0, \pi_0)})^*$ is trivial in $\Omega$. Let $0\le \chi \le 1$ be a smooth function such that $\chi(1-\chi)$ is supported on a compact subset of $\Omega$. Then there exists a neighborhood $\mathcal{U}$ of $(g_0, \pi_0)$ in $C^{4,\alpha}(\overline{\Omega})\times C^{3,\alpha}(\overline{\Omega})$ and $\epsilon >0$ such that for $(g_1, \pi_1), (g_2, \pi_2)\in \mathcal{U}$, if we set $(g,\pi) = \chi (g_1, \pi_1) + (1-\chi) (g_2, \pi_2)$,  there exists $(h, w)\in \mathcal B_2(\Omega)\times \mathcal B_2 (\Omega)$ with $\| (h,w)\|_{\mathcal B_2\times \mathcal B_2} \leq \epsilon$ such  that the initial data set $(\bar{g} , \bar{\pi})= (g+h, \pi + w)$ satisfies 
\[
	\bar{\mu} - |\bar{J}|_{\bar{g}} \ge \chi (\mu_1 - |J_1|_{g_1}) + (1-\chi) (\mu_2 - |J_2|_{g_2}).
\] 
\end{theorem}
\begin{proof}
Let $\mathcal U$, $\mathcal W$, $\epsilon$ and $C>1$ be as in Theorem~\ref{thm:nl} (with $W_0 = 0$).  We choose $(\psi, V)$ as follows (cf. Lemma~\ref{lemma:interpolation-Phi-2}),
\[
	(2\psi, V) =-\Phi(g, \pi)+ \chi \Phi(g_1, \pi_1) + (1-\chi) \Phi(g_2, \pi_2) + (2\psi_0, 0),
\]
where $\psi_0\in \mathcal B_0$ is a smooth function, positive on $\Omega$, with $\| \psi_0\|_{\mathcal B_0} \leq \frac{\epsilon}{2C}$. For  $\mathcal U$ sufficiently small, we have $V\in \mathcal W$ and $\|(\psi, V)\|_{\mathcal B_0 \times \mathcal B_1} \leq \frac{\epsilon}{C}$ by Lemma~\ref{lemma:interpolation-Phi-2}. 

Applying Theorem~\ref{thm:nl} gives $(h,w)$ that satisfies
\[
	\Phi_{(g,\pi)}^V(g+h, \pi+w) = \Phi_{(g,\pi)}^V(g, \pi) + (2\psi, V). 
\] 
The mass and current densities $(\bar{\mu}, \bar{J})$ of the deformed initial data set $(\bar{g}, \bar{\pi})= (g+h, \pi+w)$ satisfy, by Lemma~\ref{lemma:almost-DEC}, 
\[
	\bar{\mu}-|\bar{J}|_{\bar{g}}\ge \mu + \psi  - |J+V|_g = \chi \mu_1 + (1-\chi) \mu_2 +\psi_0 - | \chi J_1 + (1-\chi)J_2 |_g.
\]
Applying Lemma~\ref{lemma:interpolation-Phi-2} to estimate the last term yields 
\begin{align*}
	\bar{\mu}-|\bar{J}|_{\bar{g}}&\ge \chi (\mu_1- |J_1|_{g_1}) + (1-\chi)(\mu_2 - |J_2|_{g_2})\\
	&\quad + \psi_0 - \chi(1-\chi) (|g_1-g_2|_{g_1} |J_1|_{g_1} + | g_1-g_2|_{g_2} |J_2|_{g_2}).
\end{align*}
Because $\psi_0>0$ on $\Omega$ and $\chi(1-\chi)$ is supported on a compact subset of $\Omega$,  we have, by further shrinking $\mathcal U$ if necessary,
\[
	\psi_0 \ge  \chi(1-\chi) (|g_1-g_2|_{g_1} |J_1|_{g_1} + | g_1-g_2|_{g_2} |J_2|_{g_2}).
\]

\end{proof}
\begin{remark} 
It is clear from the proof that, assuming $(g_1, \pi_1)$ and $(g_2,\pi_2)$ satisfy the dominant energy condition, we can obtain the \emph{strict} dominant energy condition in $\Omega$ by choosing $\psi_0$ strictly greater than the error term. Also, the theorem still holds if $\chi(1-\chi)$ is supported in $\overline{\Omega}$ with an appropriate fall-off rate toward the boundary so that the error term can still be dominated by $\psi_0$. 
\end{remark}

\begin{remark} \label{rmk:cpt-spt} 
To conclude the proof of Theorem~\ref{theorem:interpolating-main} for $k=0$, we need only remark how we could arrange the support of $(h,w)$ to be contained in $\Omega$.  In fact from the above proof, we see it suffices to take $\psi_0\geq 0$, compactly supported inside $\Omega$, with $\psi_0>0$ on the support of $\chi(1-\chi)$.  A simple argument using Proposition \ref{proposition:kernel} shows that we can pull the domain in slightly to a precompact smooth subdomain $\Omega'\subset \Omega$ which contains the support of $\psi_0$, and on which $(D\Phi^0_{(g_0, \pi_0)})^*$ still has trivial kernel.  Applying the proof above of Theorem~\ref{theorem:interpolating} on $\Omega'$ allows us to solve for $(h,w)$ with compact support in $\Omega$ as desired.
\end{remark}

As another application, we provide a refined version of the density-type results with the dominant energy condition. A density-type theorem says that given an asymptotically flat initial data set $(M, g, \pi)$, there exists a sequence of initial data sets $(M, g_k, \pi_k)$ with the desired asymptotic properties, e.g. harmonic asymptotics (cf. \cite[Theorem 18]{Eichmair-Huang-Lee-Schoen:2016}), such that $(g_k, \pi_k)$ converges to $(g, \pi)$ in some appropriate topology. If $(g, \pi)$ satisfies the no-kernel condition, then the approximate sequence can be made to be identical to $(g, \pi)$ in a fixed compact set. The reason is that an initial data set interpolating between $(g,\pi)$ and $(g_k, \pi_k)$ would utimately satisfy the no-kernel condition in a fixed compact set for $k$ sufficiently large, so one can perform the localized deformation to reimpose the dominant energy condition.

\begin{corollary}
Let $(M, g, \pi)$ be a $C^{4,\alpha}_{\mathrm{loc}}\times C^{3,\alpha}_{\mathrm{loc}}$ asymptotically flat initial data set with the dominant energy condition. Let $B_R$ be a coordinate ball of radius $R$.  Suppose the kernel of $(D\Phi^0_{(g, \pi)})^*$ is trivial in $B_{R_2} \setminus \overline{B_{R_1}}$, $R_2>R_1$.  
For any sequence of asymptotically flat initial data sets $(g_k, \pi_k)$  with the dominant energy condition that converges to $(g, \pi)$ in $C_{\textup{loc}}^{4,\alpha}(M) \times C_{\textup{loc}}^{3,\alpha}(M)$, there exists $k_0 $ and a sequence $(\bar{g}_k, \bar{\pi}_k)\in C^{2,\alpha}_{\mathrm{loc}}(M)\times C^{2,\alpha}_{\mathrm{loc}}(M)$  with the dominant energy condition for $k\ge k_0$  that converges to $(g, \pi)$ in $C^{2,\alpha}_{\mathrm{loc}}(M)\times C^{2,\alpha}_{\mathrm{loc}}(M)$
and 
\begin{align*}
	(\bar{g}_k, \bar{\pi}_k) &= (g, \pi) \;\; \;\quad \mbox{ in } B_{R_1}\\
	(\bar{g}_k, \bar{\pi}_k) &= (g_k, \pi_k) \quad \mbox{ in } M\setminus B_{R_2}.
\end{align*}
\end{corollary}

\begin{remark}
Clearly from the proof we can reverse the roles of $(g, \pi)$ and $(g_k,\pi_k)$ and construct the converging sequence of initial data sets which is $(g,\pi)$ outside $B_{R_2}$ and is $(g_k, \pi_k)$ in $B_{R_1}$. 
\end{remark}

\section{Gluing in the asymptotically flat region} \label{N-body}

In this section, we prove gluing results with the dominant energy condition for initial data sets that are arbitrarily close to the flat data.  We focus on the three-dimensional case since our examples of admissible families are in three dimensions (see Section~\ref{section:admissible}), but the analysis presented here can be generalized to higher dimensions. Please refer to Appendix~\ref{section:asymptotically-flat} for the definitions of asymptotically flat initial data sets and the ADM integrals.

\subsection{Rescaling}
Given any pair of asymptotically flat initial data sets defined on the exterior of a ball in $\mathbb{R}^3$, they are both close to the flat data on $A_R$ for  $R$ large, by the asymptotic flatness.  Hence the error from interpolation between those two initial data sets is small enough, for $R$ sufficiently large, so that the localized deformation is applicable. Instead of working on $A_R$, it is convenient to perform the analysis over a fixed region $A_1$ via rescaling.  

\begin{notation} \label{notation:cut-off}
Let $B_R$ be an open ball of radius $R$ in $\mathbb{R}^3$. Denote by $A_R =  B_{2R} \setminus \overline{B_R}$ the open annulus. Let $\chi$ be a smooth cutoff function that is $\chi= 1$ on $B_1$ and $\chi= 0$ outside $B_2$ with $\chi(1-\chi)$ supported on a compact subset of $A_1$. Denote by $\chi_R(y) = \chi(y/R)$ the rescaled cutoff function. 
\end{notation}

We  make some general remarks on rescaling. Let $(g, \pi)$ be an initial data set on $\mathbb{R}^3\setminus B$. Here we write $\pi$ as a $(0,2)$ tensor with the indices lowered by $g$. Let $F_R: A_1 \to A_R$ be the diffeomorphism sending $x \mapsto y = Rx$. Define the rescaled initial data set on $A_1$, via pullback, as follows:
\[
	g^R = R^{-2} F_R^* g, \qquad \pi^R = R^{-1} F_R^* \pi.
\]
Noting $(F_R)_*\left(\partial_{x^i}\right) = R \partial_{y^i}$, in coordinates we have 
\begin{align*}
	g^R(\partial_{x^i}, \partial_{x^j}) (x)&=R^{-2} g( (F_R)_*(\partial_{x^i}), (F_R)_*(\partial_{x^j})) (y)= g(\partial_{y^i}, \partial_{y^j}) (y) \\
	\pi^R(\partial_{x^i}, \partial_{x^j}) (x)&= R^{-1} \pi ( (F_R)_*(\partial_{x^i}), (F_R)_*(\partial_{x^j})) (y)= R\pi(\partial_{y^i}, \partial_{y^j}) (y).
\end{align*}
The constraint operator is scaling invariant in the sense that
\[
	\Phi (g^R, \pi^R) = R^2 F_R^*\Phi(g, \pi) . 
\]
If  $(g, \pi)$ is asymptotically flat at the rate $(q, q_0)$ with respect to $y$, then the rescaled initial data sets  on $A_1$ satisfy
\begin{align*}
	|g^R_{ij} - \delta_{ij} |\le CR^{-q}, \qquad | \pi^R_{ij} | \le CR^{-q}, \qquad |\Phi(g^R, \pi^R)| \le C R^{-1-q_0}
\end{align*}
and  for a multi-index $I$ with $|I|\le k$,
\begin{align*}
	&|\partial^I_x g_{ij}^R (x)|= |R^{|I|}\partial^I_y g_{ij} (y)| \le CR^{-q}\\
	&|\partial^I_x \pi^R_{ij}(x)| =  |R^{1+|I|}\partial^I_y \pi_{ij} (y)| \le CR^{-q},
\end{align*}
where $C$ depends on $\| g-g_{\mathbb{E}}\|_{C^{k}_{-q}(A_R)},  \| \pi \|_{C^{k}_{-1-q}(A_R)}$. 

The following computational lemma says that the interpolation between $\mu$ and $J$ gives the interpolation between the dominant energy inequality, up to a controllable error term. The lemma  will be applied to initial data sets on $A_1$ that come from rescaling.

\begin{lemma} \label{lemma:estimates}
Let $(g_1, \pi_1), (g_2, \pi_2)\in C^{2,\alpha}(\overline{A_1})\times C^{2,\alpha}(\overline{A_1})$ be initial data sets on $\overline{A_1}$. Suppose $C_1>0$ is such that
\begin{align*}
	\| (g_1- g_2, \pi_1-\pi_2)\|_{C^{2,\alpha}(\overline{A_1})\times C^{2,\alpha}(\overline{A_1})} &\leq C_1 R^{-q}.
\end{align*}
 Let $(g, \pi) = \chi (g_1, \pi_1) + (1-\chi) (g_2, \pi_2)$, and let
 \begin{align*}
	(2\psi, V) = - \Phi( g,\pi)+ \chi \Phi(g_1, \pi_1) + (1-\chi) \Phi(g_2, \pi_2) + (2\psi_0R^{-1-q_0}, 0),
\end{align*}
for some $\psi_0\in \mathcal B_0(A_1)$. Then 
\begin{align}\label{equation:estimate-1}
	\| (\psi, V) \|_{\mathcal B_0 \times \mathcal B_1} \le CR^{-\min(q, 1+q_0)},
\end{align}
where $C$ depends on $C_1, \| \psi_0 \|_{\mathcal{B}_0}, \| \chi(1-\chi)\|_{\mathcal{B}_1}, \| \nabla^2 \chi \|_{\mathcal B_0}, \| \nabla \chi \|_{\mathcal B_1}$ and locally uniformly on $(g_1, \pi_1), (g_2, \pi_2) \in C^{2,\alpha}(\overline{A_1})\times C^{2,\alpha}(\overline{A_1})$ and $\chi \in C^{2,\alpha}(\overline{A_1})$. 

Furthermore, if 	$\|J_1\|_{C^0}+ \|J_2\|_{C^0} \le C_1 R^{-1-q_0}$, and if there exists $(h, w)$ so that the initial data set $(\bar g, \bar \pi)=(g+h, \pi + w)$ satisfies
 \[
 \Phi^V_{(g,\pi)}(g+h, \pi+w)  = \Phi^V_{(g,\pi)}( g,h) + (2\psi, V),
 \]
with $|h|_g\leq 3$, then 
\begin{align} \label{equation:estimate-2}
\begin{split}
	\bar{\mu} - |\bar{J}|_{\bar{g}} &\ge \chi( \mu_1 - | J_1|_{g_1} ) + (1-\chi) (\mu_2 - |J_2|_{g_2}) \\
	&\quad +  \left( \psi_0 - 2C_1^2\chi(1-\chi) R^{-q}\right)R^{-1-q_0},
\end{split}
\end{align}
where $(\bar{\mu}, \bar{J}), (\bar{\mu}_1, \bar{J}_1), (\bar{\mu}_2, \bar{J}_2)$ are the mass and current densities of $(\bar{g}, \bar{\pi}), (g_1,\pi_1), (g_2,\pi_2)$, respectively.

\end{lemma} 
\begin{proof}
By Lemma~\ref{lemma:interpolation-Taylor}, we have the following 
\begin{align*}
	 \| ( 2\psi, V) \|_{\mathcal B_0 \times \mathcal B_1} &=\| \chi \Phi(g_1, \pi_1) + (1-\chi) \Phi(g_2, \pi_2) -  \Phi(g, \pi)\| _{\mathcal B_0 \times \mathcal B_1} + \|2\psi_0 \|_{\mathcal{B}_0} R^{-1-q_0} \\
	& \le CR^{-\min(q, 1+q_0)}.
\end{align*}

By Lemma~\ref{lemma:interpolation-Phi-2}, we have
\begin{align*}	
	|J+V|_g &\le \chi |J_1|_{g_1} + (1-\chi) |J_2|_{g_2}\\
	&\quad + \chi (1-\chi) ( |g_1 - g_2|_{g_1}  |J_1|_{g_1}  +  |g_1 - g_2|_{g_2}  |J_2 |_{g_2}).
\end{align*}
Lemma \ref{lemma:almost-DEC} gives $\bar{\mu} - |\bar{J}|_{\bar{g}}\ge \mu + \psi - |J+V|_g$, so that inequality \eqref{equation:estimate-2} follows using 
\[
	\bar{\mu} = \chi \mu_1 + (1-\chi) \mu_2 + \psi_0 R^{-1-q_0}.
\]
\end{proof}

For $R$ large,  the rescaled initial data set $(g^R, \pi^R)$ on $A_1$ is close to the flat data  $(g_{\mathbb{E}}, 0)$. Hence the adjoint of the linearized operator at $(g^R, \pi^R)$ has an  \emph{approximate kernel} comprised of the ten-dimensional kernel $K$ of the flat data (see Example~\ref{example:linearization}), in the sense that elements of unit norm in $K$ are mapped by the adjoint operator at $(g^R, \pi^R)$ to elements of small norm.  Thus there is no uniform coercivity estimate for the adjoint operator at $(g^R, \pi^R)$ as $R$ tends to infinity. The following theorem, whose proof is deferred to  Section \ref{section:projected},  says that the modified constraint operator can be solved transverse to the approximate kernel.

Let $K$ be the kernel of $(D\Phi^{W_0}_{(g_0, \pi_0)})^*$, which is finite-dimensional by Proposition~\ref{proposition:kernel}.  Let $\mathcal{U}_0$ be a bounded neighborhood of a Riemannian metric $g_0$, as in Section~\ref{subsection:weighted}.  We fix a smooth bump function $\zeta$ supported in $\overline{\Omega_0} \subset \Omega$, where the precompact smooth subdomain $\Omega_0$ is chosen so that $\rho_g \equiv 1$ on $\Omega_0$ for all $g\in \mathcal{U}_0$.  Denote by $S_g$ the $L^2(d\mu_g)$-orthogonal complement of $\zeta K$. Let $\Pi_g: \mathcal{B}_0\times \mathcal B_1 \to (\mathcal{B}_0\times \mathcal B_1) \cap S_g$ be the $L^2(d\mu_g)$-orthogonal projection. 
\begin{theorem} \label{theorem:projected-modified}
Let $(g_0, \pi_0)\in C^{4, \alpha}(\overline{\Omega}) \times C^{3, \alpha}(\overline{\Omega})$ be an initial data set, and let $W_0\in C^{2,\alpha}(\overline{\Omega})$ be a vector field. There is a neighborhood $\mathcal U$ of $(g_0, \pi_0)$ in $C^{4, \alpha}(\overline{\Omega}) \times C^{3, \alpha}(\overline{\Omega})$, a neighborhood $\mathcal W$ of $W_0$ in $C^{2,\alpha}(\overline{\Omega})$,  and constants  $\epsilon>0$, $C>0$ such that for $(g, \pi)\in \mathcal U$, $W\in \mathcal W$, and $(\psi,V)\in \mathcal B_0(\Omega)\times \mathcal{B}_1(\Omega)$ with $\|(\psi,V)\|_{\mathcal B_0\times \mathcal{B}_1} \leq \epsilon$, there is a pair of symmetric tensors $(h,w)\in \mathcal B_2(\Omega)\times \mathcal B_2(\Omega)$ such that the initial data set $(g+h, \pi+w)$ satisfies
\[
	\Pi_{g_0}\circ \Phi^W_{(g,\pi)} (g+h , \pi + w)  = \Pi_{g_0}\circ \Phi^W_{(g,\pi)} (g, \pi )+ \Pi_{g_0}(2\psi, V)
\]
with  
\[
\|(h,w)\|_{\mathcal B_2\times \mathcal B_2} \leq C \| \Pi_{g_0}(2\psi, V)\|_{\mathcal B_0\times \mathcal{B}_1}.
\] 
\end{theorem}

We now apply this in the setting we study here, interpolating $\mu$ and $J$. 

\begin{proposition} \label{proposition:nl2} 
Let $(g_1, \pi_1)$ and $(g_2, \pi_2)\in  C_{\mathrm{loc}}^{4,\alpha}\times C^{3,\alpha}_{\mathrm{loc}}$ be asymptotically flat initial data sets at the rate $(q, q_0)$ on $\mathbb{R}^3 \setminus B$.  Consider the corresponding rescaled initial data sets $(g_1^R, \pi_1^R), (g_2^R, \pi_2^R)$ on~$A_1$. Define the initial data set 
\[
	(\gamma^R, \tau^R) = \chi (g_1^R, \pi_1^R) + (1-\chi) (g_2^R, \pi_2^R).
\]
Let 
\[
	(2\psi^R, V^R) =-  \Phi(\gamma^R, \tau^R) + \chi \Phi(g_1^R, \pi_1^R) + (1-\chi) \Phi(g_2^R, \pi_2^R) + (2\psi_0 R^{-1-q_0}, 0 )
\]
for some $\psi_0\in \mathcal B_0(A_1)$. There is  $R_0>0$ and $C>0$, depending only on $g_{\mathbb{E}}$, $\chi$, $\| \psi_0\|_{\mathcal B_0}$, $\| (g_1-g_{\mathbb{E}}, \pi_1 ) \|_{C^{2,\alpha}_{-q} \times C^{2,\alpha}_{-1-q}}$, $\| (g_2-g_{\mathbb{E}}, \pi_2 ) \|_{C^{2,\alpha}_{-q} \times C^{2,\alpha}_{-1-q}}$,  such that for each $R\geq R_0$, there exists a pair of symmetric tensors $(h^R,w^R)\in \mathcal B_2(A_1)\times \mathcal B_2(A_1)$ such that the initial data set $(\gamma^R+h^R, \tau^R+w^R)$ satisfies 
\[
	\Pi_{g_{\mathbb{E}}}\circ\Phi^{V^R}_{(g_{\mathbb{E}}, 0)} (\gamma^R+h^R , \tau^R+ w^R)  = \Pi_{g_{\mathbb{E}}}\circ\Phi^{V^R}_{(g_{\mathbb{E}}, 0)} (\gamma^R, \tau^R)+ \Pi_{g_{\mathbb{E}}} (2\psi^R, V^R)
\]
with  the estimate
\begin{align} \label{equation:fall-off-rate}
	\|(h^R,w^R)\|_{\mathcal{B}_2\times \mathcal{B}_2} \le CR^{-\min(q, 1+q_0)}.
\end{align}
\end{proposition}
\begin{proof}
Apply Theorem~\ref{theorem:projected-modified} with $(g_0, \pi_0) = (g_{\mathbb{E}}, 0)$ and $W_0=0$. By Lemma~\ref{lemma:estimates} and Lemma~\ref{lemma:projection}, for $R$ sufficiently large we have $V^R\in \mathcal W$ and
\begin{align*}
	 \| \Pi_{g_{\mathbb{E}}} ( 2\psi^R, V^R ) \|_{\mathcal B_0 \times \mathcal B_1}\leq CR^{-\min(q, 1+q_0)}< \epsilon.
\end{align*}
 Therefore we can apply Theorem~\ref{theorem:projected-modified} to solve for $(h^R, w^R)$ with the estimate~\eqref{equation:fall-off-rate}. 
\end{proof}

\subsection{An admissible family and gluing} \label{section:admissible}
As observed in \cite{Corvino-Schoen:2006}, the $10$-dimensional approximate kernel in the asymptotically gluing is ultimately connected to the $10$-dimensional parameter space of an \emph{admissible family}. An improved rescaling argument to handle a more general situation is discussed in \cite{Chrusciel-Corvino-Isenberg:2011}. However, since neither paper states in sufficiently explicit terms the requisite conditions on an admissible family that can be used to model the asymptotics of initial data sets by asymptotic gluing, we proceed to do so now.

In order to handle  initial data sets whose center of mass and angular momentum integrals may not converge, we define the center of mass and angular momentum integrals of $(g, \pi)$ at the finite radius $R$ as
\[
	\mathcal C^R_i = \frac{1}{16\pi} B^R_{(g,\pi)} (x^i, 0), \qquad \mathcal J^R_k = \frac{1}{8\pi} B^R_{(g, \pi)} (0, x\times \frac{\partial}{\partial x^k}).
\] 
It is clear that $\mathcal C^R, \mathcal{J}^R$ are continuous in $R$. If  $(g, \pi)$ is asymptotically flat at the rate $q=q_0=1$, then $(\mathcal C^R, \mathcal{J}^R)=O(\log R)$ as $R$ tends to infinity. For other values of  $q, q_0$, $(\mathcal C^R, \mathcal{J}^R)=O(R^{1-\min\{1, q_0, 2q-1\}})$. Note that if $(g, \pi)$ satisfies the Regge-Teitelboim conditions, then $(\mathcal C^R, \mathcal J^R)$ converges to a pair of vectors (see, e.g. \cite{Corvino-Schoen:2006, Corvino-Wu:2008, Huang:2009, Huang:2010}).  We denote $a =\min\{1, q_0, 2q-1\}\in (0, 1]$.

\begin{definition}\label{definition:admissible}
Let $(g, \pi)\in C_{\mathrm{loc}}^{k+4,\alpha}\times C^{k+3,\alpha}_{\mathrm{loc}}$ be an asymptotically flat initial data set on $\mathbb{R}^3\setminus B_{R_0}$, at the rate $( q, q_0)$. 
Let $\mathcal{S} =\{ (g^{\theta}, \pi^{\theta}) \}_{\theta\in\Theta}$ be a family of $C^{k+4,\alpha}_{\textup{loc}}\times C^{k+3,\alpha}_{\textup{loc}}$ asymptotically flat initial data sets defined on $\mathbb{R}^3 \setminus B_{R_0}$, where the components of $\theta = (E^\theta, P^\theta, \mathcal C^\theta, \mathcal{J}^\theta)$ are the ADM energy, linear momentum, center of mass, and angular momentum of $(g^{\theta}, \pi^{\theta})$, and let $(2\mu^{\theta}, J^{\theta})=\Phi(g^{\theta}, \pi^{\theta})$.   The family $\mathcal{S}$ is said to be an \emph{admissible family for $(g, \pi)$} if the following properties hold, with respect to a fixed asymptotically flat coordinate chart:
\begin{enumerate}
\item The map $ (g^{\theta}, \pi^{\theta}) \mapsto \theta \in \Theta=\Theta_1 \times \Theta_2$ is a homeomorphism where  $\Theta_1 \subset \mathbb{R}^4$,  $\Theta_2 \subset \mathbb{R}^{6}$ are open sets such that $\Theta_1$  contains $(E, P)$ of $(g, \pi)$ and $\Theta_2 = \cup_{R\ge R_0} \Theta^R_2$ where $\Theta_2^R$ is the ball centered at $(\mathcal C^R, \mathcal J^R)$ of radius $ R^{1-a}\log R$.\label{definition:admissible1}

\item  $(g^{\theta}, \pi^{\theta})$ satisfies the following uniformity conditions:  there is a constant $\kappa>0$ such that for all $R\geq R_0$ and $\theta \in \Theta_1\times \Theta_2^R$,
\begin{align}\label{equation:fall-off}
\begin{split}
	\| (g^{\theta} - g_{\mathbb{E}}, \pi^{\theta} )\|_{C^2_{-q} ( \mathbb{R}^3\setminus B_R)\times C^1_{-1-q} ( \mathbb{R}^3\setminus B_R)}&\le \kappa\\
	\| (\mu^\theta, J^\theta)  \|_{C^0_{-3-q_0} ( \mathbb{R}^3\setminus B_R)} &\le \kappa
\end{split}
\end{align}
and 
\begin{align}\label{equation:fall-off-RT}
\begin{split}
	|B^R_{(g^\theta, \pi^\theta)}(x^k, 0) - 16 \pi \mathcal C^\theta_k |&\le \kappa |\theta|^2R^{-1}\\
	|B^R_{(g^\theta, \pi^\theta)}(0, x\times \frac{\partial}{\partial x^k}) - 8 \pi \mathcal J^\theta_k | &\le \kappa |\theta|^{2} R^{-1}.
\end{split}
\end{align}
\end{enumerate}
\end{definition}

\begin{remark}
The definition of the parameter set $\Theta_1\times \Theta_2^R$ is set up in a way so that $\theta$ can be large enough to account for the error terms from $(\mathcal C^R, \mathcal J^R)$ and from the right hand side of \eqref{equation:fall-off-RT}, which may compete with largeness of $\theta$. This subtle balance shows up in Lemma~\ref{lemma:degree} below in the degree argument. We remark that in Definition~\ref{definition:admissible}, the term $|\theta|^2$ appearing in \eqref{equation:fall-off-RT} can be replaced by $|\theta|^{\kappa_2}$ for $\kappa_2>0$. If $q=q_0=1$, the same parameter set $\Theta_1\times \Theta_2^R$ is still valid. For other values of $q, q_0$, the radius of $\Theta_2^R$ needs to be modified accordingly, depending on $\kappa_2$. This can be done by tracking the exponents in the proof of Lemma~\ref{lemma:degree}.
\end{remark}

\begin{example} Let $(g, \pi)$ be an asymptotically flat initial data set with $E>|P|$. Let $\Theta=\Theta_1 \times \mathbb R^6$, where $(E, P)\in \Theta_1$ is a precompact open subset of $ \{ (a, b) \in \mathbb{R}\times \mathbb{R}^3: a > |b|\}$. There is an admissible \emph{vacuum} family  $\mathcal{S}_{\textup{Kerr}}=\{(g^{\theta}, \pi^{\theta})\}_{\theta\in \Theta}$ for $(g,\pi)$ obtained from the family of Kerr spacetimes. Condition \eqref{definition:admissible1} is shown in \cite[Appendix F]{Chrusciel-Delay:2003}. Inequalities \eqref{equation:fall-off} in Condition (2) follow from the asymptotic expansions of $(g^\theta, \pi^\theta)$ and precompactness of $\Theta_1$. Inequalities \eqref{equation:fall-off-RT} follows from a standard argument:  by the divergence theorem, $B^R_{(g^\theta, \pi^\theta)}(x^k, 0) - 16 \pi \mathcal C^\theta_k  = \int_{M\setminus B_R} x^k \sum_{i,j} (g^\theta_{ij,ij} - g^\theta_{ii,jj}) \, dx$. The estimate then follows by using the constraint equations to rewrite the integrand and estimate the resulting terms in an obvious way. The inequality for the angular momentum can be obtained in the same fashion. 

The same argument would also allow us to find an admissible \emph{non-vacuum} family $\mathcal{S}_{\textup{Kerr-Newman}} = \{(g^{\theta}, \pi^{\theta})\}_{\theta\in \Theta}$ for $(g,\pi)$ obtained from the Kerr-Newman spacetime for each fixed pair of electric and magnetic charges.
\qed
\end{example}

Admissible families will be used in an asymptotic gluing construction in Theorem~\ref{theorem:gluing2} below, and we want to highlight how the assumptions in Definition~\ref{definition:admissible}  will appear.  The condition on the parameter set $\Theta_2^R$ is to handle the scaling in the center of mass and angular momentum components of a map $\mathcal I^R$, in order to be able to apply a degree argument as used in the following technical lemma.   The uniformity assumptions  \eqref{equation:fall-off} and \eqref{equation:fall-off-RT}  will be used in the proof of the theorem to establish the desired estimate of the error term $\mathcal I_0^R$, as appears in the following lemma.

\begin{lemma}\label{lemma:degree}
 Let $\Theta_1\times \Theta^R_2$ be as in Definition~\ref{definition:admissible}. Let $\mathcal I^R : \Theta_1\times \Theta^R_2 \to \mathbb{R}^{10}$ be a family of continuous maps, parametrized by $R \in [ R_0, \infty)$. Suppose that there is a constant $C>0$ such that for any $R\ge R_0$ and $\theta = (E^\theta, P^\theta, \mathcal C^\theta, \mathcal J^\theta)\in \Theta_1\times \Theta^R_2$,
\begin{align*}
	\mathcal{I}^{R}(\theta) =( E^\theta - E, P^\theta - P, R^{-1} (\mathcal C^\theta - \mathcal C^R), R^{-1}(\mathcal J^\theta- \mathcal J^R)) +  \mathcal{I}_0^R(\theta),
\end{align*}
where
\begin{align*}
	|\mathcal{C}^R|&\le C\log R\\
	|\mathcal J^R |&\le  C\log R\\
	|\mathcal{I}_0^R (\theta)|&\le C( R^{-a} (\log R)^{\frac{1}{2}}+   |\theta|^{2}R^{-2}).
\end{align*}
Then for each $R$ sufficiently large, there exists $\theta \in \Theta_1\times \Theta^R_2$ such that $\mathcal{I}^R(\theta)=0$. 
\end{lemma}
\begin{proof}
Let $\theta_0^R= (E, P, \mathcal{C}^R, \mathcal J^R)$.  We first translate and rescale $\Theta_1\times \Theta^R_2$, in order to employ a degree argument over  sets centered around the same point.  Suppose $\Theta_1$ contains the ball centered at $ (E, P)$ of radius $\epsilon$ for some $\epsilon>0$. Let $T^R:B_\epsilon(0) \times B_{R^{-a}\log R}(0) \to \Theta_1 \times \Theta^R_2$  be given by 
\[
	T^R(v, w) = (v, Rw) + \theta_0^R. 
\]
The composition map satisfies
\[
	\mathcal I^R\circ T^R (v,w) = (v,w) + \mathcal{I}_0^R\circ T^R(v,w).
\]
Hence, for $R$ sufficiently large, we have
\[
	|\mathcal I^R\circ T^R (v,w) - (v,w) |  \le C(R^{-a} (\log R)^{\frac{1}{2}}+  |T^R(v,w)|^{2}R^{-2})<R^{-a} \log R,
\]
where we use that $\theta= T^R(v,w) \in \Theta_1\times \Theta_2^R$ and hence $|\theta|\le C |R^{1-a}\log R|$. The rest of proof follows from a standard degree argument (see, e.g. \cite[Lemma 5.2]{Huang-Schoen-Wang:2011}). 

\end{proof}

We now restate Theorem~\ref{theorem:gluing} for $k=0$ in  a more specific form and give a proof.  The version for higher $k$ is discussed in Section~\ref{subsection:higher-regularity}.

\begin{theorem} \label{theorem:gluing2}
Let $( g, \pi)\in C^{4,\alpha}_{\mathrm{loc}}\times C^{3,\alpha}_{\mathrm{loc}}$ be an asymptotically flat initial data set at the rate $(q, q_0)$ on $\mathbb{R}^3 \setminus B$ with the ADM energy-momentum $(E, P)$. Let $\epsilon>0$. There is a  constant $R_0>0$ such that for $R\geq R_0$, there is an initial data set $(\bar{g}, \bar{\pi})\in C^{2,\alpha}_{\mathrm{loc}}\times C^{2,\alpha}_{\mathrm{loc}}$ such that
\begin{align*}
	(\bar{g}, \bar{\pi}) &= (g, \pi) \quad \mbox{in }  B_R\\
	(\bar{g}, \bar{\pi})&=(g^{\theta}, \pi^{\theta}) \quad \mbox{in } M\setminus B_{2R}
\end{align*}
for some $(g^\theta, \pi^\theta)$ in the admissible family for $(g, \pi)$, and $(\bar{g}, \bar{\pi})$ satisfies the inequality
\[
	\bar{\mu} - |\bar{J}|_{\bar{g}} \ge \chi_R(\mu - |J|_g) + (1-\chi_R) (\mu^{\theta} - |J^{\theta}|_{g^{\theta}})
\]
with strictly larger ADM energy $E^\theta >E$ and  
\[
	|P^\theta - P|<  E^\theta- E< \epsilon.
\]
\end{theorem}
\begin{proof}  
We prove the case  $q=q_0=1$, as the proof for other values of $q, q_0$ is similar.  Given $R\geq 1$ sufficiently large and $\theta\in \Theta$, we apply Proposition~\ref{proposition:nl2} with $(g_1, \pi_1) = (g, \pi)$, $(g_2, \pi_2) = (g^\theta, \pi^\theta)$ and $(\gamma^R, \tau^R) = \chi(g^R, \pi^R) + (1-\chi) ((g^\theta)^R, (\pi^\theta)^R)$. Let $\psi_0\in \mathcal B_0(A_1)$ be a fixed  positive function in $A_1$. Define $(\psi^R, V^R)$, as in~Proposition ~\ref{proposition:nl2}, by
\begin{align*}
	(2\psi^R, V^R) &=-  \Phi (\gamma^R, \tau^R) + \chi \Phi(g^R, \pi^R) + (1-\chi) \Phi((g^\theta)^R, (\pi^\theta)^R) + (2\psi_0 R^{-2}, 0 ).
\end{align*}
There exists $(h^R, w^R)\in \mathcal B_2(A_1)\times \mathcal B_2(A_1)$ (depending on $\theta$) such that $(\bar{g}^R, \bar{\pi}^R)  =( \gamma^R + h^R , \tau^R+w^R)$ solves the following projected problem, for some radial bump function $\zeta$ supported in $A_1$,
\[
	\Phi^{V^R}_{(\gamma^R, \tau^R)}  (\bar{g}^R , \bar{\pi}^R)  -\Phi^{V^R}_{(\gamma^R, \tau^R)}  (\gamma^R, \tau^R)- (2\psi^R, V^R) \in \zeta K,
\]
where $K$ is the kernel at the flat data (see Example~\ref{example:linearization}). Fix a constant $\lambda >0$.   We show that, for each $R$ sufficiently large, there exists $\theta\in \Theta_1\times \Theta_2^R$ such that 
\[
	\mathcal{E}^R(\theta) :=\Phi^{V^R}_{(\gamma^R, \tau^R)}  (\bar{g}^R , \bar{\pi}^R)  -\Phi^{V^R}_{(\gamma^R, \tau^R)}  (\gamma^R, \tau^R)- (2\psi^R, V^R) - (\lambda \zeta R^{-2} (\log R)^\frac{1}{2}, 0)= 0.
\]
We emphasize that $(h^R, w^R)$ has been generated independently from $\lambda$.  We remark that the additional term  $(\lambda \zeta  R^{-2} (\log R)^\frac{1}{2}, 0)$ from the kernel will help  to bring up the ADM energy and also keep the dominant energy condition.  

It suffices to show that $\mathcal{E}^{R}(\theta)$ is $L^2(dx)$-orthogonal to $K$, because $K^{\perp}$ is transverse to $\zeta K$.  Consider the $L^2(dx)$-projection of $\mathcal{E}^{R}(\theta)$ to $K$, which via a basis of $K$ is given by a map $\mathcal{I}^{R} (\theta) : \Theta_1\times \Theta_2^R \to \mathbb{R}^{10} $ that sends $\theta$ to  $(e, \bf p, c, j)$, where ${\bf p}= (p_1, p_2, p_3)$, ${\bf c} = (c_1, c_2, c_3)$,  ${\bf j} = (j_1, j_2, j_3)$:
\begin{align*}
	e &=\frac{R}{16\pi } \int_{A_1}\mathcal{E}^{R}(\theta)\cdot (1, 0) \, dx \\
	p_i &=\frac{R}{8\pi } \int_{A_1} \mathcal{E}^{R}(\theta) \cdot (0, \frac{\partial}{\partial x^i})\, dx \\
	c_k &=\frac{R}{ 16\pi } \int_{A_1} \mathcal{E}^{R} (\theta) \cdot (x^k, 0)\, dx \\
	j_\ell & =\frac{R}{8 \pi } \int_{A_1} \mathcal{E}^{R} (\theta) \cdot (0, x\times \frac{\partial}{\partial x^\ell})\, dx.
\end{align*}
The map $\mathcal{I}^{R}$ is continuous because $(h^R, w^R)$ depends continuously on $(\psi^R, V^R)$ and $(\gamma^R  , \tau^R)$, which are continuous in $\theta$ by definition, cf. Remark \ref{rmk:cts-dep}.  

Expressing $\mathcal{E}^{R}( \theta) $ in terms of the usual constraint map, we have 
\begin{align*}
	\mathcal{E}^{R}(\theta) &=\Phi (\bar{g}^R , \bar{\pi}^R)+  \mathcal{E}^R_2(\theta) -(\lambda \zeta R^{-2}(\log R)^\frac{1}{2}, 0),
\end{align*}
where 	
\[
	\mathcal{E}^R_2(\theta)= -\chi \Phi(g^R, \pi^R) - (1-\chi) \Phi((g^\theta)^R, (\pi^\theta)^R)- (2\psi_0 R^{-2}, \tfrac{1}{2} h^R\cdot_{\gamma^R} (\textup{div}_{\gamma^R} \tau^R +V^R)).
\]
The $L^2$~projection of $\Phi (\bar{g}^R , \bar{\pi}^R)$ is handled exactly as in the vacuum case; estimates of the projection are included in Lemma~\ref{lemma:L^2-projection} for which we employ the uniformity conditions, in particular \eqref{equation:fall-off-RT}.  The $L^2$ projection of $\mathcal{E}^R_2(\theta)$ is of lower order because by using \eqref{equation:fall-off} and the estimates for $(h^R, w^R)$ in Proposition~\ref{proposition:nl2}, we obtain $\left| \mathcal{E}^R_2(\theta)\right|_{g_\mathbb{E}}  \le C R^{-2}$. Because $\zeta$ is radial, the $L^2$ projection of the last term $-(\lambda \zeta R^{-2}(\log R)^\frac{1}{2}, 0)$ is non-zero only onto the kernel element $(1,0)$, and we find  
\[
	\int_{A_1}  \lambda \zeta R^{-2}(\log R)^\frac{1}{2} \, dx =16\pi \tilde{\lambda} R^{-2}(\log R)^\frac{1}{2}> 0,
\]
where $\tilde{\lambda}: =(16\pi)^{-1} \int_{A_1}  \lambda \zeta \, dx$. We then obtain, together with Lemma~\ref{lemma:L^2-projection},  for $\theta \in \Theta_1 \times \Theta^R_2$,
\begin{align*}
	\mathcal{I}^{R}(\theta) =( E^\theta - E, P^\theta - P, R^{-1} (\mathcal C^\theta - \mathcal C^R), R^{-1}(\mathcal J^\theta- \mathcal J^R)) +  \mathcal{I}_1^R(\theta)- (\tilde{\lambda}R^{-1}(\log R)^\frac{1}{2},0)
\end{align*}
where 
\[
	|\mathcal{I}_1^R (\theta)|\le C(R^{-1}+  |\theta|^{2} R^{-2}).
\]
By Lemma~\ref{lemma:degree} with $\mathcal I_0^R(\theta) =\mathcal{I}_1^R (\theta) - (\tilde{\lambda}R^{-1}(\log R)^\frac{1}{2},0)$, there is  $\theta \in \Theta_1 \times \Theta^R_2$ such that $\mathcal{I}^R(\theta)=0$ for $R$ sufficiently large. Because  the term $\tilde{\lambda}R^{-1}(\log R)^\frac{1}{2}$ dominates other error terms in the identity $\mathcal{I}^R(\theta)=0$ with a favorable sign, we obtain $E<E^\theta$ and $|P^\theta - P| < E^\theta - E<\epsilon$ for $R$ large.

Last, we verify the dominant energy condition for $(\bar{g}^R, \bar{\pi}^R)$ on $A_1$.  We have solved
\[
	\Phi^{V^R}_{(\gamma^R, \tau^R)}  (\bar{g}^R , \bar{\pi}^R)  =\Phi^{V^R}_{(\gamma^R, \tau^R)}  (\gamma^R, \tau^R)+ (2\psi^R  +  \lambda \zeta R^{-2}(\log R)^\frac{1}{2}, V^R).
\]
Because $\lambda>0$, and by the uniformity estimates and rescaling, we see that by Lemma~\ref{lemma:estimates}, there is a $C_1>0$ for which it suffices to show that on $A_1$
\begin{align*}
	\psi_0 \ge 2C_1^2 \chi(1-\chi) R^{-1} .
\end{align*}
Since $\psi_0 $ is positive in $A_1$ and $\chi(1-\chi)$ is supported on a compact subset of $A_1$, the above inequality holds for $R$ sufficiently large.

\end{proof}


\begin{remark} As an application of Theorem~\ref{theorem:gluing2}, we consider the problem of combining some number $N$ of asymptotically flat initial data sets into a single asymptotically flat initial data set, so that the construction preserves the dominant energy condition.  So let $(M_k, g^k, \pi^k)$, $k=1,\ldots, N$, be three-dimensional asymptotically flat initial data sets that satisfy the dominant energy condition and have time-like ADM energy-momentum vector, and let $U_k \subset M_k$ chosen open subsets so that $M_k\setminus U_k$ is a single asymptotically flat end.   A natural question is whether there exists an initial data set $(M, g, \pi)$ satisfying the dominant energy condition and which contains an open set $U$ so that the restriction $(U, g, \pi)$ is given by the disjoint union $\bigcup\limits_{k=1}^N (U_k, g^k,\pi^k)$, so that $(M\setminus U, g, \pi)$ has one asymptotically flat end.   Such a construction in the case the data sets are vacuum near infinity appears in \cite{Chrusciel-Corvino-Isenberg:2011}.   Since we can use the Kerr initial data as the admissible family in Theorem~\ref{theorem:gluing2} above, we can indeed achieve the construction just posed by using the multi-Kerr template of \cite{Chrusciel-Corvino-Isenberg:2011}, cf. \cite{Chrusciel-Delay:2003}.  By the local nature of the construction in the proof of Theorem~\ref{theorem:gluing2}, we can modify the initial data sets $(M_k, g^k, \pi^k)$ to be Kerr near infinity, preserving the data $(U_k, g^k, \pi^k)$, and then paste them in to an appropriate multi-Kerr vacuum initial data set.  \end{remark}



\section{Deforming the modified constraint map} \label{section:modified}

In this section, we prove Theorem~\ref{thm:nl}. The argument is similar to that for the constraint map {\cite[Section 4.2]{Corvino-Schoen:2006}}. We first solve the linearized equation via a variational  approach with estimates and then use iteration for the nonlinear problem. We also pay special care to ensure uniform estimates. Throughout the section, the weighted Sobolev and H\"{o}lder norms are all taken on $\Omega$.

\subsection{The linearized equation} \label{subsection:linearized}

The goal is to solve the following linearized system for given $(\psi, V)$
\[	
	D\Phi^W_{(g, \pi)}|_{(g,\pi)}(h,w) = (\psi, V).
\]  
We recall that to simplify notation, we let $D\Phi^W_{(g, \pi)}= D\Phi^W_{(g, \pi)}|_{(g,\pi)}$, the linearization at $(g,\pi)$. Consider the functional $\mathcal{G}$ defined as follows (where $(\psi, V) \cdot_g (f, X)= \psi f + g(V, X)$):
\[
	\mathcal G(f,X)= \int_{\Omega} \left( \frac12 \rho_g \left|(D \Phi^W_{(g,\pi)})^* (f, X) \right|_g^2 - (\psi, V) \cdot_g (f, X) \right) \; d\mu_g.
\]
Clearly the functional is convex. To derive the key coercivity property, we need some basic estimates. Recall $L_g^* f =  - (\Delta_g f) g +\textup{Hess}_g f - f \, \textup{Ric}(g)$.  

\begin{lemma} \label{lemma:f}
Let $g_0\in C^2(\overline{\Omega})$. There is a $C^2(\overline{\Omega})$ neighborhood $\mathcal{U}_0$ of $g_0$ and a constant $C>0$ such that for $g\in \mathcal{U}_0$ and for $(f, X)\in H^2_{\rho_g}(\Omega)\times H^1_{\rho_g}(\Omega)$, 
\begin{align}
		\| f \|_{H^2_{\rho_g}} &\le  C  \left( \| L_g^* f \|_{L^2_{\rho_g}} + \| f \|_{L^2_{\rho_g}}\right)  \label{equation:H^2}\\
		\| X\|_{H^1_{\rho_g}} &\le C  \left( \|\mathcal{D}_gX \|_{L^2_{\rho_g}} + \| X\|_{L^2_{\rho_g}}\right). \label{equation:H1}
\end{align}
\end{lemma}
\begin{proof} 
The uniform dependence of the constant $C$ on the metric in \eqref{equation:H^2} follows from the proof of \cite[Proposition 3.1-3.2, Theorem 3]{Corvino:2000}; in particular, one uses \cite[Equation (13)]{Corvino:2000} and the co-area formula as in the end of the proof of \cite[Theorem 3]{Corvino:2000}, cf. \cite[Proposition 3.1 and Remark 3.6]{Corvino-Eichmair-Miao:2013}.

For \eqref{equation:H1}, it suffices to prove that there is a uniform constant $C$ such that
\[
	\int_{\Omega} |\nabla_g X|^2   \rho_g \, d\mu_g \le C\left( \int_{\Omega} |\mathcal{D}_g X|^2 \rho_g \, d\mu_g + \int_{\Omega} |X|^2 \rho_g \, d\mu_g\right).
\]
The estimate for a fixed metric is obtained in \cite[Proof of Lemma 4.1]{Corvino-Schoen:2006}. In their proof, an analysis of the arguments in \cite[pp.~201-202 and the first paragraph of p.~203]{Corvino-Schoen:2006} shows that there is a neighborhood $\mathcal{U}_0$ so that for $X \in H^1_{\rho_g}(\Omega)$
\begin{align} \label{equation:X}
\begin{split}
	\int_{\Omega}& |\nabla_g X|^2   \rho_g \, d\mu_g \\
	&\le C_0\left( N^2\int_{\Omega} |\mathcal{D}_g X|^2 \rho_g \, d\mu_g + \int_{\Omega} |X|^2 \rho_g \, d\mu_g + \int_{\Omega} |X|^2 d_g^{-4} \rho_g \, d\mu_g \right),
\end{split}
\end{align}
where $C_0$ is independent of $g \in \mathcal{U}_0$, and $N$.  It is important that the coefficient of the last integral above does not depend on $N$. (Note that we use the exponential weight function $\rho_g$, instead of a power weight function, so the power of the distance function in the last term is different from \cite[p.~203]{Corvino-Schoen:2006}.) The last term in the right hand side was handled by an indirect argument in \cite{Corvino-Schoen:2006}.  Here we apply \eqref{inequality:boundary-estimate} (with $j=k=1$) in the proof of Corollary~\ref{corollary:weight} to the last term and derive
\[
	\int_{\Omega} |X|^2 d_g^{-4} \rho_g d\mu_g \le \frac{4}{N} \int_{\Omega} |\nabla_g X|^2 \rho_g\, d\mu_g +C \int_{\Omega} |X|^2 \rho_g \, d\mu_g .
\]
The first term on the right hand side can be absorbed into the left hand side of \eqref{equation:X} for our choice of $N$ in \eqref{equation:N}.

\end{proof}

\begin{lemma} \label{lemma:mbe0}  
Let $(g_0, \pi_0)\in C^2(\overline{\Omega}) \times C^1(\overline{\Omega})$ be an initial data set, and  let $W_0\in C^0(\overline{\Omega})$ be a vector field.  There is a neighborhood $\mathcal U$ of $(g_0, \pi_0)$ in $C^2(\overline{\Omega})\times C^1(\overline{\Omega})$, a neighborhood $\mathcal W$ of $W_0$ in $C^0(\overline{\Omega})$, and  a constant $C>0$  such that for $(g, \pi) \in \mathcal {U}$, $W\in \mathcal W$, and for $(f, X)\in H^2_{\rho_g}(\Omega)\times H^1_{\rho_g}(\Omega)$, 
the following estimate holds: 
\begin{align}
\|(f, X)\|_{H^2_{\rho_g}\times H^1_{\rho_g}} &\leq C \left( \| (D  \Phi^W_{(g, \pi)})^*(f, X)\|_{L^2_{\rho_g}} + \| (f, X)\|_{L^2_{\rho_g}}\right) \label{equation:(f,X)}.
\end{align}
\end{lemma}
\begin{proof}
The terms in $(D \Phi^W_{(g, \pi)})^*(f, X)$ that have the highest order of derivatives are $L^*_g f$ and $\mathcal D_gX$. Hence,
\begin{align*}
		\| L_{g}^* f \|_{L^2_{\rho_{g}}} &\le C\left(\| (D  \Phi^W_{(g, \pi)})^*(f, X)\|_{L^2_{\rho_{g}}} + \| f \|_{L^2_{\rho_{g}}} +\| X\|_{H^1_{\rho_{g}}} \right)\\
	\| \mathcal{D}_{g}  X\|_{L^2_{\rho_{g}}} &\le C\left( \| (D \Phi^W_{(g, \pi)})^*(f, X)\|_{L^2_{\rho_{g}}} + \| (f , X)\|_{L^2_{\rho_{g}}}\right).
\end{align*}
The desired inequalities follow from Lemma~\ref{lemma:f}.
\end{proof}

\begin{theorem} \label{theorem:coercivity}
Let $(g_0, \pi_0)\in C^{2} (\overline{\Omega})\times C^{1}(\overline{\Omega})$ be an initial data set, and let $W_0\in C^{0}(\overline{\Omega})$ be a vector field. Suppose that the kernel of $(D \Phi^{W_0}_{(g_0,\pi_0)})^*$ is trivial on $\Omega$.  Then there is a  neighborhood $\mathcal U$ of $(g_0, \pi_0)$ in $C^2(\overline{\Omega})\times C^1(\overline{\Omega})$, a neighborhood $\mathcal W$ of $W_0$ in $ C^0(\overline{\Omega})$, and a constant $C>0$ such that for  $(g, \pi ) \in \mathcal U$, $W\in \mathcal W$, and $(f, X) \in H^2_{\rho_{g}}(\Omega)\times H^1_{\rho_{g}}(\Omega)$,
the following estimate holds: 
\begin{align} \label{equation:coercivity}
	 \|(f, X)\|_{H^2_{\rho_{g}}\times H^1_{\rho_{g}}} \leq C \| (D \Phi^W_{(g, \pi)})^*(f , X)\|_{L^2_{\rho_{g}}} . 
 \end{align}
\end{theorem}
\begin{proof}
The proof is a standard argument, but we include it for the reader's convenience. Let $\mathcal U, \mathcal W$ be from Lemma~\ref{lemma:mbe0}. By shrinking the neighborhoods if necessary, we may assume that $(D\Phi^W_{(g, \pi)})^*$ has a trivial kernel on $\Omega$ for $(g, \pi) \in \mathcal{U}$, $W\in \mathcal W$. Suppose there were sequences $(g_k,\pi_k)\to (g, \pi)$ in $\mathcal U$, $W_k \to W$ in $\mathcal {W}$, and $(f_k, X_k)\in H^2_{\rho_k} \times H^1_{ \rho_k}$, for which 
 \[
	 \|(f_k, X_k)\|_{H^2_{\rho_k}\times H^1_{\rho_k}}=1
 \] 
 but with 
 \[
	 \| (D \Phi^{W_k}_{(g_k, \pi_k)})^*(f_k, X_k)\|_{L^2_{\rho_k}}\rightarrow 0,
 \]
 where $\rho_k = \rho_{g_k}$.  By Proposition~\ref{proposition:weighted-norm}, the sequence $(f_k, X_k)\rho^{\frac{1}{2}}_k $  is bounded in $H^2(\Omega)\times H^1(\Omega)$. By the Rellich theorem, upon choosing a suitable subsequence and relabeling, there is $(f, X)\in H^2_{\mathrm{loc}}(\Omega)\times H^1_{\mathrm{loc}}(\Omega)$ such that 
 \[
 \| (f_k, X_k)\rho^{\frac{1}{2}}_k-(f, X)\rho_{g}^{\frac{1}{2}}\|_{H^1\times L^2}\to 0.
 \] 
(Since the ${H^1(\Omega)\times L^2(\Omega)}$ norms are equivalent for $g$ in a neighborhood of $g_0$, the above convergence can be taken, for example,  with respect to $g$.) Because $\rho_k$ has a uniform positive lower bound on any given compact subset of $\Omega$, it implies that $(f_k, X_k)$ converges in $L^2_{\mathrm{loc}}$ to $(f, X)$ and  $(D\Phi^W_{(g, \pi)})^*(f, X)= 0$ weakly and hence $(f, X)$ is in the kernel of $(D\Phi^W_{(g, \pi)})^*$.  Thus,  $(f, X)=(0,0)$  because the kernel of $(D\Phi^W_{(g, \pi)})^*$ is trivial. Thus, $(f_k, X_k)\rho_{k}^{\frac{1}{2}}$ converges to zero in $H^1(\Omega)\times L^2(\Omega)$.  Lemma~\ref{lemma:mbe0} implies $ \|(f_k, X_k)\|_{H^2_{\rho_k}\times H^1_{\rho_k}}\rightarrow 0$ and contradicts our assumption.
\end{proof} 

\begin{remark}\label{remark:coercivity}
The above theorem is essentially the only place we need to assume that the kernel is trivial.  If the kernel were not trivial, the coercivity estimate would still hold transverse to the kernel.  More precisely, let $S$ be a complete linear subspace of $H^2_{\rho_{g}}(\Omega)\times H^1_{\rho_{g}}(\Omega)$ such that $S\cap K= \{0\}$ where $K = \textup{ker }(D \Phi^{W_0}_{(g_0,\pi_0)})^*$. Then the above argument implies that the coercivity estimate \eqref{equation:coercivity} holds for $(f, X) \in S$. The only difference in the proof is that after showing that the sequence $(f_k, X_k)$ converges to $(f, X) \in \textup{ker } (D\Phi^W_{(g, \pi)})^* =:K'$, one uses that $K'$ is also transverse to $S$ for sufficiently small neighborhoods $\mathcal U, \mathcal W$ to conclude $(f, X)=(0,0)$. 
\end{remark}

We now apply the coercivity estimate to obtain the variational solution of the linearized equation.

\begin{theorem}\label{theorem:var-sol} 
Let $(g_0, \pi_0) \in C^{4}(\overline{\Omega})\times C^{3}(\overline{\Omega})$ be an initial data set, and let $W_0 \in C^0(\overline{\Omega})$ be a vector field.  Suppose that the kernel of $(D \Phi^{W_0}_{(g_0,\pi_0)})^*$ is trivial on $\Omega$. Let the neighborhoods $\mathcal{U}$, $\mathcal W$ and the constant~$C$ be as in Theorem~\ref{theorem:coercivity}.  Then for $(g, \pi) \in \mathcal{U}$,  $W\in \mathcal W$,  and $(\psi, V)\in L^2_{\rho_g^{-1}}(\Omega)\times L^2_{\rho_g^{-1}}(\Omega)$, the functional $\mathcal{G}(f, X)$  has a global minimizer $(f,X)\in H^2_{\rho_g}(\Omega)\times H^1_{\rho_g}(\Omega)$. Furthermore, $(f,X)$ is the unique weak solution of the linear system \begin{equation}\label{equation:linearized}
D \Phi^W_{(g,\pi)} \circ \rho_g(D \Phi^W_{(g,\pi)})^*(f, X) = (\psi, V)
\end{equation}
and satisfies the estimate
\begin{align} \label{eq:var-est}
\|(f,X)\|_{H^2_{\rho_g}\times H^1_{\rho_{g}}} \leq 2C \| (\psi, V)\|_{L^2_{\rho_{g}^{-1}}\times L^2_{\rho_{g}^{-1}}}.
\end{align}
\end{theorem}

\begin{proof}
Theorem~\ref{theorem:coercivity} implies that the infimum of the functional $\mathcal{G}$ is bounded from below because 
\[
	\mathcal{G}(f, X) \ge \frac{1}{2C} \| (f, X) \|^2_{H^2_{\rho_{g}}\times H^1_{\rho_{g}}}-\| (\psi, V)\|_{L^2_{\rho_{g}^{-1}}\times L^2_{\rho_{g}^{-1}}}\| (f, X) \|_{L^2_{\rho_{g}} \times L^2_{\rho_{g}}}.
\] 
By standard variational theory, e.g. as in \cite[p.~150-152]{Corvino:2000}, the functional has a global minimizer $(f,X)\in H^2_{\rho_{g}} (\Omega)\times H^1_{\rho_{g}}(\Omega)$. Deriving the Euler-Lagrange equation for the functional $\mathcal{G}$ yields \eqref{equation:linearized}. The estimate \eqref{eq:var-est} follows because $\mathcal{G}(f, X)\le 0$. 

\end{proof}

\subsection{Weighted Schauder estimates}\label{subsection:Schauder}

This section is devoted to deriving the following weighted interior Schauder estimates for the linearized equation.

\begin{theorem} \label{theorem:var-sol-est} 
Let $(g_0, \pi_0)\in C^{4,\alpha}(\overline{\Omega})\times C^{3,\alpha}(\overline{\Omega})$ be an initial data set, and let $W_0 \in C^{2,\alpha}(\overline{\Omega})$ be a vector field.   Suppose that the kernel of $(D \Phi^{W_0}_{(g_0,\pi_0)})^*$ is trivial on $\Omega$.  There exists a  neighborhood $\mathcal{U}$  of  $(g_0, \pi_0)$ in $C^{4,\alpha}(\overline{\Omega})\times C^{3,\alpha}(\overline{\Omega})$,  a  neighborhood  $\mathcal W$ of $W_0$ in $C^{2,\alpha}(\overline{\Omega})$, and a constant  $C>0$ such that for $(g, \pi) \in \mathcal{U}$,  $W\in \mathcal W$,  and for $(\psi, V)\in \mathcal B_0 \times \mathcal{B}_1$, if $(f,X) \in H^2_{\rho_g}(\Omega) \times H^1_{\rho_g}(\Omega)$ weakly solves the linear system 
\[
	D \Phi^W_{(g,\pi)} \circ \rho_g (D\Phi^W_{(g,\pi)})^* (f, X)= (\psi, V),
\]
then $(f,X)\in \mathcal B_4  \times \mathcal{B}_3$ and 
\begin{align} \label{equation:Schauder-f-X}
 	\| (f,X)\|_{\mathcal B_4 \times \mathcal{B}_3}\leq C \| (\psi, V)\|_{\mathcal{B}_0 \times \mathcal{B}_1}.
\end{align}
Furthermore, if we set $(h,w)=  \rho_{g}(D \Phi^W_{(g,\pi)})^* (f, X)$, then 
\begin{align}\label{equation:Schauder-h-w}
	\|(h,w) \|_{\mathcal B_2\times \mathcal B_2} \leq C \| (\psi, V)\|_{\mathcal{B}_0 \times \mathcal B_1}.
\end{align}
\end{theorem}

We set up the framework of Douglis-Nirenberg \cite{Douglis-Nirenberg:1955} for the interior Schauder estimate for an elliptic system. Denote the linear system
\begin{align} \label{equation:L}
	L(f, X):&= \rho_g^{-1} D\Phi^W_{(g,\pi)}\rho_g (D \Phi^W_{(g,\pi)})^*(f, X).
\end{align}
Note the leading order terms of this operator are the same as the operator with the usual constraint map. Hence $L$ is strictly elliptic as a system of mixed order in the sense of Douglis-Nirenberg \cite{Douglis-Nirenberg:1955}, cf. \cite[pp.~4-5]{Chrusciel-Delay:2003}, \cite[p.~207]{Corvino-Schoen:2006}. In local coordinates, we write $X= (X^1, \dots, X^n)$ and set $U$ to be the vector-valued function
\[
	U=(U^1, U^2, \dots, U^{n+1}) = (f, X^1, \dots, X^n).
\]
Denote the $j$th component of $LU$ by $(L U)_j$ for $j= 1, \dots, n+1$, which we express locally as
\[
	(LU)_j = \sum_{k=1}^{n+1} \sum_{|\beta|=0}^4 b_{jk}^{\beta} \partial^{\beta}U^k =: \sum_{k=1}^{n+1} L_{jk} U^k,
\]	
where $L_{jk}$ is the differential operator $L_{jk} =\sum_\beta b_{jk}^{\beta} \partial^{\beta}$. Using the notation as in \cite{Douglis-Nirenberg:1955}, we solve for integers $s_1,\dots, s_{n+1}$ and $t_1,\dots, t_{n+1}$  such that $s_j + t_k$ is the order of the differential operator $L_{jk}$ and such that $s_1=0$. This gives 
\begin{align*}
	&s_1 = 0,\quad t_1 = 4, \quad  s_j = -1,\quad t_k=3 \qquad (j,k = 2, \dots, n+1).
\end{align*}
By direct computation, the function $b_{jk}^{\beta}$ is a degree one polynomial of $\rho^{-1} \nabla^\ell \rho$, $0\le \ell \le s_j + t_k - |\beta|$ with coefficients depending locally uniformly in $(g, \pi)\in C^{4,\alpha}(\overline{\Omega})\times C^{3,\alpha}(\overline{\Omega})$ and $W\in C^{2,\alpha}(\overline{\Omega})$. By Proposition~\ref{proposition:phi}, we have the following estimate. 
\begin{lemma} \label{lemma:weights}
The coefficients $b_{jk}^{\beta}$ satisfy 
\[
\| b^{\beta}_{jk}\|_{C^{-s_j, \alpha}_{\phi, \phi^{s_j + t_k - |\beta|}}(\Omega)} \le C,
\] 
where the constant $C$ depends locally uniformly  on $(g, \pi) \in C^{4,\alpha}(\overline{\Omega})\times C^{3,\alpha}(\overline{\Omega})$ and $W\in C^{2,\alpha}(\overline{\Omega})$.
\end{lemma}

We have the following interior Schauder estimate. The proof follows from a scaling argument as in \cite[Appendix B]{Chrusciel-Delay:2003} and is included in Appendix~\ref{section:Schauder}, where we spell out the dependence of the constant  $C$.

\begin{theorem}\label{theorem:U}
Let $L$ be a linear differential operator on $U=(U^1, \dots, U^{n+1})$ such that the $j$-th component of the operator $L$ is defined by
\[
	  (LU)_j= \sum_{k=1}^{n+1} L_{jk} U^k, \qquad (j= 1, \dots, n+1)
\]
where $L_{jk} = \sum_{|\beta|=0}^{s_j + t_k}b^{\beta}_{jk} \partial^{\beta}$ is a differential operator of order $s_j + t_k$ with
\begin{align*}
	&s_1 = 0,\quad t_1 = 4\\
	& s_j = -1,\quad t_k=3, \qquad (j,k = 2, \dots, n+1).
\end{align*}
Let $\varphi_j = \phi^{r+4-t_j} \rho^{s}$ for $r, s\in \mathbb{R}$. Then
\begin{align}  \label{equation:estimate-U}
	\sum_{j=1}^{n+1}\| U^j \|_{C^{t_j,\alpha}_{\phi, \varphi_j}(\Omega)} \le C\left(\sum_{j=1}^{n+1} \| (L U)_j\|_{C^{-s_j, \alpha}_{\phi, \phi^{t_j+s_j}\varphi_j}(\Omega)} + \sum_{j=1}^{n+1} \| U^j \|_{L^2_{\phi^{-n}\varphi^2_j}(\Omega)} \right)
\end{align}
where $C$ depends only on $n$, $\alpha$, $\sup\limits_{\stackrel{j, k= 1, \dots, n+1}{ |\beta|\leq s_j+t_k}} \| b^{\beta}_{jk}\|_{C^{-s_j, \alpha}_{\phi, \phi^{s_j + t_k - |\beta|}}(\Omega)}$, and the lower bound of ellipticity of the operator $L$, as well as $r$, $s$ and the constant in \eqref{equation:weight}. 
\end{theorem}
\begin{remark}\label{remark:higher-order}
We also note that higher order estimates take the form
\[
	\sum_{j=1}^{n+1} \| U^j \|_{C^{\ell+t_j, \alpha}_{\phi, \varphi_j}(\Omega)} \le C \left(\sum_{j=1}^{n+1} \| (L U)_j \|_{C^{\ell-s_j, \alpha}_{\phi, \phi^{t_j+s_j}\varphi_j}(\Omega)} + \sum_{j=1}^{n+1} \| U^j \|_{L^2_{\phi^{-n}\varphi_j^2}(\Omega)}\right),
\]
where $C$ depends only  on $n$, $\ell$, $\alpha$, $\sup\limits_{\stackrel{j, k= 1, \dots, n+1}{ |\beta|\leq s_j+t_k}} \| b^{\beta}_{jk}\|_{C^{\ell-s_j, \alpha}_{\phi, \phi^{s_j + t_k - |\beta|}}(\Omega)}$, and the lower bound of ellipticity of the operator $L$, as well as $r$, $s$ and the constant in \eqref{equation:weight}.
\end{remark}

\begin{proof}[Proof of Theorem~\ref{theorem:var-sol-est}]
Applying Theorem~\ref{theorem:U} (with $r=\frac{n}{2}, s = \frac{1}{2}$) to the operator 
\[
L(f, X)= \rho_g^{-1} D\Phi^W_{(g,\pi)} \rho_g (D\Phi^W_{(g,\pi)})^* (f, X),
\] 
we have 
\begin{align*}
	&\| f \|_{C^{4, \alpha}_{\phi, \phi^{\frac{n}{2}} \rho_g^{\frac{1}{2} }} }+ \| X \|_{C^{3, \alpha}_{\phi, \phi^{1+\frac{n}{2}}\rho_g^{\frac{1}{2}}}}=\sum_{i=1}^{n+1} \| U^i \|_{C^{t_i, \alpha}_{\phi, \varphi_i}}\\
	&\le C \left[\sum_{i=1}^{n+1} \| (L U)_i \|_{C^{-s_i, \alpha}_{\phi, \phi^{t_i+s_i}\varphi_i}} + \sum_{i=1}^{n+1} \| U^i \|_{L^2_{\phi^{-n}\varphi_i^2}}\right]\\
	&=C \left(\| \rho_g^{-1}\psi \|_{C^{0,\alpha}_{\phi, \phi^{4+\frac{n}{2}} \rho_g^{\frac{1}{2}}}} +  \| \rho_g^{-1} V \|_{C^{1, \alpha}_{\phi, \phi^{3+\frac{n}{2}}\rho_g^{\frac{1}{2}}}} + \| f \|_{L^2_{\rho_g} }+ \| X \|_{L^2_{\phi^2 \rho_g}}\right)\\
	&\le C \left(\| \psi \|_{C^{0,\alpha}_{\phi, \phi^{4+\frac{n}{2}} \rho_g^{-\frac{1}{2}}}} + \| V \|_{C^{1, \alpha}_{\phi, \phi^{3+\frac{n}{2}}\rho_g^{-\frac{1}{2}}}} + \| f \|_{L^2_{\rho_g} }+ \| X \|_{L^2_{\phi^2 \rho_g}}\right).
\end{align*}
The above H\"{o}lder estimate, together with the Sobolev estimate (\ref{eq:var-est}), implies \eqref{equation:Schauder-f-X}.

The estimate \eqref{equation:Schauder-h-w}  for $(h,w)$ follows because differentiation is a continuous operator from $C^{k,\alpha}_{\phi, \varphi}$ to $C^{k-1, \alpha}_{\phi, \phi \varphi}$ and from $H^k_{\rho_g^{-1}}$ to $H^{k-1}_{\rho_g^{-1}}$.
\end{proof}

\subsection{Solving the nonlinear problem by iteration}

We complete the proof of Theorem~\ref{thm:nl}, which can be formulated more explicitly as follows (and where we make the harmless change ``$2\psi$" to ``$\psi$" in (\ref{eq:sys-2psi}) for simplicity). 
\begin{theorem} \label{theorem:nl}
Let $(g_0, \pi_0) \in C^{4,\alpha}(\overline{\Omega})\times C^{3,\alpha}(\overline{\Omega})$ be an initial data set, and let $W_0\in C^{2,\alpha}(\overline{\Omega})$ be a vector field.  Suppose that the kernel of $(D\Phi^{W_0}_{(g_0,\pi_0)})^*$ is trivial in $\Omega$.   Let $C$ be the constant  from Theorem \ref{theorem:var-sol-est}.  There is a neighborhood $\mathcal U$ of $(g_0, \pi_0)$ in $C^{4, \alpha}(\overline{\Omega}) \times C^{3, \alpha}(\overline{\Omega})$, and a neighborhood $\mathcal W$ of $W_0$ in $C^{2,\alpha}(\overline{\Omega})$, and   $\epsilon>0$  such that for $(g, \pi)\in \mathcal U$, $W\in \mathcal W$ and for  $(\psi,V)\in \mathcal B_0 \times \mathcal B_1$ with $\|(\psi,V)\|_{ \mathcal B_0 \times \mathcal B_1} \leq \epsilon$, there exists $(f, X) \in \mathcal{B}_4\times \mathcal B_3$ such that
\[
(h, w) = \rho_g (D\Phi^W_{(g, \pi)})^*(f, X)
\] 
satisfies $(h,w)\in \mathcal B_2\times \mathcal B_2$ and 
\[
 \Phi^W_{(g, \pi)} (g+h , \pi + w)  = \Phi^W_{(g, \pi)} (g, \pi )+ (\psi, V)
 \]
with $\|(h,w)\|_{\mathcal B_2\times \mathcal B_2} \leq C \| (\psi, V)\|_{\mathcal B_0\times \mathcal B_1}$.
\end{theorem}

\begin{proof}
By Theorem~\ref{theorem:var-sol},  there exists $(f_0, X_0)\in \mathcal{B}_4\times \mathcal B_3$ such that $(h_0, w_0)= \rho_g (D\Phi^W_{(g,\pi)})^*(f_0, X_0)$ solves $D\Phi^W_{(g, \pi)}(h_0, w_0) = (\psi, V)$ and, by Theorem~\ref{theorem:var-sol-est}, the following estimates hold
\[
	\| ( h_0, w_0) \|_{\mathcal B_2\times \mathcal B_2} \le C\| (\psi, V)\|_{\mathcal B_0\times \mathcal B_1}, \quad\quad \| (f_0, X_0)\|_{\mathcal B_4\times \mathcal B_3} \le C \| (\psi, V)\|_{\mathcal B_0\times \mathcal B_1}.
\]
By Lemma~\ref{lemma:Taylor-modified} we obtain that 
\begin{align*}
	&\| \Phi^W_{(g,\pi)}(g, \pi)+ (\psi, V) - \Phi^W_{(g,\pi)} (g+ h_0, \pi+w_0) \|_{\mathcal B_0\times \mathcal B_1}\\
	&= \| Q^W_{(g,\pi)} (h_0,w_0)\|_{\mathcal B_0\times \mathcal B_1}\le D \| (h_0, w_0)\|_{\mathcal B_2\times \mathcal B_2}^2 \le DC^2 \| (\psi, V)\|_{\mathcal B_0\times \mathcal B_1}^2. 
\end{align*}
Write $(\gamma_1, \tau_1 ) = (g, \pi) + (h_0, w_0)$. Note that $\gamma_1$ is a Riemannian metric provided $\|(h_0, w_0)\|_{\mathcal B_0\times \mathcal B_1}$ is sufficiently small. Fix $\delta \in (0, 1)$. We set $\epsilon$ sufficiently small so that $DC^2 \epsilon^{1-\delta} \le 1$.

We then proceed recursively as in the following lemma, whose proof is included in Appendix~\ref{section:recursion}.
\begin{lemma}\label{lemma:induction}
Fix $(g_0, \pi_0)$ and $\delta \in (0, 1)$. There exists a neighborhood $\mathcal{U}$ of $(g_0, \pi_0)$ in $C^{4,\alpha}(\overline{\Omega})\times C^{3,\alpha}(\overline{\Omega})$, a neighborhood $\mathcal W$ of $W_0$ in $C^{2,\alpha}(\overline{\Omega})$,  and $\epsilon \in (0, \frac{1}{2})$ depending only on $\delta$ and $\Omega$, such that for  $(g, \pi)\in \mathcal{U}$ and  for $W\in \mathcal W$ the following holds. Suppose that $m\ge 1$ and we have constructed $(f_0, X_0), \dots, (f_{m-1}, X_{m-1})\in \mathcal{B}_4(\Omega)\times \mathcal B_3(\Omega)$, $(h_0, w_0), \dots, (h_{m-1}, w_{m-1}) \in \mathcal{B}_2(\Omega)\times \mathcal B_2(\Omega)$ where $(h_p, w_p) = \rho_g (D\Phi^W_{(g, \pi)})^* (f_p, X_p)$  and $(\gamma_1, \tau_1), \dots, (\gamma_m, \tau_m) \in C^{2,\alpha}(\overline{\Omega}) \times C^{1,\alpha}(\overline{\Omega})$ where $\gamma_j = g + \sum_{p=0}^{j-1} h_p$ and $\tau_j = \pi + \sum_{p=0}^{j-1} w_p$. Assume that $\| (\psi, V) \|_{\mathcal B_0\times \mathcal B_1} \le \epsilon$ and that for all $0 \le p \le m-1$,
\begin{align} \label{equation:induction1}
	\| (f_p, X_p ) \|_{\mathcal B_4\times \mathcal B_3} \le C \| (\psi, V) \|_{\mathcal B_0\times \mathcal B_1}^{1+p\delta} \quad \mbox{and} \quad \| (h_p, w_p) \|_{\mathcal B_2\times \mathcal B_2} \le C\| (\psi, V) \|_{\mathcal B_0\times \mathcal B_1}^{1+p\delta}, 
\end{align} 
and that for all $1 \le j \le m$, 
\begin{align} \label{equation:induction2}
	\|\Phi_{(g,\pi)}^W (g, \pi) + (\psi, V) - \Phi^W_{(g,\pi)} (\gamma_j, \tau_j) \|_{\mathcal B_0\times \mathcal B_1} \le \| (\psi, V) \|_{\mathcal B_0\times \mathcal B_1}^{1+ j \delta}.
\end{align}
If we define $(h_m, w_m) = \rho_g D(\Phi^W_{(g,\pi)} )^* (f_m, X_m)$ where $(f_m, X_m)$ is the variational solution to 
\[
	D\Phi^W_{(g, \pi)}\rho_g (D\Phi^W_{(g,\pi)})^* (f_m, X_m)= \Phi^W_{(g,\pi)} (g,\pi) + (\psi, V) - \Phi^W_{(g,\pi)}(\gamma_m, \tau_m),
\]
 and if we set $(\gamma_{m+1}, \tau_{m+1}) = (\gamma_m, \tau_m) + (h_m, w_m)$, then the estimates \eqref{equation:induction1} and \eqref{equation:induction2} hold for $p = m$ and $j = m+1$.
\end{lemma}

We obtain the series $\sum_{p=0}^{\infty}(f_p, X_p)$ converging in $\mathcal{B}_4(\Omega)\times \mathcal B_3(\Omega)$ to some $(f, X)$. Let $(h, w) = \rho_g (D\Phi^W_{(g,\pi)})^* (f, X)$, then $(g+h, \pi+w)$ satisfies the nonlinear system $\Phi_{(g,\pi)}^W(g+h, \pi+w) = \Phi_{(g,\pi)}^W(g, \pi)+(\psi, V)$. 
\end{proof}

\section{Proof of Theorem~\ref{theorem:projected-modified} and higher regularity} \label{section:projected}

The proof of Theorem~\ref{theorem:projected-modified} is along the same lines as the proof of Theorem~\ref{thm:nl} and Theorem \ref{theorem:nl} in Section~\ref{section:modified}.  Recall that $\mathcal{U}_0$ is a bounded neighborhood of a Riemannian metric $g_0$, as in Section~\ref{subsection:weighted}, and the precompact smooth subdomain $\Omega_0$ of $\Omega$ is chosen so that $\rho_g \equiv 1$ on $\Omega_0$ for all $g\in \mathcal{U}_0$.  Fix a smooth bump function $\zeta$ supported in $\overline{\Omega_0}$.  Denote by $S_g$ the $L^2(d\mu_g)$-orthogonal complement of $\zeta K$, where $K$ is the kernel of $(D\Phi^{W_0}_{(g_0, \pi_0)})^*$.  Let $\Pi_g: \mathcal{B}_0\times \mathcal B_1 \to (\mathcal{B}_0\times \mathcal B_1) \cap S_g$ be the $L^2(d\mu_g)$-orthogonal projection. 

We begin with a basic lemma on the projection map. Throughout this section, the function spaces are all taken on $\Omega$, unless otherwise indicated. 

\begin{lemma}\label{lemma:projection}
There is a constant $C>0$ such that for $g\in \mathcal{U}_0$, 
\begin{align*}
	\| \Pi_g (\psi, V)\|_{L^2_{\rho^{-1}_g}} &\le \| (\psi, V)\|_{L^2_{\rho^{-1}_g}} \\ 
	\| (\psi, V)^{\perp}\|_{L^2_{\rho^{-1}_g}} &\le \| (\psi, V)\|_{L^2_{\rho^{-1}_g}} \\	\| \Pi_g(\psi, V)\|_{\mathcal B_0\times \mathcal B_1} &\le C \| (\psi, V) \|_{\mathcal B_0\times \mathcal B_1}, 
\end{align*}
where $(\psi, V)^{\perp}= (\psi, V) - \Pi_g (\psi, V) \in \zeta K$.
\end{lemma}
\begin{proof}
Because $\zeta$ is supported in $\overline{\Omega_0}$ where $\rho_g\equiv 1$, the following weighted orthogonality holds:
\[
	\langle \Pi_g(\psi, V), (\psi, V)^{\perp} \rangle_{L_{\rho^{-1}_g}^2(\Omega)}=\langle \Pi_g(\psi, V), (\psi, V)^{\perp} \rangle_{L^2(\Omega_0)}=0.
\]
The first two inequalities follow by
\begin{align*}
	 \| (\psi, V)\|^2_{L^2_{\rho^{-1}_g}}&=  \| \Pi_g (\psi, V)\|^2_{L^2_{\rho^{-1}_g}}+ \| (\psi, V)^{\perp}\|^2_{L^2_{\rho^{-1}_g}}.
\end{align*}
 To establish the last inequality, we recall that all norms of the finite-dimensional space $\zeta K$ are equivalent and the $\mathcal B_0 \times \mathcal B_1$-norms on $\zeta K$ are all uniformly equivalent to each other for $g\in \mathcal U_0$:
\begin{align*}
 	\| \Pi_g (\psi, V) \|_{\mathcal B_0\times \mathcal B_1} &\le \| (\psi, V) \|_{\mathcal B_0\times \mathcal B_1} + \| (\psi, V)^{\perp} \|_{\mathcal B_0\times \mathcal B_1}\\
	&  \le\| (\psi, V) \|_{\mathcal B_0\times \mathcal B_1}+ C \| (\psi, V)^{\perp} \|_{L^2}\\
	& =\| (\psi, V) \|_{\mathcal B_0\times \mathcal B_1}+ C \| (\psi, V)^{\perp} \|_{L_{\rho_g^{-1}}^2} \\
	&\le (1+ C) \| (\psi, V) \|_{\mathcal B_0\times \mathcal B_1}.
\end{align*}
In the third line we use that $\zeta$ is supported where $\rho_g\equiv 1$.

\end{proof}

\subsection{The linearized equation}
In this section, we solve the linearized equation for the operator $\Pi_{g_0}\circ \Phi^W_{(g, \pi)}$. We again denote by $D\Phi^W_{(g,\pi)}= D\Phi^W_{(g,\pi)}|_{(g,\pi)}$ the linearization at $(g, \pi)$.  

\begin{theorem} \label{theorem:linear-projected}
Let $(g_0, \pi_0)\in C^{4,\alpha}(\overline{\Omega})\times C^{3,\alpha}(\overline{\Omega})$ be an initial data set, and let $W_0\in C^{2,\alpha}(\overline{\Omega})$ be a vector field. There is a neighborhood $\mathcal{U}$ of $(g_0, \pi_0)$ in $C^{4,\alpha}(\overline{\Omega})\times C^{3,\alpha}(\overline{\Omega})$, a neighborhood $\mathcal W$ of $W_0$ in $C^{2,\alpha}(\overline{\Omega})$, and  a constant $C>0$,  such that for  $(g, \pi)\in \mathcal{U}, W\in \mathcal W$ and for $(\psi, V)\in L^2_{\rho_g^{-1}}$, there is a unique $(f, X)\in  (H_{\rho_g}^2 \times H^1_{\rho_g})\cap S_{g}$
that weakly solves
\begin{align} \label{equation:projected-linear}
	\Pi_{g_0} \circ D\Phi^W_{(g,\pi)} \circ \rho_g (D\Phi^W_{(g,\pi)})^* (f, X) = \Pi_{g_0}(\psi, V),
\end{align}
or, equivalently, 
\[
	 D\Phi^W_{(g,\pi)}\circ \rho_g (D\Phi^W_{(g,\pi)})^*(f, X) -(\psi, V)\in \zeta K.
\]
Moreover, $(f, X)$ satisfies the estimate
\begin{align}\label{equation:estimate-variational-sol}
	\| (f, X)\|_{H^2_{\rho_{g}}\times H^1_{\rho_{g}}}\le C\| \Pi_{g_0}(\psi, V)\|_{L^2_{\rho_g^{-1}}\times L^2_{\rho_g^{-1}}}\le C\| (\psi, V)\|_{L^2_{\rho_g^{-1}}\times L^2_{\rho_g^{-1}}}.
\end{align}
\end{theorem}

The proof of the theorem is a modification of the variational argument in Theorem~\ref{theorem:var-sol}. While the projection map is taken with respect to a fixed metric $g_0$, the functional $\mathcal{G}$ below naturally involves $g$. To resolve this, we consider the analogous linear equation whose projection is with respect to $g$ and look for solutions $(f, X)\in S_g$. Theorem~\ref{theorem:linear-projected} follows from the proposition below.

\begin{proposition}\label{proposition:linear-projected}
Let $(g_0, \pi_0)\in C^{4,\alpha}(\overline{\Omega})\times C^{3,\alpha}(\overline{\Omega})$ be an initial data set, and let $W_0\in C^{2,\alpha}(\overline{\Omega})$ be a vector field. There is a neighborhood $\mathcal{U}$ of $(g_0, \pi_0)$ in $C^{4,\alpha}(\overline{\Omega})\times C^{3,\alpha}(\overline{\Omega})$, a neighborhood $\mathcal W$ of $W_0$ in $C^{2,\alpha}(\overline{\Omega})$, and a constant $C>0$ such that for  $(g, \pi)\in \mathcal{U}$, $W\in \mathcal W$, and for $(\psi, V)\in  L^2_{\rho_g^{-1}}(\Omega)$, there is a unique $(f, X)\in (H_{\rho_g}^2(\Omega) \times H^1_{\rho_g}(\Omega))\cap S_{g}$ that weakly solves
\begin{align} \label{equation:projected-linear-2}
	\Pi_{g} \circ D\Phi^W_{(g,\pi)} \circ \rho_g (D\Phi^W_{(g,\pi)})^* (f, X) =\Pi_g (\psi, V),
\end{align}
or equivalently, 
\[
	D\Phi^W_{(g,\pi)} \circ \rho_g (D\Phi^W_{(g,\pi)})^* (f, X) - (\psi, V) \in \zeta K.
\]
Moreover, $(f, X)$ satisfies the estimate
\begin{align}\label{equation:estimate-variational}
	\| (f, X)\|_{H^2_{\rho_{g}}\times H^1_{\rho_{g}}}\le C\|\Pi_g (\psi, V)\|_{L^2_{\rho_g^{-1}}\times L^2_{\rho_g^{-1}}}\le C\| (\psi, V)\|_{L^2_{\rho_g^{-1}}\times L^2_{\rho_g^{-1}}}.
\end{align}
\end{proposition}
\begin{proof}
For given $(\psi, V)\in L^2_{\rho_g^{-1}}(\Omega)$, let $\mathcal{G}$ be a similar functional as in Theorem~\ref{theorem:var-sol}, whose domain is restricted to the linear subspace $ (H^2_{\rho_g}(\Omega)\times H^1_{\rho_g}(\Omega))\cap S_g$: for $(f,X)\in (H^2_{\rho_g}(\Omega)\times H^1_{\rho_g}(\Omega))\cap S_g$, let 
\begin{align*}
	\mathcal{G}(f, X) &= \int_{\Omega} \left[ \frac{1}{2} \rho_g \left| (D\Phi^W_{(g,\pi)})^*(f, X)\right|_g^2 - \Pi_g(\psi, V)\cdot_g (f, X) \right]d\mu_g.
\end{align*}
By the coercivity estimate (see Remark~\ref{remark:coercivity}),
\[
	\mathcal{G}(f, X)\ge \frac{1}{2C} \|(f, X)\|^2_{H^2_{\rho_g}\times H^1_{\rho_g}} - \| \Pi_g(\psi, V)\|_{L^2_{\rho_g^{-1}}} \| (f, X) \|_{L^2_{\rho_g}}
\]  It is clear that functional is still convex when restricted on the linear subspace, so there is a unique minimizer $(f, X)\in (H^2_{\rho_g}(\Omega)\times H^1_{\rho_g}(\Omega))\cap S_g$. Furthermore, since $\mathcal{G}(f, X)\le 0$ and by Lemma~\ref{lemma:projection}, the estimate \eqref{equation:estimate-variational} holds.

To see that $(f, X)$ solves~\eqref{equation:projected-linear-2} weakly, we need to show that, for all smooth compactly supported test fields $(u, Y)$, 
\begin{align} \label{equation:weak-solution}
	\int_{\Omega} \left(\Pi_{g} \circ D\Phi^W_{(g,\pi)}\circ \rho_g (D\Phi_{(g,\pi)})^* (f, X) - \Pi_g(\psi, V)\right) \cdot_g (u, Y)\, d\mu_g = 0.
\end{align}
The equality trivially holds for $(u, Y)\in \zeta K$. For  $(u, Y) \in S_g$, because $(f, X)$ is a minimizer in $  (H^2_{\rho_g}(\Omega)\times H^1_{\rho_g}(\Omega))\cap S_g$, we see that 
\begin{align*}
	0 &= \left.\frac{d}{dt}\right|_{t=0} \mathcal{G}\left((f,X)+t(u, Y)\right)\\
	&=\int_{\Omega} \rho_g (D\Phi^W_{(g,\pi)})^*(f, X) \cdot_g (D\Phi^W_{(g,\pi)})^*(u, Y) -\Pi_g (\psi, V)\cdot_g (u, Y) \, d\mu_g\\
	&=\int_{\Omega} \left( D\Phi^W_{(g,\pi)} \circ \rho_g (D\Phi^W_{(g,\pi)})^*(f, X) - \Pi_g(\psi, V) \right)\cdot_g (u, Y) \, d\mu_g\; ,
\end{align*}
which implies \eqref{equation:weak-solution}.
\end{proof}

\subsection{Weighted Schauder estimates}

\begin{proposition} \label{prop:projected-Schauder}
Let $(g_0,\pi_0)\in C^{4,\alpha}(\overline{\Omega})\times C^{3,\alpha}(\overline{\Omega})$ be an initial data set, and let $W_0\in C^{2,\alpha}(\overline{\Omega})$ be a vector field. There is a neighborhood $\mathcal U $ of $(g_0, \pi_0)$ in $C^{4,\alpha}(\overline{\Omega})\times C^{3,\alpha}(\overline{\Omega})$, a neighborhood $\mathcal W$ of $W_0$ in $C^{2,\alpha}(\overline{\Omega})$, and a constant $C>0$ such that for $(g,\pi)\in \mathcal U$, $W\in \mathcal W$, and for $(\psi,V)\in \mathcal B_0\times \mathcal B_1$, if $(f, X) \in (H^2_{\rho_g}(\Omega)\times H^1_{\rho_g}(\Omega))\cap S_g$ weakly solves the linear system 
\[
	\Pi_{g_0}\circ D\Phi^W_{(g,\pi)}\circ \rho_g (D\Phi^W_{(g,\pi)})^* (f, X) =\Pi_{g_0}(\psi, V),
\] 
then $(f, X)\in \mathcal B_4\times \mathcal B_3$ and 
\[
	\|(f, X)\|_{\mathcal B_4\times \mathcal B_3} \le C \| \Pi_{g_0}(\psi, V) \|_{\mathcal B_0\times \mathcal B_1}.
\]
\end{proposition}
\begin{proof}
Define
\begin{align*}
	L(f, X)&= \rho_g^{-1} D\Phi^W_{(g,\pi)} \circ \rho_g (D\Phi^W_{(g,\pi)})^*(f, X)\\
	U &= (f, X^1,\dots, X^n) \quad \mbox{(with respect to a fixed coordinate chart)}
\end{align*}
as in the proof of Theorem~\ref{theorem:var-sol-est}. Below we denote $\rho_g$ by $\rho$. We write $$LU = \rho^{-1} \Pi_{g_0} (\rho L U) +\rho^{-1} ( \rho LU)^\perp= \rho^{-1} \Pi_{g_0} (\psi, V) + (LU)^{\perp}\; ,$$ where we used the fact that $\rho_{g_0} \equiv 1$ on the support of $\zeta$.  Applying Theorem~\ref{theorem:U} as in the proof of Theorem~\ref{theorem:var-sol-est}, we have
\begin{align*}
	&\| f \|_{C^{4, \alpha}_{\phi, \phi^{\frac{n}{2}} \rho^{\frac{1}{2} }} }+ \| X \|_{C^{3, \alpha}_{\phi, \phi^{1+\frac{n}{2}}\rho^{\frac{1}{2}}}}=\sum_{j=1}^{n+1} \| U^j \|_{C^{t_j, \alpha}_{\phi, \varphi_j}}\\
	&\le C \left[\sum_{j=1}^{n+1} \| (L U)_j \|_{C^{-s_j, \alpha}_{\phi, \phi^{t_j+s_j}\varphi_j}} + \sum_{j=1}^{n+1} \| U^j \|_{L^2_{\phi^{-n}\varphi_j^2}}\right]\\
	&\le C\left( \sum_{j=1}^{n+1} \|  (\rho^{-1} \Pi_{g_0} (\psi, V))_j\|_{C^{-s_j, \alpha}_{\phi, \phi^{t_j+s_j}\varphi_j}} + \sum_{j=1}^{n+1} \|  ( L U)^{\perp}_j\|_{C^{-s_j, \alpha}_{\phi, \phi^{t_j+s_j}\varphi_j}}+ \sum_{j=1}^{n+1} \| U^j \|_{L^2_{\phi^{-n}\varphi_j^2}}\right)\\
	&\le C\left( \sum_{j=1}^{n+1} \|  (\Pi_{g_0} (\psi, V))_j\|_{C^{-s_j, \alpha}_{\phi, \phi^{t_j+s_j}\varphi_j\rho^{-1}}} + \sum_{j=1}^{n+1} \|  ( L U)^{\perp}_j\|_{C^{-s_j, \alpha}_{\phi, \phi^{t_j+s_j}\varphi_j}}+ \sum_{j=1}^{n+1} \| U^j \|_{L^2_{\phi^{-n}\varphi_j^2}}\right). 
\end{align*}
We now estimate the $(L U)^{\perp}$-term. Just as in the proof of Lemma~\ref{lemma:projection}, using the fact that all norms of a finite-dimensional space are equivalent, and that the support of $(L U)_j ^{\perp}$ is contained in $\overline{\Omega_0}$, we obtain the following estimate, uniformly across $\mathcal U$ and $\mathcal W$:
\[
	 \| (L U)_j^{\perp}\|_{C^{-s_j, \alpha}_{\phi, \phi^{t_j+s_j}\varphi_j} (\Omega)} \le C\|(L U)^{\perp}\|_{L^2(\Omega)} =C\|(LU)^{\perp}\|_{L^2(\Omega_0)}\le C\| L U\|_{L^2(\Omega_0)}.
\]
To estimate $\| L U\|_{L^2(\Omega_0)}$, we note that $L$ is a differential operator that contains four derivatives on $f$ and three derivatives on $X$ and that $\overline{\Omega_0}$ is compact. Then by enlarging the constant $C$ if necessary and by interpolation, we have
\begin{align*}
	\| (L U)_j^{\perp}\|_{C^{-s_j, \alpha}_{\phi, \phi^{t_j+s_j}\varphi_j} (\Omega)}&\le C \| (f, X) \|_{C^4(\Omega_0)\times C^3(\Omega_0)}\\
	& \le \epsilon \| (f, X) \|_{C^{4,\alpha}(\Omega_0)\times C^{3,\alpha}(\Omega_0)} + C(\epsilon) \| (f, X) \|_{L^2(\Omega_0) \times L^2(\Omega_0) }. \\
	&\le C\left[ \epsilon \left( \| f \|_{C^{4,\alpha}_{\phi, \phi^{\frac{n}{2}} \rho^{\frac{1}{2}} }}  + \| X \|_{C^{3,\alpha}_{\phi, \phi^{1+\frac{n}{2}} \rho^{\frac{1}{2}} }}\right)  +  C(\epsilon) \left( \| f\|_{L^2_{\rho} }+ \| X\|_{ L^2_{\phi^2 \rho}} \right) \right],
\end{align*}
where in the last inequality, we replace the norms on $\Omega_0$ by the corresponding weighted norms on $\Omega$, up to the multiple of a constant that is uniform in $(g,\pi)\in \mathcal U$ and $W\in \mathcal W$.  The weighted Sobolev estimates on $(f, X)$ follow by Theorem~\ref{theorem:linear-projected}, and the fact that the solution to the linear system is unique.  Using this and  choosing $\epsilon>0$ sufficiently small in the preceding inequality, we can absorb the weighted H\"{o}lder norm on $(f, X)$ to give the desired H\"{o}lder estimates.  \end{proof}

\subsection{Solving the nonlinear projected problem by iteration}

We discuss how to to solve the nonlinear problem for the proof of Theorem~\ref{theorem:projected-modified}:  for $(\psi, V)\in \mathcal{B}_0\times \mathcal B_1$ sufficiently small,  there is  $(f, X)\in (\mathcal{B}_4 \times \mathcal{B}_3) \cap S_g$ such that  $(h,w) =\rho_g (D\Phi^W_{(g,\pi)})^*(f, X)$  solves 
\[
	\Pi_{g_0}\circ \Phi^W_{(g, \pi)} (g+h , \pi + w)  = \Pi_{g_0}\circ \Phi^W_{(g, \pi)} (g, \pi )+ \Pi_{g_0}(\psi, V)
\]
with $\|(h,w)\|_{\mathcal B_2\times \mathcal{B}_2} \le C\| \Pi_{g_0}(\psi, V)\|_{\mathcal B_0 \times \mathcal B_1}$.

The proof follows the same iteration scheme as in the proof of Theorem~\ref{theorem:nl} by replacing $\Phi^W_{(g,\pi)}$ with $\Pi_{g_0} \circ  \Phi^W_{(g,\pi)}$ and $(\psi, V)$ with $\Pi_{g_0} (\psi, V)$.  
The initial step of the iteration is solving the following for $(f_0, X_0)\in S_g$
\[
	\Pi_{g_0} \circ D\Phi^W_{(g,\pi)} \circ \rho_g (D\Phi^W_{(g,\pi)})^* (f_0, X_0) = \Pi_{g_0} (\psi, V)\; , 
\] 
and then setting $(h_0, w_0)= \rho_g (D\Phi^W_{(g,\pi)})^* (f_0, X_0)$ and $(\gamma_1, \tau_1) = (g+h_0, \pi+w_0)$. We then solve inductively for $m\ge 0$
\[
	\Pi_{g_0} \circ D\Phi^W_{(g,\pi)} \circ \rho_g (D\Phi^W_{(g,\pi)})^* (f_m, X_m) =  \Pi_{g_0} \circ \Phi^W_{(g,\pi)} (g,\pi) + \Pi_{g_0}(\psi, V) - \Pi_{g_0} \circ \Phi^W_{(g,\pi)}(\gamma_m, \tau_m)
\]
and set $(h_{m}, w_{m})= \rho_g (D\Phi^W_{(g,\pi)})^* (f_{m}, X_{m})$
and $(\gamma_{m+1}, \tau_{m+1}) = (g+\sum_{p=0}^m h_p, \pi + \sum_{p=0}^m w_p)$. The essential estimates to guarantee the iteration procedure converges are the following:
\begin{align*}
	\| (f_m, X_m)\|_{\mathcal B_4\times \mathcal B_3} &\le C  \|  \Pi_{g_0} \circ \Phi^W_{(g,\pi)} (g,\pi) + \Pi_{g_0}(\psi, V) - \Pi_{g_0} \circ \Phi^W_{(g,\pi)}(\gamma_m, \tau_m)\|_{\mathcal B_0\times \mathcal B_1}\\
	\| (h_m, w_m)\|_{\mathcal B_2 \times \mathcal B_2} &\le C \| (f_m, X_m)\|_{\mathcal B_4\times \mathcal B_3}
\end{align*}
and
\begin{align*}
	\| \Pi_{g_0} \circ D\Phi^W_{(g,\pi)} |_{(\gamma, \tau)} (h,w)- \Pi_{g_0}\circ D\Phi^W_{(g,\pi)} |_{(\gamma', \tau')} (h,w) \|_{\mathcal B_0 \times \mathcal B_1 }&\le D \| (h, w) \|_{\mathcal B_2\times \mathcal B_2} \| (\gamma-\gamma', \tau- \tau')\|_{\mathcal B_2 \times \mathcal B_2}\\
	\| \Pi_{g_0} \circ Q^W_{(g,\pi)}(h_0, w_0)\|_{\mathcal B_0\times \mathcal B_1}&\le D\| (h_0, w_0)\|^2_{\mathcal B_2\times \mathcal B_2}.
\end{align*}
The first two estimates follow by Proposition~\ref{prop:projected-Schauder} and the fact that the differential operator is continuous between the corresponding weighted spaces. The last two estimates follow by Lemma~\ref{lemma:projection} and the estimates for unprojected operators from Lemma~\ref{lemma:Taylor-modified}.

\subsection{Higher order regularity and continuous dependence}\label{subsection:higher-regularity}

The previous analysis, in particular Remark~\ref{remark:higher-order}, implies the following version of the local surjectivity theorem with higher order regularity.  For simplicity we state the theorem for $(\psi, V)$ of compact support, but one can more generally pose that $(\psi, V)$ lies in a suitable weighted space (an infinite intersection of such spaces for the $C^{\infty}$-case, for example).

\begin{theorem} \label{thm:nl-k} 
Let $k\ge 0$. Let $( g_0, \pi_0) \in C^{k+4,\alpha}(\overline{\Omega})\times C^{k+3,\alpha}(\overline{\Omega})$ be an initial data set, and let $W_0\in C^{k+2,\alpha}(\overline{\Omega})$.  Suppose that the kernel of $(D\Phi^{W_0}_{(g_0, \pi_0)})^*$ is $K$ (may be trivial).  Then there is a $C^{k+4,\alpha}(\overline{\Omega}) \times C^{k+3,\alpha}(\overline{\Omega})$ neighborhood $\mathcal U$ of $(g_0, \pi_0)$, a neighborhood $\mathcal W$ of $W_0$ in $C^{k+2,\alpha}(\overline{\Omega})$, and constants $\epsilon>0$, $C>0$ such that for $(g, \pi)\in \mathcal U$, $W\in \mathcal W$, and for $(\psi,V)\in C^{k,\alpha}_c(\Omega)\times C_c^{k+1,\alpha}(\Omega)$ with $\|(\psi,V)\|_{\mathcal B_0\times \mathcal B_1} \leq \epsilon$, there is a pair of symmetric tensors $(h,w)\in C^{k+2,\alpha}_c(\Omega)\times C^{k+2,\alpha}_c(\Omega)$ with $\|(h,w)\|_{\mathcal B_2\times \mathcal B_2} \leq C \| \Pi_{g_0}(\psi, V)\|_{\mathcal B_0\times \mathcal{B}_1}$, such that the initial data set $(g+h, \pi+w)\in C^{k+2,\alpha}(\overline{\Omega})\times C^{k+2,\alpha}(\overline{\Omega})$ satisfies
\[
	\Pi_{g_0} \circ \Phi^W_{(g, \pi)} (g+h , \pi + w)  = \Pi_{g_0} \Phi^W_{(g, \pi)} (g, \pi )+ \Pi_{g_0}(\psi, V).
\]
If, in addition, $(g, \pi)\in C^{\infty}(\overline{\Omega})$ and $(\psi, V)\in C^{\infty}_c(\Omega)$, then $(h, w)\in C^{\infty}_c(\Omega)$.
\end{theorem}

\begin{proof}
In the proof of Theorem~\ref{theorem:projected-modified} for $k=0$, we have obtained $(f, X)\in (\mathcal{B}_4 \times \mathcal{B}_3) \cap S_g$ such that  
\[
(h,w) =\rho_g (D\Phi^W_{(g,\pi)})^*(f, X)
\] 
solves the nonlinear equation. That is, $(f, X)$ satisfies the quasi-linear elliptic system
\[
	 \Phi^W_{(g, \pi)} \left((g , \pi) +\rho_g (D\Phi^W_{(g,\pi)})^*(f, X) \right)-\Phi^W_{(g, \pi)} (g, \pi )- (\psi, V) \in \zeta K
\]
along with the desired estimate.  Because elements in $\zeta K$ are smooth with compact support, using the initial regularity $(f, X)\in \mathcal{B}_4 \times \mathcal{B}_3$, along with bootstrapping, one can get higher-order estimates and boundary decay, by applying the estimates in Remark~\ref{remark:higher-order} to the quasi-linear system, cf. \cite[Sec. 3.7]{Corvino-Eichmair-Miao:2013}.  This yields the $C^{\infty}$-regularity statement as well.  As for the compact support, we could proceed as in \cite[Sec. 3.7]{Corvino-Eichmair-Miao:2013}, by replacing $\Omega$ with a precompact smooth subdomain $\Omega'$, suitably chosen so that (i) $\Omega'\supset\overline{\Omega_0} \cup \mathrm{supp}(\psi, V)$; (ii) $K$ restricts to the kernel of $(D\Phi^{W_0}_{(g_0, \pi_0)})^*$ on $\Omega'$; (iii) there is a diffeomorphism $F: \Omega\rightarrow \Omega'$ which is sufficiently close to the identity and restricts to the identity on $\Omega_0$.   
\end{proof}

\begin{remark}  We note that if we fix a compact subset $L$, and work only with $(\psi, V)$ supported in $L$, we could get the finite regularity result above by posing the smallness condition on the norm $\|(\psi, V)\|_{C^{k,\alpha} \times C^{k+1, \alpha}}$ and proving convergence of the iteration scheme in $C^{k+2,\alpha}\times C^{k+2,\alpha}$ to the limit $(h,w)$ (along with weighted decay).  
\end{remark}

\begin{remark} \label{rmk:cts-dep}  In Theorem~\ref{thm:nl}, Theorem~\ref{theorem:projected-modified} and Theorem~\ref{thm:nl-k}, the solution $(h, w) \in \mathcal B_2 \times \mathcal B_2$ can be chosen to depend continuously on $(g, \pi)\in C^{4, \alpha}(\overline{\Omega}) \times C^{3, \alpha}(\overline{\Omega})$, $W\in  C^{2, \alpha}(\overline{\Omega})$, and $(\psi, V)\in \mathcal B_0 \times \mathcal B_1$, with analogous continuous dependence in higher regularity as well.  The proof of continuous dependence follows readily from the above analysis, cf. \cite[Proposition 3.7]{Corvino-Eichmair-Miao:2013}.   One can mimic the proof of the cited proposition to obtain continuity with respect to $\psi$, $V$, $\pi$, $W$.  Since the weight $\rho_g$ used above involves the metric, a slight bit more care is needed to show continuity with respect to $g$, for which it is convenient to use a weight function at a fixed background metric in place of $\rho_g$ in the analysis, and hence in generating the solution $(h, w)$. 

We briefly indicate the point, in the context of Theorem~\ref{thm:nl}.  For fixed $W$ and $(\psi, V)$, we consider for $i=1, 2$, $(g_i, \pi_i)\in \mathcal U$, from which we generate $(h_i,w_i)= \rho_i  (D\Phi^{W}_{(g_i, \pi_i)})^*(f_i, X_i)$ solving $\Phi^W_{(g_i, \pi_i)}(g_i+h_i, \pi_i+w_i)-\Phi^W_{(g_i, \pi_i)}(g_i,\pi_i)= (2\psi, V)$. Using the Taylor expansion (\ref{eq:Phi-mod-Tay-exp}) in Appendix A, we have $D\Phi^W_{(g_1, \pi_1)} (h_1, w_1)+ Q^W_{(g_1,\pi_1)}(h_1, w_1)=D\Phi^W_{(g_2, \pi_2)} (h_2, w_2)+ Q^W_{(g_2,\pi_2)}(h_2, w_2)$, where we let $Q^W_{(g,\pi)}= Q^W_{(g,\pi),(g,\pi)}$.  This can be re-written as an elliptic system $$\rho_1^{-1} D\Phi^W_{(g_1, \pi_1)} [ \rho_1 (D\Phi^W_{(g_1, \pi_1)})^*(f_1-f_2, X_1-X_2)]= (\phi, Z),$$ where 
\begin{align*}
(\phi, Z)& = - \Big[ \rho_1^{-1} (D\Phi^W_{(g_1, \pi_1)}-D\Phi^W_{(g_2, \pi_2)})(h_2, w_2) + \rho_1^{-1} \Big( Q^W_{(g_1,\pi_1)}(h_1, w_1)-Q^W_{(g_2,\pi_2)}(h_2, w_2)\Big)\\
& \qquad + \rho_1^{-1} D\Phi^W_{(g_1, \pi_1)} \Big( \big( \rho_1 (D\Phi^W_{(g_1, \pi_1)} )^*-\rho_2 (D\Phi^W_{(g_2, \pi_2)})^*\big)(f_2, X_2)\Big)\Big]. 
\end{align*}
We can estimate the difference $(h_1-h_2, w_1-w_2)$ using weighted estimates for the elliptic system, but to estimate the very last term, we want to take $\rho_1=\rho_2$ when using an exponential weight (in case we use a power weight instead, it would not matter since we would have uniform estimates on $\rho_1^{-1} \rho_2$ and derivatives), cf. Remark \ref{rmk:weight-cts}. 
\end{remark}

This completes the proof of Theorem~\ref{theorem:main2}, and this also gives us what is required in the proof of Theorem~\ref{theorem:interpolating-main} and Theorem~\ref{theorem:gluing}.  In particular for Theorem~\ref{theorem:gluing}, we can choose $\psi_0>0$ in the proof of Theorem~\ref{theorem:gluing2} to decay fast enough at the near the boundary (as a suitable power of the weight function, say).

\appendix
\section{Estimates on the Taylor expansions}
Consider the Taylor expansion of the constraint map $\Phi$ at an initial data set $(g, \pi)$
 \begin{align} \label{eq:Phi-Tay-exp}
	\Phi(g+h, \pi + w) &= \Phi(g, \pi) + D\Phi|_{(g,\pi)}(h,w) + Q_{(g,\pi)} (h,w).
\end{align}
In local coordinates, the first component of $D\Phi|_{(g, \pi)}(h,w)$ is a homogeneous linear polynomial in $\partial^2_{ij} h_{kl}, \partial_i h_{kl}, h_{kl}$ and $\partial_i w^{kl}, w^{kl}$ whose coefficients are smooth functions of $\partial^2_{ij} g_{kl}, \partial_i g_{kl}, g_{kl}$, $\partial_i \pi^{kl}, \pi^{kl}$, and the second component of $D\Phi|_{(g, \pi)}(h,w)$ is of the same type but contains no second derivatives of $g_{kl}, h_{kl}$. The remainder term $Q_{(g, \pi)}(h,w)$ is a homogeneous quadratic polynomial in $\partial^2_{ij} h_{kl}, \partial_i h_{kl}, h_{kl}, \partial_i w^{kl}, w^{kl}$ whose coefficients are smooth functions of $\partial^2_{ij} g_{kl}, \partial_i g_{kl}, g_{kl}$, $\partial_i \pi^{kl}, \pi^{kl}$ and $\partial^2_{ij} h_{kl}, \partial_i h_{kl}, h_{kl}, \partial_i w^{kl}, w^{kl}$, and note that the second component of  $Q_{(g, \pi)}(h,w)$  contains no second derivatives in $g_{kl}$ and $h_{kl}$. 

It is clear that if $(g,\pi)\in C^{k+1,\alpha}(\overline{\Omega})\times C^{k+1,\alpha}(\overline{\Omega})$ for $k\ge 1$, we have
\begin{align}\label{equation:simple-estimate}
\begin{split}
	\| D\Phi|_{(g, \pi)} (h, w)\|_{C^{k-1,\alpha}(\overline{\Omega})\times C^{k,\alpha}(\overline{\Omega}) } &\le  C \| (h, w) \|_{C^{k+1,\alpha}(\overline{\Omega})\times C^{k+1,\alpha}(\overline{\Omega})}\\
	\|Q_{(g,\pi)} (h,w) \|_{C^{k-1,\alpha}(\overline{\Omega})\times C^{k,\alpha}(\overline{\Omega}) } & \le  C \| (h, w) \|_{C^{k+1,\alpha}(\overline{\Omega})\times C^{k+1,\alpha}(\overline{\Omega})}^2
\end{split}
\end{align}
where $C$ depends locally uniformly on $(g, \pi), (h,w) \in C^{k+1,\alpha}(\overline{\Omega})\times C^{k+1,\alpha}(\overline{\Omega})$. By direct analysis (with a bit more care), we have the following estimates involving the weights. 

\begin{lemma} \label{lemma:Taylor}
Suppose that $f$ is a $C^{2,\alpha}(\overline{\Omega})$ function such that  $\nabla f$ is supported on a compact subset of $\Omega$. Then 
\begin{align*}
	\|D\Phi|_{(g,\pi)} (fh,fw) - f D\Phi|_{(g,\pi)}(h,w)\|_{\mathcal B_0 \times \mathcal B_1} &\le C  \big(\|\nabla^2 f\|_{\mathcal B_0}+ \| \nabla f\|_{\mathcal B_1}\big) \| (h, w) \|_{C^{2,\alpha}(\overline{\Omega})\times C^{2,\alpha}(\overline{\Omega})}\\
	 \| D\Phi|_{(g+h, \pi+w)} (h,w) - D\Phi|_{(g, \pi)}(h, w)\|_{C^{0,\alpha}(\overline{\Omega}) \times C^{1,\alpha}(\overline{\Omega})}&\le C \| (h, w)\|^2_{C^{2,\alpha}(\overline{\Omega})\times C^{2,\alpha}(\overline{\Omega})} \\
	\|Q_{(g, \pi)}(fh,fw) - f^2\widetilde{Q}_{(g,\pi)}(h,w) \|_{\mathcal B_0 \times \mathcal B_1} &\le C\big(\|\nabla^2 f\|_{\mathcal B_0}+  \|\nabla f\|_{\mathcal B_1}\big) \| (h, w)\|^2_{C^{2,\alpha}(\overline{\Omega})\times C^{2,\alpha}(\overline{\Omega})}
\end{align*}
for some 
\[
\| \widetilde{Q}_{(g,\pi)}(h,w) \|_{C^{0,\alpha}(\overline{\Omega})\times C^{1,\alpha}(\overline{\Omega})} \le C  \| (h, w)\|^2_{C^{2,\alpha}(\overline{\Omega})\times C^{2,\alpha}(\overline{\Omega})},
\] 
where $C$ depends locally uniformly on $(g, \pi), (h,w) \in C^{2,\alpha}(\overline{\Omega})\times C^{2,\alpha}(\overline{\Omega})$ and $ f \in C^{2,\alpha}(\overline{\Omega})$. 
\end{lemma}

 We apply those estimates to interpolation between initial data sets $(g_1,\pi_1)$ and $(g_2, \pi_2)$.

\begin{lemma}\label{lemma:interpolation-Taylor}
Let  $ 0\le \chi \le 1$ be a $C^{k+2,\alpha}(\overline{\Omega})$ bump function such that $\chi (1-\chi)$ is supported on a compact subset of  $\Omega$. Denote by $(g,\pi) = \chi(g_1,\pi_1) + (1-\chi) (g_2, \pi_2)$. Then 
\begin{enumerate}
\item \label{item:estimate-1} The following H\"{o}lder estimate holds:
\begin{align*}
	&\| \Phi(g,\pi ) - \chi \Phi(g_1, \pi_1) - (1-\chi) \Phi(g_2, \pi_2) \|_{C^{k-1, \alpha}(\overline{\Omega})\times C^{ k, \alpha}(\overline{\Omega}) }\\
	&\qquad \le C \| (g_1 - g_2, \pi_1 - \pi_2) \|_{C^{k+1,\alpha}(\overline{\Omega})\times C^{k+1,\alpha}(\overline{\Omega})},
\end{align*}
where $C$ depends locally uniformly on $(g_1, \pi_1), (g_2,\pi_2) \in C^{k+1,\alpha}(\overline{\Omega})\times C^{k+1,\alpha}(\overline{\Omega})$ and $ \chi \in C^{k+1,\alpha}(\overline{\Omega})$. 
\item \label{item:estimate-2} The following weighted estimate holds:
\begin{align*}
	&\| \Phi(g,\pi ) - \chi \Phi(g_1, \pi_1) - (1-\chi) \Phi(g_2, \pi_2) \|_{\mathcal{B}_0 \times \mathcal B_1}\\
	&\qquad  \le C \big(\| \chi (1-\chi)\|_{\mathcal{B}_1} +\|\nabla^2 \chi\|_{\mathcal B_0}+ \| \nabla \chi \|_{\mathcal B_1} \big) \| (g_1-g_2, \pi_1 - \pi_2 ) \|_{C^{2,\alpha}(\overline{\Omega})\times C^{2,\alpha}(\overline{\Omega})},
\end{align*}
where $C$ depends locally uniformly on $(g_1, \pi_1), (g_2,\pi_2) \in C^{2,\alpha}(\overline{\Omega})\times C^{2,\alpha}(\overline{\Omega})$ and $ \chi \in C^{2,\alpha}(\overline{\Omega})$. 
\end{enumerate}
\end{lemma}
\begin{proof}
Writing $\Phi(g, \pi) = \chi \Phi (g, \pi) + (1-\chi) \Phi(g,\pi)$, we apply Taylor expansion to the first term at $(g_1, \pi_1)$ and the second term at $(g_2, \pi_2)$ and derive
\begin{align*}
	& \Phi(g, \pi) - \left( \chi \Phi(g_1,\pi_1) + (1-\chi) \Phi(g_2,\pi_2) \right)\\
	&=\chi D\Phi|_{(g_1, \pi_1)} ((1-\chi)(g_2 - g_1, \pi_2-\pi_1)) + \chi Q_{(g_1, \pi_1)}((1-\chi)(g_2 - g_1, \pi_2-\pi_1)) \\
	&\quad+(1- \chi) D\Phi|_{(g_2, \pi_2)} (\chi(g_1 - g_2, \pi_1-\pi_2)) + (1-\chi) Q_{(g_2, \pi_2)}(\chi(g_1 - g_2, \pi_1-\pi_2)).
\end{align*}
The estimate \eqref{item:estimate-1} follows by \eqref{equation:simple-estimate}.

The weighted estimate \eqref{item:estimate-2} follows by Lemma~\ref{lemma:Taylor}. For example, in analyzing the preceding equation, the following term appears
\[
	\chi(1-\chi) \left( D\Phi|_{(g_2, \pi_2)} ( g_1 - g_2, \pi_1-\pi_2) - D\Phi|_{(g_1,\pi_1)}(g_1 - g_2, \pi_1 - \pi_2)\right)
\]
and satisfies the desired estimate. 
\end{proof}

We also need to estimate the Taylor expansion of the modified constraint map for the iteration scheme. For a fixed vector field $W\in C^{k,\alpha}(\overline{\Omega})$, the Taylor expansion of the modified map $\Phi^W_{(g,\pi)}$ at $(\gamma, \tau)$ is  
\begin{align}\label{eq:Phi-mod-Tay-exp}
	\Phi^W_{(g,\pi)}(\gamma+h, \tau+w)  = \Phi^W_{(g,\pi)}(\gamma, \tau) + D\Phi^W_{(g,\pi)}|_{(\gamma,\tau)} (h, w) + Q^W_{(g, \pi), (\gamma, \tau)} (h, w).
\end{align}
In local coordinates, the linearized equation and the quadratic error term have a similar type of expressions as those for the usual constraint map. By direct analysis, we have the following estimates.

\begin{lemma}\label{lemma:Taylor-modified}
There is a constant $D$ depending locally uniformly on $(g, \pi), (\gamma, \tau), (\gamma', \tau'), (h,w) \in C^{2,\alpha}(\overline{\Omega})\times C^{2,\alpha}(\overline{\Omega})$ and $W\in C^{1,\alpha}(\overline{\Omega})$ such that 
\begin{align*}
	\| D \Phi^W_{(g,\pi)} |_{(\gamma, \tau)} (h, w) - D\Phi^W_{(g,\pi)}|_{(\gamma', \tau')} (h, w) \|_{\mathcal{B}_0\times \mathcal B_1}  &\le D \| (h, w) \|_{\mathcal{B}_2 \times \mathcal B_2} \| (\gamma-\gamma', \tau-\tau')\|_{\mathcal{B}_2 \times \mathcal B_2}\\		
	\|Q^W_{(g,\pi), (\gamma,\tau)} (h, w) \|_{\mathcal B_0\times \mathcal B_1} &\leq D\|( h, w)\|_{\mathcal{B}_2 \times \mathcal B_2}^2.
\end{align*}
\end{lemma}

\section{Asymptotically flat initial data sets}\label{section:asymptotically-flat}

 Let $B$ be a closed ball in $\mathbb{R}^3$. For every $k\in\{0, 1, \ldots\}$, $\alpha\in(0,1)$, and $q\in\mathbb{R}$ we define the norm $C^{k,\alpha}_{-q}(\mathbb{R}^3\setminus B)$ for $f \in C^{k, \alpha}_{\mathrm{loc}} (\mathbb{R}^3\setminus B)$ as 
\[
		\| f\|_{C^{k,\alpha}_{-q}(\mathbb{R}^3\setminus B)} = \sum_{|I| \le k} \sup_{x\in \mathbb{R}^3\setminus B} \left| |x|^{ | I| + q }(\partial^{I} f)(x) \right| +  \sum_{|I| = k} \left[ |x|^{k + q+\alpha }(\partial^{I} f)(x) \right]_{\alpha, \mathbb{R}^3\setminus B}. 
\]
Let  $M$ be a smooth manifold such that there is a compact set $K \subset M$ and a diffeomorphism $M\setminus K \cong \mathbb{R}^3\setminus B$. The $C^{k,\alpha}_{-q}$ norm on $M$ is defined by taking the maximum of the $C^{k,\alpha}_{-q}(\mathbb{R}^3\setminus B)$ norm and the $C^{k,\alpha}$ norm on the compact set $K$. The weighted H\"{o}lder space $C^{k,\alpha}_{-q}(M)$ is the collection of those $f\in C^{k,\alpha}_{\mathrm{loc}}(M)$ with finite $C^{k,\alpha}_{-q}(M)$ norm.

Let $q>\frac{1}{2}$, $q_0 >0$. We say that an initial data set $(M,g,\pi)$ is \emph{asymptotically flat at the rate $(q, q_0)$} if there is a compact set $K \subset M$ and a diffeomorphism $M\setminus K \cong \mathbb{R}^3\setminus B$ for a closed ball $B\subset \mathbb{R}^3$ such that
\[
	(g-\delta_{\mathbb{E}}, \pi)\in C^{2,\alpha}_{-q}(M)\times C^{1,\alpha}_{-q-1}(M),
\]
where $\delta_{\mathbb{E}}$ is a smooth symmetric $(0,2)$ tensor that coincides with the Euclidean metric $g_\mathbb{E}$ on $M\setminus K\cong \mathbb{R}^3\setminus B$, and such that 
\[
	\mu, J \in C^{0,\alpha}_{-3-q_0}(M).
\]
Our definition focuses on the analysis of one asymptotically flat end, but can obviously accommodate $M$ with multiple asymptotically flat ends. 

For an asymptotically flat initial data set $(M, g, \pi)$, one can define the following boundary integral, for a function $N$ and a vector field $X$, 
\[
	B^r_{(g,\pi)}( N, X) = \int_{|x|=r} \sum\limits_{i,j=1}^3 \left[ N \left(g_{ij,i}-  g_{ii,j} \right) -\left (N_{,i} g_{ij} - N_{,j} g_{ii}\right)+ X^i \pi_{ij} \right] \nu_0^j \, d\sigma_0.
\]
Here, the integrals are computed in the coordinate chart $M \setminus K \mathcal \cong_x \mathbb{R}^3\setminus B$, $\nu_0^j= x^j / |x|$, and $d\sigma_0$ is $2$-dimensional Euclidean Hausdorff measure. The ADM energy $E$, linear momentum $P$,  center of mass\footnote{We remark that (for $E\neq 0$) $\mathcal C$ is sometimes written $\mathcal C=Ec$, where $c$ is the center of mass in other references, e.g. \cite{Chrusciel-Corvino-Isenberg:2011}} $\mathcal C$, and angular momentum $\mathcal{J}$ are defined by
\begin{align*}
E &= \tfrac{1}{16\pi}  \lim_{r\to\infty} B^r_{(g,\pi)}(1,0)\\
P_i&=\tfrac{1}{8\pi} \lim_{r\to\infty}  B^r_{(g,\pi)}(0, \frac{\partial}{\partial x^i})\\
\mathcal C_i & = \tfrac{1}{16\pi}  \lim_{r\rightarrow \infty} B^r_{(g,\pi)}(x^i, 0) \\
\mathcal{J}_{k} &= \tfrac{1}{8\pi} \lim_{r\rightarrow \infty} B^r_{(g,\pi)}(0, x\times \frac{\partial}{\partial x^k}),
\end{align*}
where $i, k=1, 2, 3$, and $ x\times \frac{\partial}{\partial x^k}$ is the cross product.  The well-definedness of those quantities can be found in, e.g., \cite{Bartnik:1986, Huang:2009, Huang:2010}.

For an asymptotically flat initial data set $(g,\pi)$, by Taylor expansion~\eqref{eq:Phi-Tay-exp} at the flat data $(g_{\mathbb{E}}, 0)$, 
\[
	 \Phi(g, \pi) = D\Phi|_{(g_\mathbb{E}, 0)} (g - g_\mathbb{E}, \pi) + Q_{(g_\mathbb{E}, 0)}(g-g_\mathbb{E}, \pi).
\]
For $(N,X)$ so that $ D\Phi|_{(g_{\mathbb E},0)}^*(N, X) $ vanishes, we have (omitting the Euclidean metric subscript on the dot product)
\begin{align}	\label{equation:boundary}
 \int_{\{R_1 \le |x| \le R_2\}} D\Phi|_{(g_{\mathbb{E}}, 0)}(g-g_{\mathbb{E}},\pi) \cdot (N, X) \,dx= B^{R_2}_{(g,\pi)}( N, X)  - B^{R_1}_{(g,\pi)} (N, X).
\end{align}

\begin{lemma}\label{lemma:L^2-projection}
Let $(g, \pi)$ be an asymptotically flat initial data set on $\mathbb{R}^3 \setminus B$ at the rate $q=q_0=1$. Let $(g^\theta, \pi^\theta)$ be an admissible family for $(g, \pi)$. Consider the initial data set $(\bar{g}^R, \bar{\pi}^R) $ on $A_1$ obtained in Proposition~\ref{proposition:nl2}  (with $(g_1,\pi_1) = (g,\pi), (g_2, \pi_2) = (g^\theta, \pi^\theta)$):
\[	
	(\bar{g}^R, \bar{\pi}^R) = \chi (g^R, \pi^R) + (1-\chi)((g^\theta)^R, (\pi^\theta)^R) + (h^R, w^R).
\]
Then there is a constant $C$ such that for $\theta \in \Theta_1\times \Theta_2^R$ and for $R$ sufficiently large, 
\begin{align*}
	 \left| R\int_{A_1} \Phi(\bar{g}^R, \bar{\pi}^R)\cdot (1,0) \, dx - 16\pi (E^\theta - E)\right| &\le  CR^{-1}\\
	\left| R\int_{A_1} \Phi(\bar{g}^R,\bar{\pi}^R) \cdot (0, \frac{\partial}{\partial x^i}) \, dx - 8\pi (P^\theta - P) \right| &\le CR^{-1}\\
	\left| R\int_{A_1} \Phi(\bar{g}^R,\bar{\pi}^R) \cdot (x^k, 0)\, dx - 16\pi R^{-1}(\mathcal C^\theta_k - \mathcal C^R_k)\right| &\le C(R^{-1}+|\theta|^{2} R^{-2}) \\
	\left| R\int_{A_1} \Phi(\bar{g}^R,\bar{\pi}^R) \cdot (0, x\times \frac{\partial}{\partial x^\ell})\, dx-8\pi R^{-1}(\mathcal J^\theta_\ell - \mathcal J^R_\ell)\right| &\le C(R^{-1}+|\theta|^{2} R^{-2}). 
\end{align*}
\end{lemma}
\begin{proof}
By \eqref{equation:fall-off} and the estimate for $(h^R, \pi^R)$ in Proposition~\ref{proposition:nl2}, the quadratic error terms are estimated as follows:
\[
		\left| R\int_{A_1} Q_{(g_\mathbb{E}, 0)} (\bar{g}^R - g_\mathbb{E}, \bar{\pi}^R)\cdot (N, X) \, dx\right| \le C R\|  (\bar{g}^R - g_\mathbb{E}, \bar{\pi}^R) \|_{C^2(A_1)\times C^1(A_1)}^2 \le C R^{-1}.
\]

We now estimate  the integrals of the linearized operator. The energy and momentum integrals can be estimated similarly, so we only show the one for the energy. By \eqref{equation:boundary}, 
\begin{align*}
	R\int_{A_1} D\Phi(\bar{g}^R, \bar{\pi}^R)\cdot (1,0) \, dx&= R \left(B^2_{((g^\theta)^R,(\pi^\theta)^R)}(1,0) - B^1_{(g^R,\pi^R)} (1,0)\right) \\
	&= B^{2R}_{(g^\theta, \pi^\theta)}(1,0) - B^R_{(g,\pi)}(1,0),
\end{align*}
where we note that the rescaling in the last line accounts for the factor $R$. It is standard to relate the boundary integral to the ADM energy at infinity by the divergence theorem. In particular, the uniformity assumption \eqref{equation:fall-off} implies that there is a constant $C$ such that for $\theta \in \Theta_1\times \Theta_2^R$,
\begin{align*}
	\left| B^{2R}_{(g^\theta, \pi^\theta)}(1,0) - 16 \pi E^\theta \right| \le CR^{-1}.
\end{align*}

The proofs for the center of mass and angular momentum integrals are similar, so we only show prove  the  one for the center of mass. By \eqref{equation:boundary} we have
\begin{align*}
	R\int_{A_1} D\Phi(\bar{g}^R, \bar{\pi}^R)\cdot (x^k,0) \, dx &= R \left(B^2_{((g^\theta)^R,(\pi^\theta)^R)}(x^k,0) - B^1_{(g^R,\pi^R)} (x^k,0)\right) \\
	&= R^{-1}\left(B^{2R}_{(g^\theta, \pi^\theta)}(x^k,0) - B^R_{(g,\pi)}(x^k,0)\right),
\end{align*}
where the rescaling in the last line gives an extra factor $R^{-1}$ from the rescaling of the coordinate function $x^k$. To obtain the desired estimate, we see that the term $B^{2R}_{(g^\theta, \pi^\theta)}(x^k,0) $ is estimated by the uniformity \eqref{equation:fall-off-RT} and $B^R_{(g,\pi)}(x^k,0) = 16\pi \mathcal C_k^R$ by definition.
\end{proof}

\section{Interior Schauder estimates} \label{section:Schauder}
Let $x\in \Omega$ be fixed. Consider the ball $B_{\phi(x)} (x)$ centered at $x$ of radius $\phi(x)$, where $\phi(x)$ is the weight function defined in Section~\ref{subsection:weighted-Holder}. We blur the distinction between $B_{\phi(x)}(x)$ and its coordinate image, and we consider the diffeomorphism $F_x: B_1(0) \to B_{\phi(x)}(x)$ by $z \mapsto x+ \phi(x) z=y$, where $B_1(0)$ is the unit ball in $\mathbb{R}^n$ centered at the origin. For any function $f$ defined on $B_{\phi(x)}(x)$, let 
\[
	\widetilde{f}(z) = F_x^*(f)(z) = f\circ F_x (z)
\] denote the pull-back of $f$ on $B_1(0)$.

With a minor abuse of notation, we denote for $a\in (0,1]$,
\[
	\| f \|_{C^{k,\alpha}_{\phi, \varphi}(B_{a\phi(x)}(x))} = \sum_{j=0}^k \varphi(x) \phi^j(x)\|  \nabla^j f \|_{C^0(B_{a\phi(x)}(x))} + \varphi(x) \phi^{k+\alpha}(x) [ \nabla^{k} f]_{0, \alpha; B_{a\phi(x)}(x)}.
\]
One can easily obtain the following lemma. 
\begin{lemma} \label{lemma:pull-back}
Let $f$ and $g$ be functions defined on $B_{\phi(x)} (x)$. The following properties hold.
\begin{enumerate}
\item $\widetilde{f + g} = \widetilde{f} + \widetilde{g}$ and $\widetilde{fg} = \widetilde{f} \widetilde{g}$.
\item $\widetilde{ \partial^{\beta}_{y} f}= (\phi(x))^{-|\beta|} \partial_z^{\beta} \widetilde{f}$, where  $\beta = (\beta_1, \dots, \beta_k)$ is a multi-index, $\partial^{\beta}_{y} = \partial^{\beta_1}_{y^{i_1} } \cdots \partial^{\beta_k}_{y^{i_k}}$, $i_1, \dots, i_k  \in\{ 1, 2, \ldots, n\}$, and $\partial^{\beta}_z$ is defined analogously. 
\item For any $a\in (0, 1]$, 
\begin{align*}
	\| \varphi(x) \widetilde{f} \|_{C^{k, \alpha} (B_a(0))} &= \| f \|_{C^{k,\alpha}_{\phi, \varphi} (B_{a\phi(x)} (x))}\\
	\|\varphi (x)\widetilde{f} \|_{L^2 (B_a(0))} &= \| f \|_{L^2_{\phi^{-n} \varphi^2}(B_{a\phi(x)} (x))}.
\end{align*}
\end{enumerate}
\end{lemma}

Let $U= (U^1, U^2, U^3, \ldots,  U^{n+1})$ where each $U^i$ is a function defined on $B_{\phi(x)}(x)$. Consider the partial differential system $LU$ whose $j$-th component is 
\[
	(L U)_j = \sum_{k=1}^{n+1} \sum_{|\beta|=0}^{s_j+t_k} b_{jk}^{\beta} \partial_y^{\beta} U^k, 
\]
where 
\[
	s_1 = 0,  \quad t_1 = 4, \qquad s_j=-1,\quad t_k= 3\qquad (j, k = 2,  \dots, n+1).
\]
\begin{theorem}
Suppose that the operator $L$ is strictly elliptic in the sense of \cite{Douglis-Nirenberg:1955}. For any $r, s\in \mathbb{R}$, let $\varphi_j = \phi^{r+4-t_j} \rho^s$. Then for each $k\in \{1, 2, \ldots ,n+1\}$, 
\begin{align} \label{equation:U-local}
\begin{split}
	&\| U^k\|_{C^{t_k, \alpha}_{\phi, \varphi_k}(B_{\phi(x)/2}(x))} \\
&\qquad  \le C\left( \sum_{j=1}^{n+1} \|(L U)_j\|_{C^{-s_j, \alpha}_{\phi, \phi^{t_j+ s_j}\varphi_j }(B_{\phi(x)}(x))} + \sum_{j=1}^{n+1} \| U^j \|_{L_{\phi^{-n} \varphi_j^2}^2(B_{\phi(x)}(x))}\right),
\end{split}
\end{align}
where $C$ depends only on $n$, $\alpha$, $\sup\limits_{\stackrel{j, k = 1, \dots, n+1}{ |\beta|\leq s_j+t_k}}  \| b_{jk}^{\beta} \|_{C^{-s_j, \alpha}_{\phi, \phi^{s_j+t_k-|\beta|}}(B_{\phi(x)}(x))}$, and the lower bound of ellipticity of the operator $L$. 
\end{theorem}
\begin{proof}
By Lemma~\ref{lemma:pull-back}, for $j = 1,\dots, n+1$, 
\begin{align*}
	\widetilde{(LU)_j} &= \sum_{k=1}^{n+1} \sum_{|\beta|=0}^{s_j+t_k} \widetilde{b_{jk}^{\beta}} (\phi(x))^{-|\beta|} \partial_z^{\beta} \widetilde{U^k}\\
	&=\sum_{k=1}^{n+1} \sum_{|\beta|=0}^{s_j+t_k} (\phi(x))^{-4+t_k-|\beta|} \widetilde{b_{jk}^{\beta}}  \partial_z^{\beta} ((\phi(x))^{4-t_k} \widetilde{U^k}).
\end{align*}
Multiplying $(\phi(x))^{4+s_j}$ to the above identity, we have 
\[
	(\phi(x))^{4+s_j}	\widetilde{(LU)_j} =\sum_{k=1}^{n+1} \sum_{|\beta|=0}^{s_j+t_k}  (\phi(x))^{s_j+t_k-|\beta|} \widetilde{b_{jk}^{\beta}}\partial_z^{\beta} ((\phi(x))^{4-t_k} \widetilde{U^k}).
\]
We remark the power of $\phi$ is specifically chosen such that, for our application to Theorem~\ref{theorem:var-sol-est}, each of the coefficients $ (\phi(x))^{s_j+t_k-|\beta|} \widetilde{b_{jk}^{\beta}}$ is bounded in the $C^{ - s_j, \alpha}(B_1(0))$ norm.  Now we apply the Schauder interior estimate and obtain
\begin{align*}
	\sum_{j=1}^{n+1}\| (\phi(x))^{4-t_j} & \widetilde{U^j}\|_{C^{t_j, \alpha}(B_{\frac{1}{2}}(0))} \\
	&\le C\left( \sum_{j=1}^{n+1} \| (\phi(x))^{4+s_j}\widetilde{L_j U}\|_{C^{-s_j, \alpha}(B_1(0))} + \sum_{j=1}^{n+1} \| (\phi(x))^{4-t_j} \widetilde{U^j} \|_{L^2(B_1(0))}\right), 
\end{align*} 
where $C$ depends only on $n$, $\alpha$, $\sup\limits_{\stackrel{j, k = 1, \dots, n+1}{ |\beta|\leq s_j+t_k}}  \| (\phi(x))^{s_j+t_k-|\beta|} \widetilde{b_{jk}^{\beta}}\|_{C^{ - s_j, \alpha}(B_1(0))}$, and the lower bound of ellipticity of the operator $L$. By Lemma~\ref{lemma:pull-back} (3),
\[
	\| (\phi(x))^{s_j+t_k-|\beta|} \widetilde{b_{jk}^{\beta}}\|_{C^{ - s_j, \alpha}(B_1(0))} = \| b_{jk}^{\beta} \|_{C^{-s_j, \alpha}_{\phi, \phi^{s_j+t_k-|\beta|}}(B_{\phi(x)}(x))}. 
\]
For $r, s \in \mathbb{R}$, let $\varphi_j = \phi^{r+4-t_j} \rho^s$. Multiplying $\phi^r(x) \rho^s(x)$ to the above inequality, we have 
\begin{align*}
	&\sum_{j=1}^{n+1}\| \varphi_j(x) \widetilde{U^j}\|_{C^{t_j, \alpha}(B_{\frac{1}{2}}(0))} \\
	&\le C\left( \sum_{j=1}^{n+1} \| (\phi(x))^{t_j+s_j} \varphi_j(x) \widetilde{(LU)_j}\|_{C^{-s_j, \alpha}(B_1(0))} + \sum_{j=1}^{n+1} \| \varphi_j(x) \widetilde{U^j} \|_{L^2(B_1(0))}\right).
\end{align*}
Then Lemma~\ref{lemma:pull-back} (3) implies the desired estimate.
\end{proof}

\begin{proof}[Proof of Theorem~\ref{theorem:U}]
By the definition of the weighted H\"{o}lder norm and by taking the supremum of \eqref{equation:U-local} among $x\in \Omega$, it suffices to prove that for any $u$ and $\varphi = \phi^r \rho^s$
\[
	  \sup_{x\in \Omega} \| f\|_{C^{k, \alpha}_{\phi, \varphi}(B_{\phi(x)}(x))} \le C \sup_{x\in \Omega} \| f\|_{C^{k, \alpha}_{\phi, \varphi}(B_{\phi(x)/2}(x))}
\]
where $C$ denotes a positive constant that depends only on $n$, $ k, \alpha, r, s,$ and the constant in \eqref{equation:weight}. The proof is a straightforward computation using \eqref{equation:weight}.
\end{proof}

\section{Iteration scheme} \label{section:recursion}

We now prove Lemma~\ref{lemma:induction}. The proof follows essentially \cite[Lemma 3.5]{Corvino-Eichmair-Miao:2013}. 

\begin{proof} By the induction hypothesis~\eqref{equation:induction2},  $(\gamma_m , \tau_m)$ satisfies that $\Phi^W_{(g,\pi)}(g,\pi) + (\psi, V) - \Phi^W_{(g,\pi)} (\gamma_m, \tau_m) \in L^2_{\rho^{-1}} \times L^2_{\rho^{-1}}$. Use the variational  result Theorem~\ref{theorem:var-sol}, we find $(f_m, X_m)$ such that $(h_m, w_m) = \rho_g(D \Phi^W_{(g,\pi)}|_{(g,\pi)})^* ( f_m, X_m)$ satisfies
\[	
	D\Phi^W_{(g,\pi)}|_{(g,\pi)} (h_m, w_m)= \Phi^W_{(g,\pi)}(g,\pi) + (\psi, V) - \Phi^W_{(g,\pi)} (\gamma_m, \tau_m).
\]
By Theorem~\ref{theorem:var-sol-est} and the induction hypothesis~\eqref{equation:induction2}, 
\begin{align*}
	\| (f_m, X_m) \|_{\mathcal B_4\times \mathcal B_3} &\le C \| \Phi^W_{(g,\pi)}(g,\pi) + (\psi, V) - \Phi^W_{(g,\pi)} (\gamma_m, \tau_m) \|_{\mathcal B_0\times \mathcal B_1} \\
	&\le C\| (\psi, V) \|_{\mathcal B_0\times \mathcal B_1}^{1+ m\delta}\\
	\| (h_m, w_m) \|_{\mathcal B_2\times \mathcal B_2} &\le C \| \Phi^W_{(g,\pi)}(g,\pi) + (\psi, V) - \Phi^W_{(g,\pi)} (\gamma_m, \tau_m) \|_{\mathcal B_0 \times \mathcal B_1} \\
	&\le C\| (\psi, V) \|_{\mathcal B_0 \times \mathcal B_1}^{1+ m \delta}.
\end{align*}
This gives the desired estimates~(\ref{equation:induction1}) for $p=m$. 

To prove \eqref{equation:induction2} for $j=m+1$, we note that by Taylor expansion,
\begin{align*}
	&\Phi^W_{(g,\pi)}(\gamma_{m+1}, \tau_{m+1}) \\
	&= \Phi^W_{(g,\pi)}(\gamma_m, \tau_m) + D\Phi^W_{(g,\pi)}|_{(\gamma_m, \tau_m)} (h_m, w_m) + Q^W_{(\gamma_m, \tau_m)}(h_m, w_m)\\
	&= \Phi^W_{(g,\pi)}(g,\pi) + (\psi, V) - D\Phi^W_{(g,\pi)}|_{(g,\pi)} (h_m, w_m)\\
	&\quad +D\Phi^W_{(g,\pi)}|_{(\gamma_m, \tau_m)} (h_m, w_m)+ Q^W_{(\gamma_m, \tau_m)}(h_m, w_m)\\
	&= \Phi^W_{(g,\pi)}(g,\pi) + (\psi, V) \\
	&\quad + \sum_{p=0}^{m-1} \left[ D\Phi^W_{(g,\pi)} |_{(\gamma_{p+1}, \tau_{p+1})}(h_m, w_m) - D\Phi^W_{(g,\pi)} |_{(\gamma_{p}, \tau_{p})}(h_m, w_m) \right]\\
	&\quad + Q^W_{(\gamma_m, \tau_m)}(h_m, w_m).
\end{align*}
By Lemma~\ref{lemma:Taylor-modified}, 
\begin{align*}
	 \| \Phi^W_{(g,\pi)} (g, \pi) + & (\psi, V)  - \Phi^W_{(g,\pi)}(\gamma_{m+1}, \tau_{m+1}) \|_{\mathcal B_0\times \mathcal B_1}\\
	&\le D\left( \|(h_m, w_m)\|_{\mathcal B_2\times \mathcal B_2}^2 + \| (h_m, w_m)\|_{\mathcal B_2\times \mathcal B_2} \sum_{p=0}^{m-1} \| (\gamma_{p+1}, \tau_{p+1}) - (\gamma_p, \tau_p)\|_{\mathcal B_2\times \mathcal B_2}\right) \\
	&\le DC^2 \left(\|(\psi, V)\|_{\mathcal B_0\times \mathcal B_1}^{2+2m\delta} + \| (\psi, V) \|_{\mathcal B_0\times \mathcal B_1}^{2+m\delta} \sum_{p=0}^{m-1} \| (\psi, V) \|_{\mathcal B_0\times \mathcal B_1}^{p\delta}\right)\\
	&\le 2DC^2 \epsilon^{1-\delta} (1-\epsilon^{\delta})^{-1} \| (\psi, V) \|_{\mathcal B_0\times \mathcal B_1}^{1+(m+1)\delta}.
\end{align*}
Choose $\epsilon>0$ small enough so that $ 2DC^2 \epsilon^{1-\delta} (1-\epsilon^{\delta})^{-1} \le 1$.  \end{proof}

\bibliographystyle{amsplain}
\bibliography{2017-JCrev}

\end{document}